\documentclass[10pt, reqno]{amsart}

\usepackage{latexsym}
\usepackage{amsfonts, amsmath, amsthm, amscd}
\usepackage{amssymb}
\usepackage[all, arc]{xy}
\include{xypicmacros}

\usepackage{enumerate}

\usepackage{appendix}
\usepackage{color}
\usepackage[usenames,dvipsnames]{xcolor}

 \usepackage{wrapfig}
\usepackage{graphicx, psfrag}
\usepackage{pstool}
 \usepackage{caption}
\usepackage{subcaption}

\usepackage{tikz}
\usepackage{tikz-cd}
\usetikzlibrary{calc,tqft,decorations.markings}

\usepackage{pinlabel}   

\usepackage[linktocpage]{hyperref}
\hypersetup{colorlinks=false}

\numberwithin{equation}{section}

\newtheorem{theorem}{Theorem}[section]
\newtheorem{prop}[theorem]{Proposition}
\newtheorem{proposition}[theorem]{Proposition}
\newtheorem{lemma}[theorem]{Lemma}
\newtheorem{cor}[theorem]{Corollary}

\newtheorem{defn}[theorem]{Definition}

\newtheorem{remark}[theorem]{Remark}
\newenvironment{rem}{\begin{remark}\rm}{\end{remark}}
\newtheorem{example}[theorem]{Example}
\newenvironment{ex}{\begin{example}\rm}{\end{example}}

\newtheorem*{theorem*}{Theorem}
\newtheorem*{structuretheorem*}{Structure Theorem}
 
\def\be{\begin{equation}}
\def\ee{\end{equation}}
\def\bear{\begin{eqnarray}}
\def\eear{\end{eqnarray}}
\def\best{\begin{eqnarray*}}
\def\eest{\end{eqnarray*}}

 \def\non{\noindent}
 \def\bd{\partial}
 \def\ra{\rightarrow}
\def\lra{\longrightarrow}
\def\rg{\rangle}
\def\lg{\langle}
\def\r#1{\right#1}
\def\l#1{\left#1}
\def\ti{\times}
\def\del{\overline \partial}

\def\ma#1{\mathop {#1} \limits}

\def\al{\alpha}

\def\de{\delta}
\def\ep{\varepsilon}
\def\la{\lambda}
\def\si{\sigma}
\def\Si{\Sigma}
\def\De{\Delta}
\def\La{\Lambda}

\def\Om{\Omega}
\def\ra{\rightarrow}

\def\Z{{ \mathbb Z}}
\def\R{{ \mathbb  R}}
\def\P{{\mathbb  P}}
\def\Q{{ \mathbb Q}}
\def\cx{{ \mathbb C}}

\def\cal{\mathcal}

\def\M{{\mathcal M}}

\def\O{\mathcal{O}}

\def\wt#1{\widetilde{#1}}

\def\ov#1{\overline{#1}}

\def\ind{\mathrm{ind\,}}
\def\Ind{\mathrm{Ind\,}}

\def\End{\mathrm{End}}
\def\Hom{\mathrm{Hom}}

\def\sign{\mathrm{sign}\,}

\def\vir{{\rm vir}}
\def\cob{2\mathrm{\bf Cob}}
\def\kcob{2\mathrm{\bf KCob}}
\def\SymRiem{2\mathrm{\bf SymCob}}
\def\fo{\mathfrak{o}}
 \def\dbar{\bar\partial}
 
 \def\dimh{{\rm dimh}}

\let\oldtocsection=\tocsection
\let\oldtocsubsection=\tocsubsection
\renewcommand{\tocsection}[2]{\hspace{0em}\oldtocsection{#1}{#2}}
\renewcommand{\tocsubsection}[2]{\hspace{1em}\oldtocsubsection{#1}{#2}}

\title{A Klein TQFT: the local Real Gromov-Witten theory of curves} 

\author{Penka Georgieva}
\address{Institut de Math\`ematiques de Jussieu - Paris Rive Gauche, Sorbonne Universit\'e}
\email{penka.georgieva@imj-prg.fr}
\author{Eleny-Nicoleta Ionel}
\address{Department of  Mathematics,  Stanford University}
\email{ionel@math.stanford.edu}

\begin{document}

\begin{abstract} In this paper we study the Real Gromov-Witten theory of local 3-folds over Real curves. We show that this gives rise to a 2-dimensional Klein TQFT defined on an extension of the category of unorientable surfaces. We use this structure to completely solve the theory by providing a closed formula for the local RGW invariants in terms of representation theoretic data, extending earlier results of Bryan and Pandharipande. As a consequence we obtain the local version of the real Gopakumar-Vafa formula that expresses the connected real Gromov-Witten invariants in terms of integer invariants. In the case of the resolved conifold the partition function of the RGW invariants agrees with that of the  SO/Sp Chern-Simons theory.
\end{abstract}

\maketitle 

\tableofcontents

\setcounter{equation}{0}

\vspace{-.2in}
\section{Introduction} 
 \vskip.1in 
 
 A central problem in Gromov-Witten theory is understanding the structure of the Gromov-Witten invariants. Of special interest is the case  when the target manifold is a Calabi-Yau threefold. Often, studying local versions of these theories (i.e. for non-compact targets) reveals much of the structure in the general case.  The Gromov-Witten invariants come in several flavors: (a) closed (counting closed curves), (b) open (counting curves with boundary on a Lagrangian or SFT-type curves), and (c) real (counting closed curves invariant under an anti-symplectic involution).  

 In this paper we consider Real Gromov-Witten (RGW) invariants and we prove a structure result for the local RGW invariants of Real\footnote{We use Real with capital R for spaces with anti-J-invariant involutions, following Atiyah.}   3-folds that are the total space of bundles over curves with an anti-symplectic involution (also referred to as a real structure).  We show that the local RGW invariants give rise to a semi-simple 2d Klein TQFT which allows us to completely solve the theory. The motivation for considering  3-folds of this type comes from the virtual contribution to the real GW invariants  of a Real elementary curve in a  compact Real Calabi-Yau 3-fold, sometimes referred to as multiple-covers contribution, and the real Gopakumar-Vafa conjecture.  The Gopakumar-Vafa conjecture \cite{gv} and its extension proved in \cite{ip-gv} has an analogue in the setting of Real Calabi-Yau 3-folds, cf. \cite{Wal}.  The local version of the real GV conjecture is obtained in this paper as a consequence of the structure result. The case of compact 3-folds will be discussed in a subsequent  paper.

 \medskip
 A symmetric (or Real) Riemann surface is a Riemann surface $\Si$ together with an anti-holomorphic involution $c:\Si\ra \Si$. If $L\ra \Si$ is a holomorphic line bundle, then the total space of 
 \bear\label{local.3fold}
 L\oplus c^*\ov L\ra \Si
 \eear
 is a Real manifold with an anti-holomorphic involution 
 \best\label{tw.invol} 
 c_{tw} (z;u,v)= (c(z); v, u). 
 \eest
 These are the Real 3-folds we consider in this paper, and we refer to them as \textsf{local Real curves};  note however that any rank 2 Real bundle $(V,\phi)\ra (\Si,c)$ whose fixed locus $V^\phi$ is orientable is isomorphic to a Real bundle \eqref{local.3fold} for $L$ a complex line bundle with $c_1(L)=\frac 12 c_1(V)$, cf. \cite[\S4.1]{BHH}. Moreover, an 
$U(1)$-action on the line bundle $L\ra\Si$ induces an action on the  3-fold \eqref{local.3fold} compatible with the Real structure.  In \S \ref{LocalRGW} we define \textsf{local RGW invariants} associated to the Real 3-fold  \eqref{local.3fold} as pairings between the $U(1)$-equivariant Euler class of the index bundle $\mathrm{ Ind}\; \del_L$ (regarded as an element in $K$-theory) and
  the virtual fundamental class of the real moduli space $\ov\M^{c,\bullet}_{d,\chi}(\Si)$ of degree $d$ real maps $f:C\ra \Si$ from (possibly disconnected) domains of Euler characteristic $\chi$.  The real moduli space $\ov\M^{c,\bullet}_{d,\chi}(\Si)$ is orientable, but not a priori canonically oriented; the orientation depends on a choice of orientation data $\mathfrak{o}$ 
  discussed in \S\ref{LocalRGW} and in the Appendix. The (shifted) generating functions for the local RGW invariants
  $$
  RGW^{c,\mathfrak o}_d(\Si,L) \in \Q(t)((u))
  $$
take values in the localized equivariant cohomology ring of $U(1)$ generated by $t$; here $u$ keeps track of the Euler characteristic of the domain.
 
 We also consider a relative version of the RGW invariants for a branching divisor on $(\Si,c)$ consisting of pairs of conjugate points. For the purposes of this paper, we can restrict attention to the case when none of the marked points or special points are real. The splitting formula of \cite{GI} then allows us to relate the local RGW invariants of $\Si$ with the local RGW invariants of a decomposition of $\Si$ along pairs of conjugate circles; see \S\ref{Gluing}. 
 \smallskip
 
 A priori, the local RGW invariants depend on the choice of an orientation data $\mathfrak o$ and the topological type of the real structure $c$ on $\Si$. In \S \ref{S.Can.TO}, we show that there is a canonical choice of orientation for the local RGW invariants,  and moreover these do not depend on the real structure $c$. We therefore use the notation 
 $$
 RGW_d(\Si,L) \in \Q(t)((u))
 $$ 
 afterwards, when the canonical choice is assumed. Any other choice of a twisted orientation data changes $RGW_d$ by $(\pm 1)^d$.
   \smallskip

 In \S \ref{KTQFTforRGW} we show that the local RGW invariants determine an extension ${\bf RGW}_d$ of a 2-dimensional Klein TQFT.  As we review in \S\ref{KTQFT}, a 2d Klein TQFT is a symmetric monoidal functor from the  cobordism category $\kcob$ of unoriented surfaces to the category of $R$-modules for some ring $R$. Since $\kcob$ naturally contains the oriented cobordism category $\cob$,  a Klein TQFT is an extension of a classical TQFT; it is equivalent to a Frobenius algebra with an involution $\Omega$, which is the image of the orientation reversing tube, and a special element $U$, which is the image of the crosscap (M\"obius band), cf. \S\ref{KTQFT}.  
    \smallskip
    
 The connection with real Gromov-Witten theory is obtained by considering an equivalent category $\SymRiem$ whose objects are pairs of closed oriented 1-dimensional manifolds and the cobordisms are symmetric (Real) surfaces.  It is obtained from 
 $\kcob$ by passing to the orientation double cover. Then the involution $\Omega$ is the image of the symmetric cobordism swapping the components of an object, while $U$ is the image of a symmetric sphere with a pair of (disjoint) conjugate disks removed. This perspective allows us  to define  in \S\ref{S.kcob.categ} an extension $\SymRiem^L$ which has the same objects but where the cobordisms also carry a complex vector bundle trivialized along the boundary.  As we prove in \S\ref{KTQFTforRGW}, the local RGW invariants give rise to a symmetric monoidal functor  ${\bf RGW}_d$ on $\SymRiem^L$; up to factors due to differing conventions, this extends the TQFT considered by Bryan and Pandharipande in \cite{bp1} for the anti-diagonal action. In turn, the 
 Bryan-Pandharipande construction similarly extends a classical construction studied by Dijkgraaf-Witten \cite{DW} and 
 Freed-Quinn \cite{FQ}.
    \smallskip
    
In \S\ref{Semisimple} we discuss semi-simple KTQFTs,  i.e. those for which the associated Frobenius algebra has an idempotent basis. Their restriction to the oriented cobordism category $\cob$ is determined by the eigenvalues $\{\la_\rho\}$ of the genus adding operator (which is diagonalized in the idempotent basis). To completely determine the KTQFT it then suffices to find the coefficients of $\Omega$ and $U$ in the idempotent basis. We show that $\Omega$ restricts to an involution $v_\rho \mapsto v_{\rho^*}$ on the idempotent basis and that each coefficient $U_\rho$ of $U$ is 0 when $\rho\ne \rho^*$ and otherwise is equal to a squareroot $\pm \sqrt {\la_\rho}$ of the eigenvalue $\la_\rho$. 
 
    \smallskip
In \S\ref{Solving} we prove that the KTQFT determined by the level 0 local RGW invariants is semisimple. It corresponds in fact to {\em signed} counts of degree $d$ real Hurwitz covers. The idempotent basis is indexed by irreducible representations of the symmetric group $S_d$ and $\Omega(v_\rho)=v_{\rho'}$ where $\rho'$ is the conjugate representation. In order to calculate the coefficients of $U$ in the idempotent basis, we introduce in \S\ref{SFSsection} the \textsf{signed Frobenius-Schur indicator (SFS)}. The SFS takes values 0, $\pm 1$ on irreducible real representations, unlike the standard FS indicator which is +1 on them. The SFS is 0 if and only if the representation is {\em not} self-conjugate and the sign of a self-conjugate representation is given as a function of its characters. While these considerations are valid for real representations of any finite group, in the case of the symmetric group we find a simpler expression for the latter function using the Weyl formula. In particular, for an irreducible self-conjugate representation 
 $\rho$ of $S_d$, 
 $$
 SFS(\rho)= (-1)^{(d-r(\rho))/2},
 $$ 
 where $r(\rho)$ is the rank of $\rho$, i.e. the length of the main diagonal of the Young diagram associated to $\rho$. This is precisely the sign that appears in the partition function of the SO/Sp Chern-Simons theory \cite[(6.1)]{BFM}; in the case of the resolved conifold, Theorems \ref{MainCY} and \ref{MainGV} below recover the partition function \cite[(6.3)]{BFM} and the free energy \cite[(3.2)]{BFM}, respectively.
  
Combining these results we obtain in \S\ref{Solving} a closed expression for the local RGW theory of the 3-fold \eqref{local.3fold} in terms of representation theoretic data, cf. Theorem~\ref{T.str.gen}. In the Calabi-Yau case it takes the following form: 
 \begin{theorem}[Local CY] \label{MainCY}Let $\Si$ be a connected genus $g$ symmetric surface and $L\ra \Si$  a holomorphic line bundle with Chern number $g-1$. Then the generating function of the degree $d$ local RGW invariants is equal to 
 	\best\label{R-g.partCY}
 	RGW_d(\Si, L)&=&
 	\sum_{\rho=\rho'}\Big( (-1)^{\frac{d-r(\rho)}{2}} 
	\prod_{\square\in \rho} 2 \sinh \tfrac{h(\square)u} 2 \Big)^{g-1} . 
 	\eest
 	Here the sum is over all self-conjugate partitions $\rho$ of $d$, the product is over all boxes $\square$ in the Young diagram of $\rho$, $h(\square)$ is the hooklength of $\square$, and $r(\rho)$ is the length of the main diagonal of the Young diagram of $\rho$. 
 \end{theorem}

The local RGW invariants correspond to possibly disconnected counts. As usual they can be expressed in terms of more basic invariants.  In the real GW theory these basic counts come in two flavors, $CRGW_d(\Si,L)$ and $DRGW_d(\Si,L)$, corresponding to maps from connected Real domains and respectively from doublet domains i.e. domains consisting of two copies of a connected surface with opposite complex structures and the real structure exchanging the two copies. 
In fact 
$$
1+\sum_{d=1}^\infty RGW_d(\Si,L) q^d= \exp\left(\sum_{d=1}^\infty  CRGW_d(\Si,L)q^d+\sum_{d=1}^\infty  DRGW_{2d}(\Si,L)q^{2d}\right).
$$
Furthermore, the doublet invariants are related to half of the complex GW invariants whenever the target $\Si$ is connected:
$$
DRGW_{2d}(\Si,L)(u,t) =(-1)^{d(k+1-g)} \frac{1}{2} GW^{conn}_{d}(g|k,k)(iu,it),
$$
where $g$ is the genus of $\Si$,  $k=c_1(L)[\Si]$ is the degree of $L$, and $ GW^{conn}_d(g|k,k)$ are the connected invariants defined in \cite{bp1} for the anti-diagonal action; see Corollary \ref{C.Z.part.anti-diag-balanced}.
  
     \smallskip
 As a consequence of the structure result provided by Theorem \ref{MainCY}, in \S\ref{LocalGV} we obtain the local real Gopakumar-Vafa formula; for a complete statement, see Theorem \ref{T.local.GV}.
  \begin{theorem}[Local real GV formula]\label{MainGV} Let $L\oplus c^*\ov L\lra\Si$ be a local Real Calabi-Yau 3-fold with base a genus $g$ symmetric surface $(\Si, c)$. Then the generating function for the connected real Gromov-Witten invariants has the form:
 	\best
 	\sum_{d=1}^\infty CRGW_d(\Si|L)(u)q^d = \sum_{d=1}^\infty\sum_{h=0}^\infty n^{\R}_{d,h}(g)\sum_{k\; \rm odd }\tfrac{1}{k}(2\sinh(\tfrac{ku}{2}))^{h-1}q^{kd},
 	\eest
where the coefficients $n^{\R}_{d,h}(g)$, called the real BPS states, satisfy (i) (integrality) $n^{\R}_{d,h}(g)\in \Z$, (ii) (finiteness) for each $d$, $n^{\R}_{d,h}(g)=0$ for large $h$,  and  (iii) (parity) $n^\R_{d,h}(g)$ has same parity as the complex BPS states $n^\cx_{d,h}(g)$. 
 \end{theorem}

\bigskip

\non{\bf Acknowledgments.} This paper is partially based upon work supported by the NSF grant DMS-1440140 while the authors were in residence at MSRI during the Spring 2018 program ``Enumerative Geometry Beyond Numbers". The authors would like to thank MSRI for the hospitality. P.G. would also like to thank the IHES for the hospitality during Fall 2018 and E.I. the Mittag-Leffler Institute during Fall 2015. 
The research of E.I. was partially supported by the Simons Foundation Fellowship \#340899 and that of P.G. is partially supported by ANR grant ANR-18-CE40-0009. 

\vskip.2in 
\section{Local Real Gromov-Witten   invariants}\label{LocalRGW}
\vskip.1in

\subsection{Real GW invariants} We start with a brief overview of the real Gromov-Witten invariants. Let $(X,\omega)$ be a symplectic manifold and $\phi$ an anti-symplectic involution on $X$. A symmetric Riemann surface $(C,\si)$ is a closed, oriented, possibly nodal, possibly disconnected  Riemann surface $\Si$ with an anti-holomorphic involution~$\si$.   A \textsf{real map}
$$
f:(C,\si)\longrightarrow (X,\phi)
$$
is a map $f: C \ra X$ such that $u\circ \si = \phi\circ u$. Let $\mathcal{J}^\phi_\omega$  denote the space of $\omega$-compatible almost complex structures $J$ on $X$ which satisfy $\phi^*J=-J$. For $\chi \in \Z$ and $B\in H_2(X,\Z)$, denote by 
$$
\ov \M_{B, \chi}^{\phi, \bullet}(X)
$$
the moduli space of equivalence classes (up to reparametrization of the domain) of stable degree $B$ $J$-holomorphic real maps from symmetric Riemann surfaces of Euler characteristic $\chi$. We will consider only the case when the restriction of the maps to each connected component  of the domain is {\em nontrivial}.  
\smallskip

In this paper we restrict ourselves to target manifolds which are themselves symmetric Riemann surfaces. We will use $(\Si, c)$ to denote the target curve and $d$ for the degree of the map. The real moduli space is denoted
$$
\ov \M_{d, \chi}^{c, \bullet}(\Si),
$$
and consists of real maps $f:(C, \si)\ra (\Si, c)$ whose domain may be disconnected. The involution on the domain decomposes the domain into   real components and pairs of conjugate components. Following  \cite[(1.7)]{GZ2}, an \textsf{$h$-doublet} is a real surface  
 \bear\label{def.doublet}
 (C,\si)=(C_1 \sqcup C_2, \si)=(C \sqcup \ov C, \si),  \mbox{ where}  \quad \si|_{C} =id:C \lra \ov C, 
 \eear
$C$ is a genus $h$ Riemann surface, and $\ov C$ denotes the curve $C$ but with the opposite complex structure. Note that every  real curve that has two components swapped by the involution is equivalent (up to reparametrization) to a doublet.

When $\Si$ is connected, it is therefore convenient to consider the following two moduli   spaces: 
 \bear\label{doublet.moduli}
 \ov \M_{d, h}^{c}(\Si) \quad \mbox{ and } \quad \ov{D\M}_{d, h}^{c}(\Si), 
 \eear
where  the first one consists of maps with connected domains of genus $h$   and the second one consists of maps whose  domains are $h$-doublets.
 Let 
 $$
 \ov{\mathcal{M}}_{d,\chi}^\bullet(\Si)
 $$
denote the classical moduli space of holomorphic maps from possibly nodal, possibly disconnected domains of Euler characteristic $\chi$ and degree $d$ to $\Si$ (whose restrictions to each connected component is nontrivial). Finally, denote by 
\bear\label{real.DM}
\ov{\R\M}^\bullet_{\chi, \ell}
\eear 
the real Deligne-Mumford moduli space parametrizing  (possibly disconnected)  symmetric surfaces $(C, \si)$ of Euler characteristic $\chi$ with 
$\ell$ pairs of conjugate marked points 
\bear\label{marked.pts}
\{(y_1^+,y_1^-),  \dots, (y_\ell^+,y_\ell^-)\}, \quad \mbox{ where } y_i^-=\si(y_i^+).
\eear
The corresponding moduli spaces of connected real and doublet domains are denoted by $\ov{\R\M}_{g, \ell}$ and $\ov{D\M}_{g, \ell}$ respectively. 
 
 \subsection{Twisted real orientations} The real moduli spaces are not in general orientable. In  \cite[Definition 1.2]{gz} a notion of real orientation was introduced whose existence ensures the orientability of the real moduli spaces when the target has odd complex dimension, cf. \cite[Theorem 1.3]{gz}. This notion can be extended to a twisted orientation as in Definition \ref{TRO} below when the target is a surface; see Definition~\ref{TRO.gen} for a general target.  In the appendix we show that \cite[Theorem 1.3]{gz} extends to this setting: a choice of twisted real orientation on an odd dimensional target determines a canonical orientation of the moduli spaces of real maps to that target. While a real orientation in the sense of \cite{gz} does not exist on a symmetric surface of even genus and fixed-point free involution, a twisted orientation exists on every symmetric surface.
 
As in \eqref{local.3fold}, when $L\ra (\Si,c)$ is a complex bundle then   
\bear\label{L.plus.twL}
(L\oplus c^*\ov L, c_{tw})\longrightarrow (\Si,c), \quad \mbox{ where } \quad  c_{tw} (z; u,v) = (c(z);  v, u),
\eear
is a Real bundle (i.e. a real bundle pair in the sense of \cite[\S1.1]{gz}).  Note that the projection onto the first factor identifies the fixed locus of $c_{tw}$ with
\best
(L\oplus c^*\ov L)^ {c_{tw}} \cong L_{|\Si^c}, 
\eest
where $\Si^c$ is the fixed locus of $c$.

\begin{defn}\label{TRO}  Assume $(\Si,c)$ is a symmetric surface. A \textsf{twisted (real) orientation  data}  
\bear\label{twist.or}
\fo=(\Theta, \psi,\mathfrak{s})
\eear 
for $(\Si,c)$ consists of 
\begin{enumerate}[(i)]
	\item a complex line bundle $\Theta$ over $\Si$ such that $c_1(\Theta\otimes c^* \ov \Theta)= -\chi(T\Si)$, 
	\item  a homotopy class of isomorphisms 
	\bear\label{real.isom.det}
	\Lambda^{\text{top}}(T\Si\oplus (\Theta\oplus c^*\ov\Theta), dc\oplus c_{tw})\overset{\psi}{\cong} (\Si\ti \cx, c\ti c_{std})
	\eear
	where $c_{std}:\cx \ra \cx$ is the standard complex conjugation.  
	\item  a spin structure $\mathfrak{s}$ on the fixed locus $T\Si^c\oplus \Theta_{|\Si^c}$, compatible with the orientation induced by \eqref{real.isom.det}.
\end{enumerate}	
\end{defn}
Up to deformation, rank $r$ complex or holomorphic bundles on a Riemann surface are determined by their first Chern class. Similarly,  rank $r$ Real bundles $(V,\phi)\lra (\Si,c)$ are classified by $c_1(V)$ and $w_1(V^\phi)$, cf. \cite{BHH}. 
In particular, condition (i) above ensures the existence of an isomorphism \eqref{real.isom.det}.  

\begin{ex}\label{R.or.choice} 
\begin{enumerate}[(a)]
		\item When $\Si^c$ is empty, there is no spin structure $\mathfrak{s}$ involved. Thus a choice of twisted orientation  in this case corresponds only to a choice $(\Theta, \psi)$. 
		
		\item When $(\Si,c)$ is a $g$-doublet, $\Theta$ restricts to a line bundle on each component $\Si_i$; let  $m_i=c_1(\Theta)[\Si_i] $ denote the degrees.  Since $(c^*\ov \Theta)_{|\Si_1}=c^*(\ov \Theta_{|\Si_2})$, condition (i) restricts the total degree 
		\bear\label{m=sum.mi}
		m=m_1+m_2=2g-2.
		\eear 		
For any fixed complex line bundle $\Theta$ over the doublet satisfying \eqref{m=sum.mi}, there is a unique isomorphism \eqref{real.isom.det} up to homotopy (determined by the restriction to $\Si_1$). Moreover, $\Si^c$ is empty for a doublet. Thus a twisted orientation for a doublet consists of a choice of the degrees $m_i\in \Z$ satisfying \eqref{m=sum.mi}.
		
		\item When $(\Si, c)$ is a connected genus $g$ surface, the degrees of $c^*\ov\Theta$ and $\Theta$ are equal. In this case, condition (i) implies that the degree of $\Theta$ is $m=g-1$; up to isomorphism, there is only one such complex line bundle. 
	\end{enumerate}
\end{ex}

A twisted orientation $\mathfrak{o}$ on $(\Si, c)$ equips the Real moduli spaces $\ov \M_{d, \chi}^{c, \bullet}(\Si)$ with a canonical orientation, cf. Appendix. In particular, it  gives rise to a virtual fundamental class 
$$[\ov \M_{d, \chi}^{c, \bullet}(\Si)]^{\vir,\mathfrak{o}}$$ in dimension  $b= d\chi(\Si) -\chi$.

\subsection{Absolute RGW invariants}  Consider a holomorphic bundle $E$ over a complex curve $\Si$. Then the operator $\del_E$ determines a family of complex operators over moduli spaces of maps to $\Si$; the fiber at a stable map $f:C\ra \Si$ is the pullback operator 
$\del_{f^* E}$. Denote by  $\mathrm{Ind}\;  \del_{E}$ the index bundle associated to this family of operators, regarded as an element in K-theory. 

Assume next $L$ is a holomorphic line bundle over a symmetric surface $(\Si, c)$, and let $E= L\oplus c^*\ov L$.  It is a rank 2 holomorphic bundle over $\Si$ which has a Real structure $c_{tw}$ given by \eqref{L.plus.twL}. Let $\dbar_{(L\oplus c^*\bar L, c_{tw})}$ denote the restriction of 
$\dbar_{L\oplus c^*\bar L}$ to the invariant part of its domain and target, cf. \cite[\S4.3]{gz}. Via the projection onto the first factor, the kernel and cokernel of $\dbar_{(L\oplus c^*\bar L, c_{tw})}$ are canonically identified with the kernel and cokernel of $\dbar_{L}$.

Similarly $\dbar_{(L\oplus c^*\bar L, c_{tw})}$ determines a family of pullback operators over the real moduli space of maps to 
$(\Si, c)$, and the projection onto the first factor identifies 
\bear\label{ind.E}
\Ind \dbar_{(L\oplus c^*\bar L, c_{tw})}\overset{\pi_1}\cong \Ind \dbar_{L}.
\eear
The right hand side carries a natural complex structure, which pulls back to one the left hand side.
An $U(1)$-action on $L$ induces one on $(L\oplus c^*\bar L, c_{tw})$, compatible with the real structure. In turn, these induce $U(1)$-actions on $\Ind \del_L$ and $\Ind \del_{(L\oplus c^*\bar L, c_{tw})}$ and the isomorphism \eqref{ind.E} identifies their equivariant Euler classes.

\smallskip

Motivated by the Bryan-Pandharipande construction \cite[\S2.2]{bp1}, we consider the following real version, associated to a local Real 3-fold $(L\oplus c^*\ov L,c_{tw})\ra (\Si, c)$ defined by \eqref{L.plus.twL}. 
\begin{defn}\label{D.local.real.invar} Assume $(\Si,c)$ is a symmetric surface, $L$ a holomorphic line bundle over $\Si$ and 
$\mathfrak{o}$ a twisted orientation data \eqref{twist.or} for $(\Si,c)$. The \textsf{local real GW invariants} are defined by the equivariant pairings:
	\bear\label{defn.Z.real}
	RZ^{c,\mathfrak{o}}_{d, \chi} (\Si, L)=  \int_{[\ov \M_{d, \chi}^{c, \bullet}(\Si)]^{\vir,\mathfrak{o}}} e_{U(1)}(-\mathrm{Ind}\;  \del_{(L\oplus c^*\bar L,  c_{tw})})=
	 \int_{[\ov \M_{d, \chi}^{c, \bullet}(\Si)]^{\vir,\mathfrak{o}}} e_{U(1)}(-\mathrm{Ind} \; \del_{L}).
	\eear
	Here $e_{U(1)}$ denotes the $U(1)$-equivariant Euler class. 
\end{defn}	
As in \cite[\S2.2]{bp1}, we will primarily consider the shifted partition function:  
\bear\label{defn.RGW.real}
RGW^{c,\mathfrak{o}}_{d} (\Si,L)= \sum_{\chi}u^{d(\frac{\chi(\Si)}{2 }+c_1(L)[\Si]) -\frac{\chi}{2 }} RZ^{c,\mathfrak{o}}_{d, \chi} (\Si, L).
\eear
Intrinsically, \eqref{defn.Z.real} takes values in the equivariant cohomology of a point: 
\best
RZ^{c,\mathfrak{o}}_{d, \chi} (\Si, L) \in H^*_{U(1)}(pt)= H^*(\cx\P^\infty)= \Q[t]. 
\eest
Here $t$ is the equivariant first Chern class of the standard representation of $U(1)$. Then the local invariant \eqref{defn.Z.real} can be expressed in terms of the equivariant parameter $t$ and an {\em ordinary} integral: 
\bear\label{defn.Z.real.nonequiv}
RZ^{c,\mathfrak{o}}_{d, \chi} (\Si, L)=  t^{\iota-b/2}\int_{[\ov \M_{d, \chi}^{c, \bullet}(\Si)]^{\vir,{\mathfrak{o}}}} c_{b/2}(-\mathrm{Ind} \; \del_{L}).
\eear
Here $b$ is the dimension of $\ov \M_{d, \chi}^{c, \bullet}(\Si)$ and $\iota$ the index (virtual complex rank) of 
$-\mathrm{Ind} \; \del_{L}$, given by:  
\bear\label{rank.ind.L}
\iota=\mathrm{rank}_\cx(- \; \mathrm{Ind}\;  \del_{L})= - dc_1(L)[\Si]- \tfrac12 \chi.  
\eear
\begin{rem}\label{ConnRGW}
	The invariants $RGW^{c,\mathfrak{o}}_{d} (\Si,L)$ count maps from possibly disconnected real domains.  The real structure $\si$ acts on the components of the domain decomposing them into `real components' (preserved by $\si$) and `doublets' (pairs of conjugate components swapped by $\si$). 
	When $\Si$ is connected, we denote the connected domain invariants by
\bear\label{conn.invar}
CRGW^{c,\mathfrak{o}}_{d} (\Si,L) = \sum_{h=0}^\infty u^{d(\frac{\chi(\Si)}{2 }+ c_1(L)[\Si])+h -1}  \int_{[\ov \M_{d, h}^{c}(\Si)]^{\vir,\mathfrak{o}}} e_{U(1)}(-\mathrm{Ind} \; \del_{L})
\eear
and the doublet domain invariants (which appear only in even degree when $\Si$ is connected) by
\bear\label{doublet.invar}
DRGW^{c, \mathfrak{o}}_{d} (\Si,L) = \sum_{h=0}^\infty u^{d(\frac{\chi(\Si)}{2 }+  c_1(L)[\Si])+2h-2}  
\int_{[\ov {D\M}_{d, h}^{c}(\Si)]^{\vir,  \mathfrak{o}}} 
e_{U(1)}(-\mathrm{Ind} \; \del_{L}).
\eear
Here $\ov \M_{d, h}^{c}(\Si)$ and $\ov {D\M}_{d, h}^{c}(\Si)$ are the moduli spaces \eqref{doublet.moduli} of degree $d$ maps with connected genus $h$ domain and $h$-doublet domain, respectively.  Then
\bear\label{rgw=exp}
1+\sum_{d=1}^\infty RGW^{c,\mathfrak{o}}_{d} (\Si,L)q^d = 
\exp\left(\sum_{d=1}^\infty CRGW^{c,\mathfrak{o}}_{d} (\Si,L)q^d+\sum_{d=1}^\infty DRGW^{c,\mathfrak{o}}_{2d} (\Si,L)q^{2d}\right).
\eear
\end{rem}
\subsection{Notation for partitions} A partition $\la$ is a finite sequence of positive integers 
$\la=(\la_1\ge \dots\ge \la_\ell)$. A partition of $d$, denoted $\la\vdash d$, is a partition such that the sum of its parts, denoted 
$|\la|$, is equal to $d$. Its length (number of parts  $\ell$) is denoted $\ell(\la)$. We can also write a partition in the form $\la=( 1^{m_1} 2^{m_2} \dots )$ where $m_k$ is the number of parts of $\la$ equal to $k$. Then 
\best
d=|\la|= \sum_{i=1}^\ell \la_i =\sum_{k=1}^\infty  k m_k \quad \mbox{ and } \quad \ell(\la)=\ell=\sum_{k=1}^\infty m_k. 
\eest
We will also consider the following combinatorial factor
\bear\label{mult.la}
\zeta(\lambda) = \prod m_k! k^{m_k}.
\eear
A partition $\la$ is uniquely determined by its Young diagram and the conjugate partition $\la'$ is obtained by reflecting $\la$ across the main diagonal.  The rank 
\bear\label{rank.la}
r(\la)
\eear of a partition is the length of the main diagonal of its Young diagram, cf. \cite[\S4.1]{FH}. 

\subsection{Relative RGW invariants} Assume next that $(\Si, c)$ is a {\em marked} symmetric surface, with $r$ pairs of marked points 
\bear\label{marked.points}
P_\Si=\{ (x_1^+,x_1^-), \dots, (x_r^+,x_r^-)\}, \quad \mbox{ where} \quad  x_i^-=c(x_i^+), 
\eear
cf. \eqref{marked.pts}. So in particular we have a preferred marked point $x_i^+$ (the first element of a pair) in each pair of conjugate points. 

We consider next the moduli spaces of real maps to $(\Si,c)$ that have fixed ramification pattern over the marked points of $\Si$. This moduli space is a version of \cite[Definition~3.1]{bp1}, adapted to the Real setting. The ramification pattern over each point is described by a partition $\la$. 

Let $\vec \la=(\la^1, \dots, \la^r)$ be a collection of $r$ partitions of $d$. 
\begin{defn}\label{D.rel,real} Denote by 
\bear\label{rel.M.space}
	\ov  \M^{\bullet, c}_{d, \chi}  (\Si)_{\la^1, \dots, \la^r}
	\eear 
the relative real moduli space of degree $d$ stable real maps $f:(C,\si)\ra (\Si, c)$ such that 
	\begin{itemize} 
		\item $f$ has ramification pattern $\la^i$ over $x_i^+$ (and thus also over $x_i^-=c(x_i^+)$), for all $i=1, \dots, r$;
		\item  the domain $C$ is possibly disconnected and has total Euler characteristic $\chi$;
		\item $f$ is nontrivial on each connected component of $C$.
	\end{itemize} 
\end{defn} 
 
Here, as in \cite[Definition~3.1]{bp1}, the inverse images of the marked points of the target are {\em not} ordered; in particular, an automorphism of $f$ may permute domain components or points in the inverse image of the marked points of the target. It is straightforward to express these moduli spaces in terms of unions, products, and finite quotients of the relative moduli spaces where the points in the inverse images $f^{-1}(x_i^\pm)=\{ y_{ij}^\pm\}_{j=1, \dots, \ell(\la^i)}$ are all marked, the points $y_{ij}^\pm$ are conjugate,  $f(y_{ij}^+)=x_i^+$,  
and the  ramification order of $f$ at $y_{ij}^\pm$ is $\la_j^i$, for $j=1, \dots, \ell(\la^i)$ and $i=1, \dots, r$. The moduli space 
$\ov  \M^{\bullet, c}_{d, \chi}  (\Si)_{\la^1, \dots, \la^r}$ has virtual dimension~$b$, where 
\bear\label{dim.M=b}
b=d\chi(\Si)- \chi-2\de(\vec \la) \quad \text{ and }\quad  \de(\vec \la)= \sum_{i=1}^r  ( d-\ell(\la^i)). 
\eear
Here $\ell(\la^i)$ is the length of the partition $\la^i$, i.e. the cardinality of $f^{-1}(x_i^+)$.\\

The relative real moduli space is oriented using a twisted orientation $\mathfrak{o}$ as in Definition \ref{TRO} but where $T\Si$ is the relative tangent space to the {\em marked} curve $\Si=(S, j, x_1^\pm, \dots x_r^\pm)$, i.e. 
\bear\label{T.punctured}
T\Si= TS\otimes \O\Big(-\ma{\textstyle\sum} _i x_i^+ -\ma{\textstyle\sum}_i x_i^-\Big); 
\eear
see Appendix. Definition~\ref{D.local.real.invar} then extends to the relative setting.
\begin{defn} Assume $(\Si,c)$ is a symmetric surface with $r$ pairs of marked points. Let $L\ra \Si$ be a holomorphic line bundle, 
$\mathfrak{o}$ a twisted orientation data for $(\Si,c)$, and $\vec \la=(\la^1, \dots, \la^r)$ a collection of $r$ partitions of $d$. The \textsf{local real relative GW invariants} associated with the Real 3-fold  $(L\oplus c^*\ov L,c_{tw})\ra (\Si, c)$ and the orientation data $\fo$ are the equivariant pairings:
	\bear\label{defn.Z.real.rel}
 	RZ^{c,\mathfrak{o}}_{d, \chi} (\Si, L)_{\vec \la}= \hskip-.2in
	\ma\int_{[\ov \M_{d, \chi}^{c, \bullet}(\Si)_{\vec \la}]^{\vir,\mathfrak{o}}} \hskip-.2in e_{U(1)}(-
 	\mathrm{Ind}\;  \del_{(L\oplus c^*\bar L, c_{tw})})= \hskip-.2in
 	 \ma\int_{[\ov \M_{d, \chi}^{c, \bullet}(\Si)_{\vec \la}]^{\vir,\mathfrak{o}}}  \hskip-.2in  e_{U(1)}(-\mathrm{Ind} \; \del_{L})
 	.
	\eear
\end{defn} 
The shifted partition function \eqref{defn.RGW.real} extends to the relative setting as
	\bear\label{defn.RGW.real.rel}
	RGW^{c,\mathfrak{o}}_{d} (\Si,L)_{\vec \la}= 
	 \sum_{\chi}u^{d(\frac{\chi(\Si)}{2 }+ c_1(L)[\Si]) -\frac{\chi}{2 }-\delta(\vec \la)} 
	RZ^{c,\mathfrak{o}}_{d, \chi} (\Si, L)_{\vec \la },
	\eear
	where $\delta(\vec \la)$ is as in \eqref{dim.M=b}. Note that the power of $u$ is $b/2+dk$, where $b$ is the dimension \eqref{dim.M=b} of $\ov \M_{d, \chi}^{c, \bullet}(\Si)_{\vec \la}$ and $k=c_1(L)[\Si]$. 
	
	The quantity  \eqref{defn.RGW.real.rel} is invariant under (smooth) deformations, so it depends only on the topological type of $(\Si, c, \fo)$, on $c_1(L)$, and on how the $r$ partitions $\la^1,\dots, \la^r$ are distributed on the components of $\Si$. We use the notation 
\bear\label{rgw.conn}
	RGW^{c,\mathfrak{o}}_{d} (g|k)_{\vec \la}
\eear
	for the case $\Si$ is a connected genus $g$ surface and $k=c_1(L)[\Si]$, and
\bear\label{rgw.doublet}
	RGW^{c,\mathfrak{o}}_{d} (g,g|k_1,k_2)_{\vec \la}
\eear
for the case $\Si$ is a  $g$-doublet, all the positive marked points are on the same component 
$\Si_1$ of $\Si$, and $k_i=c_1(L)[\Si_i]$ are the degrees of $L$ on the two components.\\


As before, the local invariant \eqref{defn.Z.real.rel} can be expressed in terms of the equivariant parameter $t$ and an {\em ordinary} integral: 
\bear\label{defn.Z.real.rel.nonequiv}
RZ^{c,\mathfrak{o}}_{d, \chi} (\Si, L)_{\la^1.. \la^r}=  t^{\iota-b/2}\int_{[\ov \M_{d, \chi}^{c, \bullet}(\Si)_{\la^1.. \la^r}]^{\vir,\mathfrak{o}}} c_{b/2}(-\mathrm{Ind} \; \del_{L}).
\eear
Here $b$ is the dimension \eqref{dim.M=b} of $\ov \M_{d, \chi}^{c, \bullet}(\Si)_{\la^1.. \la^r}$ and $\iota$ the index (virtual complex rank) of 
$-\mathrm{Ind} \; \del_{L}$, given respectively by \eqref{dim.M=b} and \eqref{rank.ind.L}, so the power of $t$ in \eqref{defn.Z.real.rel.nonequiv} is
\bear\label{power.t}
\iota-b/2&=&-d(\chi(\Si)/2+c_1(L)[\Si])+\de(\vec \la). 
\eear
As in the absolute case, we use similar notions for the connected and doublet relative invariants and their moduli spaces, cf. Remark~\ref{ConnRGW}.

 \vskip.2in 
\section{Doublet vs complex invariants}\label{Doublets}
 \vskip.1in 

The doublet invariants \eqref{doublet.invar} (and their extension to the relative setting) are real invariants associated with the moduli space of maps whose domain is a doublet \eqref{def.doublet}. In this section we consider two situations: (a) when the target curve is a doublet and (b) when the target curve is connected. In both cases, we relate the doublet invariants to the residue invariants defined by Bryan and Pandharipande in \cite{bp1} (for the anti-diagonal action). The latter are reviewed in  \S\ref{S.cxBP}. 

Roughly speaking, the main idea is that a doublet can be identified with a complex curve by restricting to one of the components. This defines an identification $\cal P$ between the doublet moduli space and the usual (complex) moduli space, with matching deformation obstruction theories; moreover, a bundle on a doublet corresponds to two bundles, one for each component of the doublet. 

The main results in this rather technical section are Corollaries~\ref{corR=GW.part.rel} and 
\ref{corcdefn.Z.part.anti-diag-balanced}, comparing the doublet invariants to the BP-invariants. They follow from the fact that in both cases (i) the VFC of the doublet moduli space is equal up to a scalar multiple to that of the corresponding complex moduli space, cf. Lemmas~\ref{double.vfc.lemma}  and \ref{lemcdouble.vfc} and (ii) the equivariant Euler classes of the index bundles are also equal up to sign, cf. Lemmas~\ref{P.res=} and \ref{cP.res=}. 

\subsection{Complex GW invariants}\label{S.cxBP} We begin with a brief review of the complex moduli space and the residue GW-invariants defined by Bryan and Pandharipande in \cite{bp1}. Assume $\Si$ is a complex curve with $r$ marked points $P=\{ x_1, \dots x_r\}$, and let $\vec \la=(\la^1, \dots \la^r)$ be a collection of $r$ partitions $\la^i$ of $d$. Let
\bear\label{cx.mod.bp}
\ov \M_{d, \chi}^\bullet(\Si)_{\vec \la}
\eear denote the usual (complex) relative moduli space \cite[Definition~3.1]{bp1} of degree $d$  stable maps $f:C\ra \Si$  from an  Euler characteristic $\chi$ domain having ramification prescribed by $\vec \la$ over the points $P$ (such that moreover the restriction of $f$ to each connected component of the domain is nontrivial). Here the inverse images of the marked points of  the target are {\em unordered}. The moduli space \eqref{cx.mod.bp} is canonically oriented and carries a virtual fundamental class in dimension $2b$,  where 
\best
b= d\chi(\Si)- \chi-\de(\vec \la)
\eest
and $\delta(\vec \la)$ is as in \eqref{dim.M=b}. 

If $L_1,L_2$ are two holomorphic bundles over $\Si$, the total space of 
\bear 
E=L_1\oplus L_2\ra \Si
\eear 
is a local holomorphic 3-fold with a $T=(\cx^*)^2$ action.  In \cite[\S3.2]{bp1} Bryan-Pandharipande consider residue invariants by integrating a $T$-equivariant Euler class.  When restricted to the anti-diagonal $U(1)$ action, the BP residue invariants are given by:
\bear\label{defn.Z.part.anti-diag-balanced}
Z_{d, \chi}(\Si | L_1, L_2)_{\vec \la}=  \int_{[\ov \M_{d, \chi}^{\bullet}(\Si)_{\vec \la}]^\vir} e_{U(1)}(-\mathrm{Ind} \; \del_{L_1\oplus L_2}).
\eear 
Their (shifted) generating function (cf. \cite[\S3.2]{bp1}) is 
\bear\label{shifted.GW}
GW_d(\Si| L_1, L_2)_{\vec \la} =
\sum_{\chi} u^{d(\chi(\Si)+k_1+k_2) -\chi -\delta(\vec \la)} Z_{d,\chi}(\Si |L_1,L_2)_{\vec \la}, 
\eear
where $k_i=c_1(L_i)[\Si]$ and  $\delta(\vec \la)$ is as in \eqref{dim.M=b}. We denote by  $GW^{conn}$ the corresponding invariants associated to the moduli spaces of maps with {\em connected} domains.

\subsection{Doublets and Halves} For any doublet $(C=C_1\sqcup C_2, \si)$, with $r$ pairs of marked points $P_{C}$ as in  \eqref{marked.points} the 'half' $C_1$ is a complex curve with $r$ marked points and each of these marked points inherits a  decoration of a  $\pm$ sign. 
This process defines a map
\bear\label{doublet.to.half}
(C=C_1\sqcup C_2, \si) \mapsto C_1, 
\eear
that takes a doublet to a connected complex curve with {\em signed} marked points. Formally,  a complex curve with \textsf{signed marked} points is a complex curve $\Si$ with marked points $P_{\Si}=\{x_1, \dots, x_r\}$ together with a choice $\ep: P_{\Si} \ra  \{\pm  1\}$ of a sign associated to each point.

Conversely, to every  complex curve $C$ we can associate a doublet \eqref{def.doublet} via
\bear\label{double.of.signed}
(DC, \si)\,, \quad \mbox{ where } \quad DC=C\sqcup \ov C=C_1\sqcup C_2\quad \mbox{and }\quad \si_{|C}=id: C \ra \ov C.
\eear
Note that $DC$ is the orientation double cover of $C$. When $C$ has $r$ signed marked points $P_C$, the double $DC$ is marked: it has $r$ pairs of conjugate points, and the sign $\ep$ of a marked point in $P_C$ determines 
whether it is the first or second element of the corresponding pair in $DC$, with $+$ corresponding to first. 

Therefore \eqref{doublet.to.half} is a correspondence. 
 
\subsection{Real maps to a doublet} \label{S.maps.to.doublet} Fix $\Si$ a complex marked surface (with signed marked points) and let $D\Si=(\Si\sqcup \ov\Si, c)=(\Si_1\sqcup \Si_2, c)$ denote its double \eqref{double.of.signed}. We next relate the local    $RGW$ invariants \eqref{defn.Z.real} of the double $D\Si$ to the BP-residue invariants \eqref{defn.Z.part.anti-diag-balanced} of $\Si$. 

For any real map $f:(C, \si) \ra D\Si= (\Si_1\sqcup \Si_2, c)$, let 
\bear\label{half.f} 
f_i:C_i \ra \Si_i\,,\quad i=1,2,
\eear
denote its restriction to $C_i=f^{-1}(\Si_i)$, $i=1,2$. Conversely, any map $f: C\ra \Si$ doubles to a real map
\best\label{double.both} 
\tilde f:  DC \ra D\Si\,,  \quad \mbox{ with } \quad \tilde f_1 = f.
\eest
The signs of the marked points on $\Si$ determine signs on the marked points of the domain~$C$ which are compatible under the doubling procedure \eqref{double.of.signed}. 
This defines a morphism
\bear\label{defn.double.both}
\mathcal{D}:\ov \M_{d, \chi}^\bullet(\Si)_{\la^1.. \la^r}\lra\ov \M_{d, 2\chi}^{c,\bullet}(D\Si )_{\la^1.. \la^r}, \qquad f\mapsto \wt f,
\eear
between the moduli spaces, whose inverse is
\bear\label{defn.doublet}
\cal P( f)= f_1
\eear
where $f_1$ is given by \eqref{half.f}.  

\begin{lemma}\label{double.vfc.lemma} Fix an orientation data $\mathfrak{o}$ as in \eqref{twist.or} for the 
doublet $D\Si=\Si_1\sqcup \Si_2$. With the notation above, the identification \eqref{defn.doublet} has degree 
$(-1)^{dm_2+\ell_2}$, i.e. 
	\bear\label{double.vfc}
	[\ov \M_{d,2 \chi }^{c, \bullet}(D\Si)_{\la^1.. \la^r} ]^{\vir,\mathfrak{o}}
	=(-1)^{dm_2+\ell_2}  {\cal D}_{*}[\ov \M_{d, \chi}^\bullet(\Si)_{\la^1.. \la^r}]^\vir,
	\eear
where $m_2$ is the degree  of $\Theta_{|\Si_2}$   and $\ell_2$ is the sum of the lengths 
of the partitions associated to the positive points on $\Si_2$: 
	\bear\label{ep+s}
	m_2=c_1(\Theta)[\Si_2] \quad \mbox{ and } \quad \ell_2=\sum_{x_i^+\in \Si_2} \ell (\lambda^i).
	\eear
	\end{lemma}

\begin{proof} The map \eqref{defn.double.both} and its inverse \eqref{defn.doublet} define an identification between the two moduli spaces, with matching deformation-obstruction theories (relative the domains). Thus it remains to compare the orientations. The argument is similar to that of  \cite[Theorem~1.3]{GZ2}  taking into account the difference in the orientations induced by a twisted orientation data and a real orientation data in the sense of  \cite{gz}. 

 We first recall the procedures for orienting the complex and the real moduli spaces; the Appendix contains a more detailed discussion of the real case. The orientation sheaf of the real  moduli space, (after stabilization of the domain if necessary), is canonically identified with 
	\bear\label{or.moduli.rel} 
	\det  T\ov  \M^{c, \bullet}_{d, \chi}  (\Si)_{\la^1..\la^r} =\det \del_{(T\Si, dc)} \otimes  
	\mathfrak{f}^*\det T \ov{ \R\M}^\bullet_{\chi, \ell}.
	\eear
Here $\mathfrak{f}$ is the map to the real Deligne-Mumford moduli space parametrizing real curves of Euler characteristic $\chi$ and $\ell$ pairs of conjugate marked points, and $\ell=\ma\sum_{i=1}^r \ell(\la^i)$;  see \eqref{or.moduli.rel.a}.
	Let   
	$$
	f:(C,\si)\lra (\Si,c)
	$$
	 be a point in the real moduli space. A choice of twisted orientation data $\mathfrak o=(\Theta,\psi, \mathfrak s)$ determines a homotopy class of isomorphisms 
\bear\label{oriso}
f^*(T\Si\oplus \Theta\oplus c^*\ov \Theta, dc\oplus c_{tw})\ma\lra^{\phi_\fo} (C\times \cx^{\oplus 3}, \si\oplus c_{std}^{\oplus 3}).
\eear
	Here $T\Si$ denotes the relative tangent bundle \eqref{T.punctured} of the marked curve. This induces an isomorphism
	\bear\label{detTS=twist}
	\det \dbar_{f^*(T\Si, dc)} =\det \dbar_{(\cx, c_{std})}
	\eear
by using the canonical orientation on twice a bundle and the canonical complex orientation induced by the right-hand side of  the identification
	\best
	\det \dbar_{f^*(\Theta\oplus c^*\ov\Theta, c_{tw})}\overset{\pi_1}{=} \det \dbar_{f^*\Theta}
	\eest
	as in \eqref{ind.E}. 
	By \cite[Theorem 1.3]{gz}, there is also a canonical isomorphism
	\bear\label{or.DM} 
	\det (T\ov{\R\mathcal{ M}}_{h, \ell})= \det \dbar_{(\cx, c_{std})},
	\eear
where the forgetful morphism of a pair of marked points is oriented via the first elements in the pairs.  Then the orientation on the real moduli space is obtained by combining \eqref{detTS=twist} and \eqref{or.DM}  within  \eqref{or.moduli.rel}.

Similarly the complex moduli space at $f:C\ra\Si$ is oriented via the complex orientation of $\det \dbar_{T{\Si}}$ and the complex orientation on the corresponding Deligne-Mumford moduli space as in \eqref{or.moduli.rel}. 

 Since the map $\cal D$ is compatible with the forgetful morphism to the corresponding DM spaces, its sign is determined by the comparison on the level of DM spaces and on the level of the index bundles. 

When $(C,\si)=(C_1\sqcup C_2, \si)$ is a doublet and $(V,\phi)=(V_1\sqcup V_2, \phi)\ra(C,\si)$ is a Real bundle, its index bundle 
$\Ind\dbar_{(V,\phi)}$ has a natural complex structure induced by the isomorphism:
\bear\label{detTS=cx}
\Ind \dbar_{(V,\phi)} \overset{p_1}{\cong}  \Ind  \dbar_{V_1}.
\eear
Here $p_1$ takes an invariant section $\xi=(\xi_1,\xi_2)$ of $(V_1\sqcup V_2,\phi)$ to its restriction $\xi_1$ to $C_1$. In particular, $\det \dbar_{(V,\phi)}$ has an induced orientation, which we refer to as the \textsf{complex orientation},  cf. \cite[\S3.1]{GZ2}.

On the level of Deligne-Mumford spaces, the doubling map $\mathcal{D}$ from the complex moduli space (with signed marked points) to the real moduli induces an orientation on $\ov{D\mathcal{ M}}_{h, \ell}$ which we call the \textsf{complex orientation},  cf. \cite[\S3.1]{GZ2}. By \cite[Lemma 3.2]{GZ2}, the comparison between the orientation on 
$$
\det (T\ov{D\mathcal{ M}}_{h, \ell})\otimes \det \dbar_{(\cx, c_{std})},
$$ induced by \eqref{or.DM} and by the complex orientations on the two factors is $(-1)^{\chi/2+s}$, where $\chi$ is the Euler characteristic of $C_1$ and $s$ is the number of negative marked points on the component $C_1$. Because $C$ is a doublet, $s$ is also equal to the number of positive marked points on the component $C_2$, i.e. the number $\ell_2$ of points in the inverse image of marked points $x_i^+$ that lie on $\Si_2$.

We now turn to the comparison at the level of index bundles. The twisted orientation $\fo$ determines an orientation on
 $$
	\det\dbar_{ f^*(T{\Si}, dc)}\otimes\det \dbar_{(\cx, c_{std})}
$$
via \eqref{oriso} and \eqref{detTS=twist}. In the case when the domain is a doublet, the two factors in this tensor product also have complex orientation as in \eqref{detTS=cx}. To understand the difference between the two orientations on the tensor product we consider the restriction of (\ref{oriso})  to $C_1$. This restriction is a complex isomorphism of complex bundles and  thus the induced isomorphism on the corresponding determinant bundles is orientation preserving.  
	Therefore the difference between the two orientations on the tensor product corresponds to the difference between   the complex orientation on $\det \dbar_{ f^*(\Theta \oplus c^*\ov\Theta, c_{tw})}$ induced by \eqref{detTS=cx} 
and the orientation \eqref{ind.E} on 
\best
\det \dbar_{ f^*(\Theta\oplus c^*\ov\Theta, c_{tw})}\overset{\pi_1}{=}\det \dbar_{ f^*\Theta|_{\Si_1\sqcup \Si_2}}
\eest
used in the transition from \eqref{oriso} to \eqref{detTS=twist}.

 By Lemma \ref{compor}  below,   the difference between  these orientations   is $(-1)^{\iota_2}$, where 
\best\label{ind.theta.2}
\iota_2= c_1(f^* \Theta_{|\Si_2})+\chi/2=dc_1(\Theta)[\Si_2]+ \chi/2= dm_2+ \chi/2
\eest
is the complex rank of the index bundle associated to $\Theta_{|\Si_2}$. Combined with the change $(-1)^{\chi/2+\ell_2}$ at the level of the DM moduli spaces this completes the proof. 
\end{proof}

\begin{lemma}\label{compor} 
	The index bundle of $(L\oplus c^*\ov L ,  c_{tw})\lra (C_1\sqcup  C_2,c)$ has two natural  orientations: 
	\begin{enumerate} [(i)]  
	\item one induced by the isomorphism with $\Ind \dbar_{L_{|C_1\sqcup C_2}}$  via the projection \eqref{ind.E} onto the first bundle.
	\item another one induced by the isomorphism with $\Ind\dbar_{(L\oplus c^*\ov L)_{|C_1}}$ via the restriction \eqref{detTS=cx} to $C_1$.
	\end{enumerate} These orientations differ by a factor of $(-1)^{\iota_2}$, where
	\bear\label{iota2}
\iota_2=\mathrm{rank}_\cx\dbar_{L_{ |C_2}}= c_1(L)[ C_2]+\chi/2,
	\eear
	 and $\chi$ is the Euler characteristic of $C_2$.
\end{lemma}
\begin{proof} Holomorphic sections of $L\oplus c^*\ov L\lra C_1 \sqcup  C_2$ invariant under the involutions $c,c_{tw}$ have the form $(\xi,\eta)$ where  $\xi_i= \xi_{|C_i}$ are holomorphic sections of $L|_{C_i}$ while $\eta_i= \eta_{|C_i}$ are holomorphic sections of $(c^*\ov L)|_{C_i}$, and 
	\bear\label{eta=xi.c}
	\eta_1=c^* \xi_2 \quad  \mbox{ and }\quad \eta_2=c^*\xi_1.
	\eear 
Note that $\eta_1$ is a section of $(c^*\ov L)_{|C_1}= c^* (\ov L_{|C_2})$. Thus we have two natural isomorphisms from the space of $(c,c_{tw})$-invariant sections:
	\bear\label{proj.1.2}
	\Ind\dbar_{(L\oplus c^*\ov L, c_{tw})}\lra \Ind \dbar_{L_{|C_1\sqcup C_2}}
	=\Ind  \dbar_{L_{|C_1}}\oplus \Ind  \dbar_{L_{|C_2}}, \quad (\xi,\eta) \mapsto \xi=(\xi_1, \xi_2)
	\eear
	and
	\bear\label{cxstr2}
	\Ind\dbar_{(L\oplus c^*\ov L, c_{tw})}\lra \Ind \dbar_{ (L\oplus c^*\ov L)_{|C_1}} \hskip-.1in =
	 \Ind \dbar_{L_{|C_1}} \oplus  \Ind \dbar_{(c^*\ov L)_{|C_1}}, 
	 \quad (\xi,\eta) \mapsto  
	(\xi,\eta)_{|C_1}= (\xi_1, \eta_1).
	\eear
Both $\Ind \dbar_L $ and  $\Ind \dbar_{ (L\oplus c^*\ov L)_{|C_1}}$ have natural complex structures and therefore induce two complex structures on $\Ind\dbar_{(L\oplus c^*\ov L, c_{tw})}$ which we want to compare. 

There is a complex bundle isomorphism:
\bear\label{proj.switch.12}
\Ind  \dbar_{L_{|C_2}}  \lra \ov{\vphantom{\sum}\Ind \dbar_{c^*(\ov L_{|C_2})}}, \qquad \xi_2  \mapsto \xi_2\circ c. 
\eear
(using the fact that $\ind \del_{L_2 \ra C_2}$ and $\ind \del_{\ov L_2\ra \ov C_2}$ have opposite complex structures). 
This combined with \eqref{eta=xi.c} implies that the orientations induced by \eqref{proj.1.2} and \eqref{cxstr2} differ by $(-1)^{\iota_2}$, where $\iota_2$ is the complex rank of the index of $L_{|C_2}$, given by \eqref{iota2}.
\end{proof}
\medskip

Next, given two complex line bundles $L_1, L_2\ra\Si$ over a complex curve, we obtain a complex line bundle $L \ra D\Si$ over the double $(D\Si,c)=(\Si\sqcup\ov\Si,c)=(\Si_1\sqcup\Si_2,c)$ defined by 
\bear\label{double.L}
 L_{|\Si_1}=L_1 \quad  \mbox{ and } \quad  L_{|\Si_2}=c^*\ov {L}_2.  
\eear
We denote such $L$ by 
$$
D(L_1,L_2)\lra D\Si. 
$$ 
Note that if $L_1, L_2\ra \Si$ have degrees $k_1, k_2$, then $L_{|\Si_i}$ also has degree $k_i$, $i=1,2$.

\begin{lemma}\label{P.res=} With the notation above, the morphism \eqref{defn.doublet} satisfies: 
	\bear\label{res.cx=res.real}
	e_{U(1)}(-\mathrm{Ind}\; \del_{D(L_1,L_2 )}) =	(-1)^{dc_1(L_2)[\Si]+\chi/2}
	\mathcal{P}^*e_{U(1)}(-\mathrm{Ind} \; \del_{L_1\oplus L_2}),
	\eear
	with  the anti-diagonal action on $L_1\oplus L_2$ used  for the equivariant Euler class in the last expression. 
\end{lemma} 

\begin{proof} When $L \ra D\Si$ is a line bundle over a doublet $D\Si=\Si_1\sqcup \Si_2$, the identification \eqref{proj.1.2} induces an 
isomorphism 
\best 
\Ind \del_L \lra \mathcal {P}^*_1 \Ind  \del_{L_{|\Si_1}}\oplus \mathcal {P}^*_2 \Ind \del_{L_{|\Si_2}}
\eest
where $\mathcal {P}_i(f)=f_i$ are the restrictions \eqref{half.f} to the $i$-th component of the domain; 
in particular $\mathcal {P}_1=\mathcal{P}$. Therefore
	\begin{gather*}
	e_{U(1)}(-\mathrm{Ind} \; \del_{L})= \sum_{m=0}^{\iota}t^mc_{\iota-m}(-\mathrm{Ind} \; \del_{L})
	=\sum_{m+k+l=\iota}t^m \mathcal {P}^*_1c_k(-\mathrm{Ind} \; \del_{L_{|\Si_1}})
	\mathcal {P}^*_2c_l(-\mathrm{Ind} \; \del_{L_{|\Si_2}}).
	\end{gather*}	
where $\iota=\text{rank}_\cx( -\mathrm{Ind} \; \del_{L} )=dc_1(L)[D\Si]-\tfrac{2\chi}2$ on 
$\ov \M_{d, 2\chi}^{c,\bullet}(D\Si )_{\la^1..\la^r}$.

On the other hand, for the anti-diagonal action on $L_1\oplus L_2$, we have
	\begin{gather*}
	e_{U(1)}(-\mathrm{Ind} \; \del_{L_1\oplus L_2})=\Big(\sum_{k=0}^{\iota_1}   c_k(-\mathrm{Ind} \; \del_{L_1}) t^{\iota_1-k}\Big)\Big(\sum_{l=0}^{\iota_2} c_l(-\mathrm{Ind} \; \del_{L_2})(-t)^{\iota_2-l} \Big) = \\
	=\sum_{k+l+m=\iota_1+\iota_2} t^m c_k(-\mathrm{Ind} \; \del_{L_1})c_l(-\mathrm{Ind} \; \del_{L_2})(-1)^{\iota_2-l}.
	\end{gather*}
Here $\iota_i = \text{rank}_\cx( -\mathrm{Ind} \; \del_{L_i} )=- dc_1(L_i)[\Si]-\tfrac{\chi}{2}$ on 
$\ov \M_{d, \chi}^\bullet(\Si)_{\la^1.. \la^r}$ for $i=1,2$. 
	Note that $\iota=\iota_1+\iota_2$. 
	
Since  \eqref{double.L} implies that $L_2=c^*(\ov L_{|\Si_2})\lra \Si_1$, then as in \eqref{proj.switch.12}, we have 
	$$
	\mathcal {P}_2^*( -\mathrm{Ind} \; \del_{L_{|\Si_2}})= \mathcal{P}_1^*(-\ov{\mathrm{Ind} \; \del_{L_2}}). 
	$$
Thus
	$$
	\mathcal{P}^*_{ 2}c_l (-\mathrm{Ind} \; \del_{L_{|\Si_2}})=(-1)^l\mathcal{P}^*_1c_l (-\mathrm{Ind} \; \del_{L_2}),
	$$
	and the claim follows. 	
\end{proof}

Since $\mathcal{P}$ and $\mathcal{D}$ are inverse morphisms, combining Lemmas~\ref{double.vfc.lemma} and \ref{P.res=} gives: 
\begin{cor}\label{corR=GW.part.rel} With the notation above, the local RGW invariants 
of a doublet and the BP invariants \eqref{defn.Z.part.anti-diag-balanced} of its half are related by: 
	\bear\label{R=GW.part.rel}
	RZ_{d, 2\chi} ^\fo(D \Si \; |D(L_1,L_2))_{\la^1.. \la^r}=(-1)^{d(k_2+m_2)+\chi/2+\ell_2} Z_{d, \chi} (\Si\; |L_1, L_2)_{\la^1.. \la^r}, 
	\eear
	where $m_2, \;\ell_2$ are as in (\ref{double.vfc}), and $k_2=c_1(L_2)[\Si]$ is the degree of $L_2$. 
\end{cor}

\begin{rem}\label{remR=GW.part.rel} For a doublet target, the invariants \eqref{defn.RGW.real.rel} and the equality \eqref{R=GW.part.rel} are independent of the choice of first and second component of the target doublet. This can be seen as follows. The map $\mathcal{P}$ to the complex moduli space  \eqref{defn.doublet} is defined using the first component. Choosing the second component instead corresponds to switching the order of $L_1, L_2$ on the complex GW side. This switch results only in a  change of the sign of the equivariant complex GW invariant by the parity of  $\iota-b/2$, where $\iota$ is the complex rank of $-\Ind \del_{L_1\oplus L_2}$ and $b$ is the dimension of the moduli space, cf. \eqref{defn.Z.real.rel.nonequiv}.     The quantity $\iota-b/2 \mod 2$ is also the parity of the sum of the powers of $(-1)$ in (\ref{R=GW.part.rel}) for the two choices. 
\end{rem}

\subsection{The doublet moduli space to a connected target}\label{S.maps.to.conn} Assume next $(\Si,c)$ is a genus $g$  connected  symmetric Riemann surface with $r$ pairs of conjugate marked points, and  $\vec \la=(\lambda^1,\dots, \lambda^r)$ is a collection of $r$ partitions of $2d$. Recall that a reparametrization of a doublet domain $C$ may swap its two components. So it is convenient to consider the two fold cover of the doublet moduli space
\bear\label{map.q}
q:\wt {D\M}_{2d,h}^{c}(\Si )_{\la} \lra\ov {D\M}_{2d, h}^{c}(\Si )_{\la} 
\eear
consisting of real maps  whose domain is a doublet, up to reparametrizations {\em preserving} the    order of its components. In particular,
\bear\label{DM=1/2}
[\ov {D\M}_{2d,h }^{c}(\Si)_{\vec \la} ]^{\vir, \fo}= \tfrac 12 q_*  [\wt {D\M}_{2d,h }^{c}(\Si)_{\vec \la} ]^{\vir,\fo}.
\eear

Every real map $f: (C_1\sqcup C_2, \si) \ra (\Si,c)$ from a doublet domain restricts to a pair of maps 
\bear\label{half.domain.map}
f_i=f|_{C_i}:C_i\ra \Si \quad \mbox{ where }\quad f_2=c\circ f_1\circ \si_{|C_2}. 
\eear 
The ramification points of $f$ get distributed on the two components of the domain:  if $f$ has ramification profile~$\lambda^i$ over~$x_i^+$ (and therefore also over $x_i^-$), let $\la^i_\pm$ denote the ramification profile of its restriction $f_1$, cf.   \eqref{half.domain.map}. Since $f_2=c\circ f_1\circ \si$ then $f_2$ has ramification $\la^i_-$ over $x_i^+$ and ramification $\la^i_+$ over $x_i^-$. 

This decomposes $\la^i$ into
\best
\la^i=  \la^i_+\sqcup \la^i_-, \quad \mbox{ where } \la^i_\pm \mbox{ are partitions of $d$.} 
\eest
Note that if for example $\la^i$ has parts $4, 3, 3, 2, 1$ then $\la^i_+$, $\la^i_-$ could have parts $4, 2, 1$ and $3,3,1$ respectively. Denote such decompositions $\vec\la=\vec{\la}_+\sqcup \vec{\la}_-$ where $\vec \la_\pm= (\la^1_\pm, \dots, \la^r_\pm)$ and let 
\bear\label{doublet.decomp}
\wt{D\M}_{2d,h}^{c}(\Si )_{\vec{\la}_+| \vec {\la}_-} 
\eear
denote the corresponding relative moduli space of real maps from doublet domains, with {\em ordered} components, and so that the restriction to the first component has  ramification $\la^i_+$ over $x_i^+$ and ramification $\la^i_-$ over $x_i^-$, for 
all $i=1,  \dots r$.
Therefore 
\bear
\wt {D\M}_{2d,h}^{c}(\Si )_{\vec \la}  = \bigsqcup_{\vec\la=\vec{\la}_+\sqcup \vec{\la}_-} 
\wt {D\M}_{2d,h}^{c}(\Si )_{\vec{\la}_+| \vec {\la}_-}.  
\eear

Furthermore there is a morphism
\bear\label{cor-defn.doublet}
\mathcal{P}:\wt {D\M}_{2d,h}^{c}(\Si )_{\vec{\la}_+| \vec {\la}_-} \lra \ov \M_{d, h}(\Si)_{\vec{\la}_+, \vec {\la}_-},
 \qquad f\mapsto f_1
 \eear
cf. \eqref{half.domain.map}, where $\ov \M_{d, h}(\Si)_{\vec{\la}_+, \vec {\la}_-}$ denotes the classical  moduli space of maps from a connected domain with ramification $\la^i_+$ over $x_i^+$ and ramification $\la^i_-$ over $x_i^-$, for  $i=1,  \dots r$.

Conversely, every map $f: C\ra \Si$ from a complex curve  induces a real map 
\bear\label{double.domain}
\wt f:  (C\sqcup \ov C,\si)\ra (\Si,c), \quad \mbox{ where }\quad \wt f|_{C} = f,  \; \wt f|_{\ov C}=c\circ f\circ \si|_{\ov C}
\eear
from the double of $C$ to $\Si$. This defines the inverse $\mathcal {D}$  of \eqref{cor-defn.doublet}.
 
\begin{lemma} \label{lemcdouble.vfc} Assume $(\Si, c)$ is a connected symmetric marked curve with $r$ pairs of conjugate points and let $\mathfrak{o}=(\Theta, \psi, \mathfrak{s})$ be twisted orientation data for it. Let $\la^i_\pm$, $i=1,\dots, r,$ be $2r$ partitions of $d$.   Then, with the notation above,
\bear\label{cdouble.vfc}
[\wt {D\M}_{2d,h }^{c}(\Si)_{\vec{\la}_+| \vec {\la}_-} ]^{\vir, \fo}= (-1)^{dm+ \ell^-}  
\mathcal{D}_*[\ov \M_{d, h}(\Si)_{{\vec{\la}_+,  \vec {\la}_-}}]^\vir,  
\eear
where $m$ is the degree of $\Theta$ and $\ell^-$ is the sum of the lengths of the partitions in $\vec \la_-$: 
\bear\label{m+ell}
	m =c_1(\Theta)[\Si]=g(\Si)-1+r \quad \mbox{ and } \quad \ell^-=\sum_{i=1}^r \ell(\lambda^i_-).
\eear 	 
\end{lemma}

\begin{proof} The proof is the same as that of Lemma \ref{double.vfc.lemma}. To compare orientations, it suffices to compare them on the level of DM spaces and on the level of index bundles. Let $f:(C_1\sqcup C_2, \si) \ra (\Si, c)$ denote 
an element of $\wt {D\M}_{2d,h }^{c}(\Si)_{\vec{\la}_+| \vec {\la}_-}$. Since $f\circ \si = c\circ f$ then 
\best
f^*(\Theta \oplus c^*\ov\Theta, c_{tw})=(f^*\Theta \oplus \si^*(\ov {f^*\Theta}), \si_{tw})
\eest
where the involution $\si_{tw}$ is given by \eqref{L.plus.twL} for the bundle $L=f^*\Theta \ra (C_1\sqcup C_2, \si)$. Moreover, $C_1$ has $\ell^-$ negative points and Euler characteristic $\chi=2-2h$ thus the difference in orientations on the level of the DM spaces is $(-1)^{\chi/2+\ell^-}$ as before. On the level of index bundles, it similarly comes from the difference between the complex orientation on 
\best
\det \dbar_{ f^*(\Theta \oplus c^*\ov\Theta, c_{tw})}=\det \dbar_{(f^*\Theta \oplus \si^*f^*\ov\Theta, \si_{tw})} 
\overset{p_1}{\cong} 
\det \dbar_{(f^*\Theta \oplus \si^*f^*\ov\Theta)_{|C_1}} 
\eest
induced by \eqref{detTS=cx} and the orientation \eqref{ind.E} on 
\best
\det \dbar_{ f^*(\Theta\oplus c^*\ov\Theta, c_{tw})}\overset{\pi_1}{\cong}\det \dbar_{ (f^*\Theta)_{|C_1\sqcup C_2}}.
\eest
By Lemma  \ref{compor}, this difference is $(-1)^\iota$, where $\iota=c_1(f^*\Theta)[C_2]+\chi/2=dc_1(\Theta)[\Si]+\chi/2$. Finally, the fact that $m=c_1(\Theta)[\Si]=-\chi(T\Si)=g(\Si)-1+r$ is obtained as in Example~\ref{R.or.choice}(c), but for the relative tangent bundle $T\Si$, cf. \eqref{T.punctured}. 
\end{proof}

\begin{lemma}\label{cP.res=} Assume $L\ra \Si$ is a holomorphic line bundle over a connected symmetric surface $(\Si, c)$. Then the morphism \eqref{cor-defn.doublet} satisfies: 
	\bear\label{cres.cx=res.real}
	e_{U(1)}(-\mathrm{Ind}\; \del_{L}\lra \wt{D\M}^c_{2d, h}( \Si)_{\vec{\la}_+| \vec {\la}_-}) =	(-1)^{\iota}
	\mathcal{P}^*e_{U(1)}(-\mathrm{Ind} \; \del_{L\oplus L}\lra \ov \M_{d, h}(\Si)_{\vec{\la}_+, \vec {\la}_-}),
	\eear
	where $\iota= dc_1(L)[\Si]+1-h$ and the anti-diagonal action on $L\oplus L$ is used on the right hand side. 
\end{lemma} 

\begin{proof}Denote $\wt{D\M}= \wt{D\M}^c_{2d, h}( \Si)_{\vec{\la}_+| \vec {\la}_-}$ and 
$\ov\M=\ov \M_{d, h}(\Si)_{\vec{\la}_+, \vec {\la}_-}$. Then  $\iota = - dc_1(L)[\Si]+h-1$ is the complex rank of  
$-\mathrm{Ind} \; \del_{L}$ over  $ \ov \M$; the complex rank of $-\mathrm{Ind}\; \del_{L}$ over $\wt{D\M}$ is $2\iota$.  

As in the proof of Lemma~\ref{P.res=}, 
	\begin{gather*}
	e_{U(1)}(-\mathrm{Ind} \; \del_{L}\ra \wt{D\M}) = \sum_{m=0}^{2\iota}t^m c_{2\iota-m}(-\mathrm{Ind} \; \del_{L} \ra \wt{D\M})=
	\\
	=\sum_{k+l+m=2\iota } t^m \mathcal{P}_1^*c_k(-\mathrm{Ind} \; \del_{L}\ra \ov \M)
	\mathcal{P}_2^*c_l(-\mathrm{Ind} \; \del_{L}\ra \ov \M)
	\end{gather*}
where $\mathcal{P}_i(f)=f_i$ is the restriction to the $i$-th component of the domain, cf. \eqref{half.domain.map}.  

On the other hand, for the anti-diagonal action on $L\oplus L\ra \Si$, for the index bundle over $ \ov \M$,  
	\begin{gather*}
	e_{U(1)}(-\mathrm{Ind} \; \del_{L\oplus L})=
	\Big(\sum_{k=0}^{\iota}   c_k(-\mathrm{Ind} \; \del_{L}) t^{\iota-k}\Big)
	\Big(\sum_{l=0}^{\iota} c_l(-\mathrm{Ind} \; \del_{L})(-t)^{\iota-l} \Big) = \\
	=\sum_{k+l+m=2\iota} t^m c_k(-\mathrm{Ind} \; \del_{L})c_l(-\mathrm{Ind} \; \del_{L})(-1)^{\iota-l}.
	\end{gather*}
But as in \eqref{proj.switch.12}, we have a canonical isomorphism
 $\mathrm{Ind} \; \del_{f_2^*L}\cong \ov {\mathrm{Ind} \; \del_{\si^*f_2^*\ov L}}=
\ov {\mathrm{Ind} \; \del_{f_1^*c^*\ov L}}$ that varies continuously in $f$, and therefore 
$$
\mathcal{P}_2^*( -\mathrm{Ind} \; \del_{L})\cong \mathcal{P}_1^*( -\ov{\mathrm{Ind} \; \del_{c^*\ov L}})\cong \mathcal{P}_1^*( -\ov{\mathrm{Ind} \; \del_{ L}}).
$$
The last isomorphism follows because $c^*\ov L$ has the same degree as $L$ on a connected surface, thus can be deformed to $L$, and the Euler class is deformation invariant. Therefore
	$$
	\mathcal{P}^*_{ 2}c_l (-\mathrm{Ind} \; \del_{L})= 
	(-1)^l \mathcal{P}_1^*c_l( -\mathrm{Ind} \; \del_{L}).
	$$
 Substituting into the first displayed equation and comparing it with the second one gives \eqref{cres.cx=res.real} (recall that $\mathcal{P}=\mathcal{P}_1$). 
\end{proof}

Combining Lemmas~\ref{lemcdouble.vfc} and \ref{cP.res=} we obtain:
\begin{cor}\label{corcdefn.Z.part.anti-diag-balanced} 
 Assume $(\Si,c)$ is a connected symmetric genus $g$ surface with $r$ pairs of conjugate marked points, and $L\ra \Si$ a complex line bundle.  With the notation above, 
\bear\label{cdefn.Z.part.anti-diag-balanced}
  \ma \int_{[\wt{D\M}_{2d,h}^c(\Si)_{\vec \la_+| \vec \la_-}]^{\vir, \fo}} e_{U(1)}(-\mathrm{Ind} \; \del_{L})
 = 
 (-1)^{s^-}\ma\int_{[\ov \M_{d, h}(\Si)_{\vec \la_+, \vec \la_-}]^\vir} e_{U(1)}(-\mathrm{Ind} \; \del_{L\oplus L}), 
 \eear  
 where  $s^-=dc_1(L)[\Si]+h-1+dm+\ell^-$ and $m,\; \ell^-$ are as in \eqref{m+ell}. 
 \end{cor}
The right hand side of \eqref{cdefn.Z.part.anti-diag-balanced} corresponds to the connected GW invariants defined in  \cite{bp1}, cf. \S\ref{S.cxBP}. In particular, combining it with \eqref{DM=1/2} we obtain the following corollary. 
\begin{cor}\label{C.Z.part.anti-diag-balanced} 
When $\Si$ is a connected genus $g$ symmetric surface with $r$ pairs of conjugate marked points and $L\ra \Si$ is a complex line bundle with $c_1(L)[\Si]=k$,
 the doublet invariants and the connected GW invariants of \cite{bp1} are related via
 \best\label{cdefn.Z.part.anti-diag-balanced.conn}
	DRGW_{2d}^{c, \fo}(\Si,L)(u,t)_{\vec\la}=  \tfrac 12 (-1)^{d(k+g-1+r)}\hskip-.1in
	\sum_{\vec\la_+\sqcup\vec\la_-=\vec\la} (-1)^{\ell^-} GW^{conn}_d(g|k,k)(iu,it)_{\vec\la_+,\vec\la_-},
\eest
with $\ell^-$ as in \eqref{m+ell}. 
\end{cor}
\begin{proof} In this case $s^-=d(k+m)+h-1+\ell^-$,  $m=g-1+r$, and the substitution $(u,t)\mapsto(iu,it)$ in \eqref{shifted.GW} changes the coefficient $GW_{d, \chi}(g | k, k)$ by $(-1)^{\chi/2}$, where $\chi=2-2h$.
\end{proof}

 \vskip.2in 
 \section{Splitting formulas} \label{Gluing}
  \vskip.1in 
  
To every symmetric surface $(\Si, c)$ with $r$ pairs of conjugate marked points, every complex line bundle $L$ over $\Si$, and every choice of twisted orientation data $\fo$ on $(\Si, c)$,  \eqref{defn.RGW.real.rel} associates a collection of invariants 
\best
RGW_d^{c,\mathfrak{o}}(\Si, L)_{\mu^1\dots \mu^r}= \sum_{\chi} u^{b/2+k}
 \int_{[\ov \M_{d, \chi}^{c, \bullet}(\Si)_{\mu^1\dots  \mu^r}]^{\vir,\mathfrak{o}}} e_{U(1)}(-\mathrm{Ind} \; \del_{L}), 
\eest
where $\mu^1, \dots , \mu^r$ are partitions of $d$. 
These are invariant not only under smooth deformations of the data $(\Si, c, L, \fo)$, but also under deformations as the symmetric curve $\Si$ pinches to acquire a pair of conjugate nodes as follows. 

\medskip
Recall that if $\Si_0$ is a nodal curve, then it has (a) a smooth resolution (normalization) $\wt \Si$  that replaces each node by a pair of marked points and (b) a family of deformations smoothing out the nodes. 

This extends to symmetric surfaces as in \cite[\S4.2]{GZ2}. More precisely, assume $(\Si_0, c_0)$ is a nodal symmetric surface with a pair of conjugate nodes and $r$ pairs of conjugate marked points. It has a normalization $(\wt \Si, \wt c)$ which has $r+2$ pairs of conjugate marked points. Similarly, $(\Si_0, c_0)$ has a family of smooth deformations, simultaneously smoothing out the conjugate nodes using complex conjugate gluing parameters. The generic fiber $(\Si, c)$ of the family is a symmetric surface with $r$ pairs of conjugate marked points, and a pair of `splitting circles'  (disjoint vanishing cycles) swapped by the involution; as the gluing parameters converge to 0, these circles pinch to produce the two complex conjugate nodes of $\Si_0$. 
\smallskip

A complex line bundle over the nodal curve extends to a line bundle over the family of deformations and lifts to a line bundle on the normalization. The relative tangent bundle to the family of marked curves restricts to the   tangent bundle  \eqref{T.punctured} of each fiber and gives rise to the   tangent bundle of the normalization (regarded as a marked curve).  Finally, a choice of orientation data  as in Definition~\ref{TRO} on the nodal curve extends to orientation data over the family and lifts to orientation data on the normalization.

\begin{figure}[h!]
\begin{centering}
\begin{tikzpicture}[baseline={(current bounding box.center)}, scale=1/2,  rotate=90]
\begin{scope}
\pic[
  transform shape,
  name=a,
  tqft,
  cobordism edge/.style={draw},
  between incoming and outgoing/.style={draw},
  incoming boundary components=0,
  outgoing boundary components=2,
];
\pic[
  transform shape,
  tqft,
  incoming boundary components=2,
  outgoing boundary components=2,
  cobordism edge/.style={draw},
  genus=2,
  name=b,
  at=(a-outgoing boundary),
  anchor=incoming boundary,
];
\pic [ transform shape, tqft, 
	incoming boundary components=1, outgoing boundary components=1, 
	every outgoing upper boundary component/.style={draw,blue, thick},
  	every outgoing lower boundary component/.style={draw, dotted,blue, thick},
	at={(0,-4)}, cobordism height=.5, cobordism edge/.style={draw}];  
\pic [ transform shape, tqft, 
  	incoming boundary components=1, outgoing boundary components=1, 
	at={(0,-4.5)}, cobordism height=.5, cobordism edge/.style={draw}];  
\pic [ transform shape, tqft, 
	incoming boundary components=1, outgoing boundary components=1, 
	every outgoing upper boundary component/.style={draw,blue,  thick},
  	every outgoing lower boundary component/.style={draw, dotted, blue,  thick},
	at={(2,-4)}, cobordism height=.5, cobordism edge/.style={draw}];  
\pic [ transform shape, tqft, 
	incoming boundary components=1, outgoing boundary components=1, 
	at={(2,-4.5)}, cobordism height=.5, cobordism edge/.style={draw}]; 

\node at (-1,-4.5) {$\textcolor{blue}{\gamma^-}$};
\node at (3,-4.5) {$\textcolor{blue}{\gamma^+}$};
\pic[
  transform shape,
  tqft,
  cobordism edge/.style={draw},
  incoming boundary components=2,
  outgoing boundary components=2,
  genus=1,
   at={(0,-5)}, 
  name=c,
  anchor=incoming boundary,
];
\pic[
  transform shape,
  name=a,
  tqft,
  cobordism edge/.style={draw},
  incoming boundary components=2,
  outgoing boundary components=0,
   at={(0,-7)}, 
];
		\draw[<->, rounded corners=10pt,  solid] (0,0)--(1, .5)--(2,0);
		\node at (1,1) {$c$};
		\node at (1,-9) {$\Sigma$};
\end{scope}		
\begin{scope}[rotate=0, shift={(0,-14)}]
\pic[
  transform shape,
  name=a,
  tqft,
  cobordism edge/.style={draw},
  between incoming and outgoing/.style={draw},
  incoming boundary components=0,
  outgoing boundary components=2,
   at={(0,0)},
];
\pic[
  transform shape,
  tqft,
  incoming boundary components=2,
  outgoing boundary components=2,
  cobordism edge/.style={draw},
  edge between incoming 1 and 2/.style={thick},
  genus=2,
  name=b,
  at=(a-outgoing boundary),
];
\pic [ transform shape, tqft/cup, at={(0,-4)}, cobordism edge/.style={draw}];  
\pic [ transform shape, tqft/cup, at={(2,-4)}, cobordism edge/.style={draw}];  
\node at (-1,-4.5) {$\textcolor{blue}{x^-_1}$};
\node at (3,-4.5) {$\textcolor{blue}{x^+_1}$};
\node at (-1,-6.5) {$\textcolor{blue}{x^-_2}$};
\node at (3,-6.5) {$\textcolor{blue}{x^+_2}$};
\node at (0,-4.5) {$\textcolor{blue}{\bullet}$};
\node at (2,-4.5) {$\textcolor{blue}{\bullet}$};
\node at (0,-6.5) {$\textcolor{blue}{\bullet}$};
\node at (2,-6.5) {$\textcolor{blue}{\bullet}$};
\pic [ transform shape, tqft/cap, at={(0,-5)}, cobordism edge/.style={draw}];  
\pic [ transform shape, tqft/cap, at={(2,-5)}, cobordism edge/.style={draw}];  

\pic[
  transform shape,
  tqft,
  cobordism edge/.style={draw},
  incoming boundary components=2,
  outgoing boundary components=2,
  genus=1,
  at={(0,-7)}, ];
\pic[
  transform shape,
  name=a,
  tqft,
  cobordism edge/.style={draw},
  incoming boundary components=2,
  outgoing boundary components=0,
  at={(0,-9)},
];
\draw[<->, rounded corners=10pt,  solid] (0,0)--(1, .5)--(2,0);
		\node at (1,1) {$\widetilde c$};
		\node at (1,-11) {$\widetilde \Sigma$};

\end{scope}
\begin{scope}[rotate=0, shift={(-6,-7)}]
\pic[
  transform shape,
  name=a,
  tqft,
  cobordism edge/.style={draw},
  between incoming and outgoing/.style={draw},
  incoming boundary components=0,
  outgoing boundary components=2,
   at={(0,0)},
];
\pic[
  transform shape,
  tqft,
  incoming boundary components=2,
  outgoing boundary components=2, 
  cobordism edge/.style={draw},
  genus=2,
  name=b,
  at=(a-outgoing boundary)
];
\pic [ transform shape, tqft/cup, at={(0,-4)}, cobordism edge/.style={draw}];  
\pic [ transform shape, tqft/cup, at={(2,-4)}, cobordism edge/.style={draw}];  
\pic [ transform shape, tqft/cap, at={(0,-3)}, cobordism edge/.style={draw}];  
\pic [ transform shape, tqft/cap, at={(2,-3)}, cobordism edge/.style={draw}];  
\node at (-1,-4.5) {$\textcolor{blue}{x^-_1=x_2^-}$};
\node at (3,-4.5) {$\textcolor{blue}{x^+_1=x_2^+}$};
\node at (0,-4.5) {$\textcolor{blue}{\bullet}$};
\node at (2,-4.5) {$\textcolor{blue}{\bullet}$};
\pic[
  transform shape,
  tqft,
  cobordism edge/.style={draw},
   incoming boundary components=2,
  outgoing boundary components=2,
  genus=1,
  at={(0,-5)}, ];
\pic[
  transform shape,
  name=a,
  tqft,
  cobordism edge/.style={draw},
  incoming boundary components=2,
  outgoing boundary components=0,
  at={(0,-7)},
];
		\draw[<->, rounded corners=10pt,  solid] (0,0)--(1, .5)--(2,0);
		\node at (1,1) {$c_0$};
		\node at (1,-9) {$\Sigma_0$};
\end{scope}		
\draw[->, solid] (-1.5,-5)--(-2.5,-7);
\draw[<-, densely  dashed] (-1.5,-6)--(-2.5,-8);
\draw[->, solid] (-2,-18.5)--(-3,-16.5);
\draw[<-, densely  dashed] (-2,-17.5)--(-3,-15.5);
\node at (-2.5,-5) {pinch};
\node at (-1.5,-9) {deform};
\node at (-2,-15) {resolve};
\node at (-3,-19) {attach};
\draw[->, rounded corners=20pt, densely dashed] (2, -9)--(2.5, -11)--(2,-13);
\node at (3,-11) {split};
\end{tikzpicture}
\end{centering} 
\end{figure}

\medskip
Furthermore, assume  $(\Si, c)$ is a marked symmetric surface with a pair of {\em conjugate splitting circles}, i.e. two embedded, disjoint circles 
$\gamma^\pm$ swapped by the involution and containing no marked points. Then $(\Si, c)$ can be 'split' along these circles, i.e. it can be deformed to a nodal symmetric surface $(\Si_0, c_0)$ which then has a smooth normalization $(\wt \Si, \wt c)$. Any complex line bundle $L$ over $\Si$ and choice $\fo$ of twisted orientation data for $(\Si, c)$ can  be
deformed to the nodal surface and then lifted to its normalization to give a line bundle $\wt L$ over $\wt \Si$ and a choice of orientation data $\wt \fo$ on the normalization $(\wt \Si, \wt c)$. Lastly, every line bundle $\wt L$ and orientation data $\wt \fo$ on 
$\wt \Si$   descend to $\Si_0$ and can be deformed to a line bundle $L$ and orientation data $\fo$ on $\Si$.

\smallskip

The splitting formula \cite[Theorem~3.2]{bp1}  extends to the Real setting (cf. \cite{GI}) as follows:
 \begin{theorem}[\cite{GI}]\label{gluing} Assume $(\Si, c)$ is a marked symmetric surface with $r$ pairs of conjugate points, $L$ is a complex line bundle over $\Si$, and $\fo$ is an orientation data for $(\Si,c)$. Let $(\wt \Si, \wt c)$ denote the symmetric surface obtained as described above from $(\Si, c)$ by splitting it along two conjugate splitting circles, and let $\wt L$ and $\wt \fo$ be the corresponding line bundle and orientation data on $\wt \Si$. 
 
Then for any collection $\vec \mu= (\mu^1, \dots , \mu^r)$ of $r$ partitions of $d$, the RGW invariants 
\eqref{defn.RGW.real.rel} satisfy: 
 \bear\label{split.formula}
RGW^{c, \fo}_d(\Si, L)_{\vec \mu} 
= 
\sum_{\la\vdash d} \zeta(\la) t^{2 \ell(\la)} RGW^{\wt c, \wt \fo}_d(\wt \Si, \wt L)_{\vec \mu, \la, \la},
\eear
where $\zeta(\la)$ is given by \eqref{mult.la}, $t$ is the equivariant parameter, and $\ell(\la)$ is the length of the partition $\la$. 
\end{theorem}
The basic idea of the proof comes from considering the family of moduli spaces of maps with values in $\Si$, as $\Si$ deforms to become a nodal curve, cf. \cite[Appendix A]{bp-TQFT}. When regarded as maps into the total space of the family of deformations of $\Si$, maps with values in $\Si$ limit to maps $f_0$ with values in $\Si_0$ that lift to maps with values in $\wt \Si$ having matching ramification pattern $\la$ over the nodes of $\Si_0$. Since we are splitting along a pair of conjugate nodes, the local analysis of this deformation is the same as in the complex case, and the only difference is that the gluing at one of the nodes determines the gluing at the conjugate node. As in the proof of \cite[Theorem~3.2]{bp1}, the multiplicity  $\zeta(\la)$ comes from the number of ways such a map $f_0$ deforms to a map with values in $\Si$, and  $t^{2 \ell(\la)} $ comes from the difference in the Euler class of the index bundles (the index bundles differ by a trivial rank $2\ell(\la)$ bundle obtained by pulling back over the nodes of the domain the restriction of $L$ to the nodes of the target). 
The comparison of the orientations is similar to that of \cite[Theorem~1.2]{GZ2}, except it uses the twisted orientation instead of the real orientation of \cite{gz}.

 \medskip

Define the raising of the indices by the formula
\begin{equation}\label{modified.metric}
RGW^{c,\mathfrak{o}}(\Si, L)_{\mu^1\dots \mu^r}^{\nu^1\dots \nu^s}=
RGW^{c,\mathfrak{o}}(\Si, L)_{\mu^1\dots \mu^r,\nu^1\dots \nu^s}\left(\prod_{i=1}^s\zeta(\nu^i)t^{2\ell(\nu^i)}\right).
\end{equation}
With this convention, \eqref{split.formula} implies that for any splitting $(\wt \Si, \wt c, \wt L, \wt \fo)$ of $(\Si, c, L, \fo)$ 
along a pair of conjugate splitting circles, 
\bear\label{split.formula.2}
RGW^{c, \fo}_d(\Si, L)_{\mu^1\dots\mu^r}^{\nu^1\dots\nu^s}= \sum_{\la\vdash d} 
RGW^{\wt c, \wt \fo}_d(\wt \Si, \wt L)_{\mu^1\dots \mu^r, \la} ^ {\nu^1\dots\nu^s, \la}. 
\eear
In particular, for a splitting $(\wt\Si, \wt c)$ of $(\Si, c)$ along a pair of non-separating conjugated circles, 
 \bear\label{split.nonseparate}
 RGW^{c,\mathfrak{o}}(\Si, L)_{\mu^1\dots\mu^r} = \sum_{\lambda\vdash d} 
 RGW^{\wt c,\wt{\mathfrak{o}}}(\wt\Si, L)_{\mu^1\dots \mu^r\lambda}^\lambda,
 \eear
while for a splitting along a pair of separating conjugated circles into $(\Si',c')$ and  $(\Si'',c'')$ we have 
\bear\label{split.separate}
  RGW^{c,\mathfrak{o}}(\Si,L)_{\mu^1\dots\mu^r}^{\nu^1\dots\nu^s}=\sum_{\lambda \vdash d } RGW^{c',\mathfrak{o'}}(\Si',L')^\lambda_{\mu^1\dots\mu^r}RGW^{c'',\mathfrak{o''}}(\Si'',L'')_\lambda^{\nu^1\dots\nu^s}
  \eear
where 
$L', L''$ and $\fo', \fo''$ denote the restrictions of $\wt L$ and $\wt \fo$ to $\Si'$ and $\Si''$ respectively.

This will allow us to construct a Klein TQFT associated to these invariants in \S\ref{KTQFTforRGW}.

 \vskip.2in 
 \section{The level 0 theory} \label{level 0}
  \vskip.1in 
  
The main result in this section is a calculation of the level 0 theory for a symmetric sphere relative a pair of conjugate points, cf. Proposition~\ref{crosscap}. We  start with the following preliminary result, for the level 0 theory, i.e. corresponding to the case when the line bundle $L$ in  \eqref{L.plus.twL} is trivial.
\begin{lemma}\label{L.level.0}
	The level 0 RGW series \eqref{defn.RGW.real.rel} have no nonzero terms of positive degree in $u$. 
	\end{lemma}

\begin{proof} The level 0 RGW series are built from the following integrals:
\best\nonumber
	RZ^{c,\mathfrak{o}}_{d, \chi} (\Si, \mathcal{O})_{\la^1\dots \la^r}&=& 
	t^{\iota-b/2}\int_{[\ov \M_{d, \chi}^{c, \bullet}(\Si)_{\la^1\dots \la^r}]^{\vir,\mathfrak{o}}}
	 c_{b/2}(-\mathrm{Ind} \; \del_{\mathcal{O}})
	\\ \label{rz.int}
	&=&
	t^{\iota-b/2}\int_{[\ov \M_{d, \chi}^{c, \bullet}(\Si)_{\la^1\dots \la^r}]^{\vir,\mathfrak{o}}} c_{b/2}(\mathbb{E}^\vee),
\eest
where $\mathbb{E}^\vee$ denotes the dual of the Hodge bundle, and $b$ is the dimension of the moduli space \eqref{dim.M=b}. Since the power of $u$ in the level 0 RGW invariants \eqref{defn.RGW.real.rel} is $b/2$, it suffices to show that the only nonzero contribution to $RZ^{c,\mathfrak{o}}_{d, \chi} (\Si, \mathcal{O})_{\la^1.. \la^r}$ comes from 0-dimensional moduli spaces. It suffices to show this is the case for the doublet and connected invariants of $(\Si, c)$,  when $\Si$ is itself either a doublet or connected.

By Corollaries \ref{corR=GW.part.rel}, \ref{C.Z.part.anti-diag-balanced}, the doublet invariants for a connected or a doublet target are equal up to a scalar to the connected BP invariants. By the proof of \cite[Lemma~7.5]{bp1}, for the antidiagonal action and level $(0, 0)$, the connected BP invariants vanish unless the dimension of the moduli space is 0. 

So it remains to consider the case when both the domain and target are connected. Let $\ov{\R\M}_{g,\ell}$ and $\ov\M_{g,2\ell}$ denote the real and the complex Deligne-Mumford moduli spaces of connected genus $g$ Riemann surfaces with $\ell$ pairs of conjugate and $2\ell$ marked points, correspondingly. Consider the map
\bear\label{forget.real}
\ov{\R\M}_{g,\ell}\lra \ov\M_{g,2\ell} 
\eear	
forgetting the real structure on the curve. The image of this map falls into the fixed locus of the involution on $\ov\M_{g,2\ell}$ given by 
$$
[ S, j,y_1,\dots, y_{2\ell}]\mapsto [S, -j, y_2,y_1,\dots y_{2\ell},y_{2\ell-1}].
$$
In general, the map \eqref{forget.real} is neither injective nor surjective onto the fixed locus. However, the Hodge bundle $\mathbb{E}$ over the real Deligne-Mumford space is the pull-back via \eqref{forget.real} of the Hodge bundle over the complex space. 
Over the real Deligne-Mumford space, the real structure $\si$ on a Riemann surface representing a point in the space induces a complex conjugation on the fiber of the Hodge bundle over it. Therefore the Hodge bundle  splits into invariant and anti-invariant parts of equal dimensions i.e. 
$$
\mathbb{E}\cong \mathbb{E}^\R\otimes_\R \cx\lra \ov{\R\M}_{g,\ell}.
$$ 
This implies that 
$$
c_{2k+1}(\mathbb{E})=0\in H^{4k+2}(\ov{\R\M}_{g,\ell},\Q).
$$
By Mumford's relations 
$$
0=c_i(\mathbb{E}\otimes_\R\cx)=\sum_{j=0}^i (-1)^jc_{i-j}(\mathbb{E})c_j(\mathbb{E}). 
$$
In particular, for even index, $2 c_{2k}(\mathbb{E})+ \ma \sum_{j=1}^{2k-1} (-1)^jc_{2k-j}(\mathbb{E})c_j(\mathbb{E})=0$. By induction on $k$, using the vanishing of the odd classes over the real moduli space we get
$$
c_{i}(\mathbb{E})=0\in H^{2i}(\ov{\R\M}_{g,\ell},\Q) \quad  \mbox{ for all $i\ne 0$.}
$$
Thus the only nonzero contributions to $RZ^{c,\mathfrak{o}}_{d, \chi} (\Si, \mathcal{O})_{\la^1.. \la^r}$ can come from integrating $1$ over a 0-dimensional moduli space.
\end{proof}

\subsection{Level 0 theory for a sphere relative two points} \label{S.0sphere}   Consider next $(\Si, c)$ a real sphere with a pair of conjugate points $x^\pm$. Up to reparametrization, there are only two real structures on $\Si=(\P^1, x^\pm)$:  
 \best
c_-(w)= -1/\ov w \quad  \mbox{ and } \quad \quad c_+ (w)= 1/\ov w.
  \eest
The real locus $\Si^c$ is empty for the first one and non-empty for the second one. 

For the remainder of this section, we regard $\P^1$ as $\cx\cup \infty$ with coordinate $w$, such that the preferred point $x^+$ corresponds to $w=0$ and 
$x^-$ to $w=\infty$. Let  $S^1$ be the unit circle $|w|=1$, which separates $\P^1$, and is oriented as the boundary of the component containing $x^+=0$. Then $S^1$ is the fixed locus when $c= c_+$ and is a cross-cap when $c=c_-$ (i.e. $c_-(w)=-w$ for all $w\in S^1$).

 The relative tangent bundle $T\Si$, given by \eqref{T.punctured}, is trivial for $\Si=(\P^1, x^\pm)$.
Therefore a twisted orientation data $\fo=(\Theta, \psi, s)$ for $(\P^1, x^\pm, c)$ consists of a trivial complex line bundle 
$\Theta=\Si\times \cx$ over $\Si=(\P^1, x^\pm)$, a choice of a homotopy class of Real isomorphisms \eqref{real.isom.det}, i.e. 
\bear\label{isom.psi}
\psi: \Lambda^{\text{top}}(T\Si\oplus \Theta\oplus c^*\ov\Theta, dc\oplus c_{tw}) \ma\lra^{\cong} 
	(\Si\ti \cx, c\ti c_{std}), 
\eear
and a spin structure $\mathfrak s$ on $T\Si^c\oplus \cx_{|\Si^c}$ over the real part of the target, compatible with the orientation induced by the isomorphism \eqref{isom.psi}.

Note also that, up to homotopy, there is a unique trivialization
\bear\label{choice.phi}
\phi:(T\Si,dc)\cong (\Si \times \cx, c\times c_{std})  \quad \mbox{such that it restricts to}
\eear
\bear\label{choice.phi.restr}
(T\Si,d c)_{|S^1}=(TS^1\oplus JTS^1,d c)=(S^1\times(\R\oplus j\R), c\times c_{std}). 
\eear 
This is because  there are two classes of trivializations \eqref{choice.phi} and they are distinguished by their restriction to  $S^1$, cf. \cite[Lemma 2.4]{F}; we choose the one that satisfies \eqref{choice.phi.restr}.

 Finally, to each partition $\la=(1^{m_1}2^{m_2}3^{m_3}\dots)$, associate the monomial
\bear\label{p-la}
p_\la = \prod_{k=1}^\infty p_k^{m_k}.
\eear
 With this notation, our main result in this section is:
 \begin{prop}\label{crosscap} Consider a Real sphere $\Si=(\P^1, x^\pm)$ with a pair of marked points and real structure  $c$. Let $\fo$ be an orientation data for $(\Si,c)$. Then for any partition $\la$ of $d$, 
 \bear\label{level0caps}
 	RGW^{c,\mathfrak{o}}_{d} (0|0)_{\la}=\exp\left(\ep_\mathfrak{o}\sum_{k=0}^\infty \frac{p_{2k+1}}{(2k+1)t}-\sum_{m=1}^\infty\frac{p_{m}^2}{2mt^2}\right)_{[p_\la]}, 
\eear
	where $\ep_\fo=\pm 1$ is independent of $d$. Here $[p_\lambda]$ denotes the coefficient of the monomial $p_\la$.  
	
Moreover, for each $\ep=\pm 1$ there exists a choice of a twisted orientation data $\fo$ such that $\ep_\fo=\ep$.
\end{prop}

 \begin{proof} It suffices to calculate the connected and doublet invariants; then \eqref{rgw=exp} extends to give the RGW invariants. By Corollary \ref{C.Z.part.anti-diag-balanced}, the doublet invariants $DRGW$ are related to the 
 BP invariants $GW^{conn}$ counting {\em connected} curves. The latter were computed in \cite[Lemma~6.1]{bp1} giving: 
 $$
 DRGW^{c}_{2d} (0|0)_{\la}(u,t) = -\tfrac{1}{2} GW^{conn}_{d} (0| 0,0)_{\la_+,\la_-}(iu,it)=\tfrac{-1}{2d(-t)^2}\, , \quad \mbox{ for } \la_+=\la_-=(d)
 $$   and vanish otherwise.
 
 \smallskip
 By Lemma~\ref{L.level.0}, the only contribution to the connected real invariant $CRGW$ comes from 0 dimensional moduli spaces. The dimension of $\ov \M^c_{d,h}(\P^1)_\lambda$ is 
 \best
 b=2d +2h-2 -2d +2\ell(\la)= 2h-2+ 2\ell(\la). 
 \eest 
It  vanishes only when $h=0$ and $\ell (\la)=1$ i.e. $\la=(d)$. It suffices to show that in this case 
\bear\label{triv.count}
\int_{[ \ov \M^{c, \fo}_{d,0}(\P^1)_{\la}]^\vir}1 = \begin{cases}\ep_\fo \frac {1} {d},&\text{if $d$ is odd,}\\
	0,&\text{if $d$ is even.} 
\end{cases}
\eear

Elements of $\ov \M^c_{d,0}(\P^1)_{\la}$ for $\la=(d)$ are real covers of a sphere by a sphere, fully ramified at the two points $x^\pm$, and equivariant with respect to a real structure $\si$ on the domain and $c$ on the target. 
 
\smallskip

{\sc Case 1}. Assume first that  $c(w)=-1/\ov w$, so it has no fixed locus. Then $\si$ cannot have fixed locus, and $d$ must be odd (else the moduli space is empty). When $d$ is odd, the moduli space consists of one solution 
$f(z)=z^d$, $\si(z)=-1/\ov z$, but which has $d$ automorphisms $\phi(z)=\zeta z$ where $\zeta^d=1$. It remains to calculate its sign and show it does not depend on $d$. We will first prove that there are two classes of twisted orientation data, giving  rise to opposite invariants, and then we calculate the invariants for a canonical choice $\fo=\fo_{can}$  that   corresponds to $\ep_\fo=1$.
 
A twisted  orientation data $\fo=(\Theta, \psi, \mathfrak{s})$ in this case consists of a choice of an isomorphism \eqref{isom.psi} up to homotopy; the bundle $\Theta=\Si\ti \cx$ is trivial and the real locus of $c$ is empty so the spin structure $\mathfrak {s}$ is irrelevant. 

There are two real homotopy classes of isomorphisms  \eqref{isom.psi} distinguished by the real homotopy class of $\psi$ over the unit circle $|w|=1$ in $\P^1=\cx\cup\infty$. One can switch between them by $\psi\mapsto -\psi$. The effect of this change on the orientation of the moduli space is via the change of the orientation on the bundle 
$\ind\dbar_{(\cx,c_{std})}$, which is $(-1)^{\chi/2}=-1$ since the domains are spheres. In particular, if $\fo_1$ and $\fo_2$ denote the two choices of twisted orientation, then the level 0 connected invariants satisfy 
\bear\label{diff.orient}
CRGW^{c, \fo_1}_{d}(0|0)_{\la}= - CRGW^{c, \fo_2}_{d}(0|0)_{\la}.
\eear

We next determine the sign of the invariants in each degree by looking at the moduli space in more detail. The orientation on the moduli space is induced from the determinant bundle $\det \del_{(T\Si, dc)}$ and the Deligne-Mumford moduli space,  cf. \eqref{or.moduli.rel.a}, after stabilization when necessary. So we add an extra pair of conjugate marked points $y_2^\pm$ on the domain. The moduli space is now 2 dimensional and it suffices to calculate the sign of the evaluation map at $y_2^+$. For this we first exhibit an orientation on the moduli space for which the sign of the evaluation map is clear and then we compare it with orientation induced by the twisted orientation data. 

The real DM moduli space $\R\M_{0,  2}$ is 1-dimensional and consists of 3 intervals that compactify to a circle; one of the intervals corresponds to the case the involution on the domain is fixed-point free and the other two to the case the involution has fixed locus.  We can assume that $\si(z)=\pm 1/\ov z$,  $y_1^\pm$ are $z=0, \infty$, and $y_2^+=b \in \R_+$. The orientation on 
$\R\M_{0,  2}$ agrees with the one induced by $b\in \R_+$ when $\si$ has fixed locus, and is the opposite in the case $\si$ is fixed-point free; see  \cite[\S1.4]{GZ2}. 

When the domain is fixed, the moduli space of degree $d$ real relative maps is 
\best
 f_\tau:(\P^1,\si)\lra (\P^1,c) \quad  z\mapsto e^{i\tau}z^d,\quad \tau\in \R/2\pi\Z; 
 \eest
here $\si(z)=-1/\ov z$. Thus the relative moduli space with the extra pair $y^\pm_2$ of marked points is described by 
$(\tau,b)\in \R\times\R_+$, where $b$ corresponds to the position of $y_2^+$ and $\tau$ gives the map $f_\tau$. For the orientation induced by this identification, the evaluation map at $y_2^+$ is orientation reversing. The tangent space to the first factor corresponds naturally to $\Ind \dbar_{(T\Si, dc)}$ and the tangent space to the second factor to $T\R\M_{0,  2}$. Recall that the canonical orientation on the latter is opposite that of  $b\in \R_+$ when the domain involution is fixed-point free. Thus the evaluation map at $y_2^+$ would have positive sign if the orientation induced by a twisted orientation on $\Ind \dbar_{(T\Si, dc)}$ coincides with that induced by $\tau\in\R$. We next construct such twisted orientation data. 

Let  $\fo_{can}$ be the twisted orientation data for which \eqref{isom.psi} has the form $\psi=\phi \otimes \La^{top}\theta_{tw}$, where $\phi$ is given by \eqref{choice.phi} and 
\best\label{isom.can} 
\theta_{tw}: (\Si\times \cx\oplus c^*(\Si\times \ov\cx), c_{tw})\cong (\Si\times \cx^{\oplus 2}, c\times c_{std})
\eest
is orientation preserving at the level of index bundles when the first term has the complex orientation induced via \eqref{ind.E} and the second term is oriented as twice a bundle. By   Lemma~\ref{L.tw=2bd} below, we can obtain such  $\theta_{tw}$ as the composition of \eqref{isom.tw=double} and \eqref{switch.signs}. For this choice, the twisted orientation data  
$\psi=\phi \otimes \La^{top}\theta_{tw}$ induces precisely the isomorphism \eqref{choice.phi.restr}, as explained above \eqref{detTS=twist}. On the other hand, the isomorphism \eqref{choice.phi.restr} induces an orientation on $\Ind \dbar_{(T\Si, dc)}$ that coincides with that of $\tau\in\R$. Therefore, for this choice of twisted orientation data,  the evaluation map has positive degree for all odd $d$, completing the proof of \eqref{triv.count}. 

\medskip
{\sc Case 2}. Assume $c(w)=1/\ov w$, so the involution on the target has fixed locus. The argument in this case follows along the same lines. The fixed locus $\Si^c$ is now the unit circle $S^1$ and a twisted orientation data requires a choice $\mathfrak s$ of a spin structure on $T\Si^c\oplus \cx_{|\Si^c}$ over the real part of the target, compatible with the orientation induced by the isomorphism \eqref{isom.psi}. There is still one solution for $d$ odd (with $\si(z)=1/\ov z$ on the domain), but when $d$ is even, there are now two solutions, with different real structures.  

We next construct a twisted orientation data $\fo=\fo_{can}$ for which $\ep_\fo=+1$. Let $\psi=\phi\otimes \La^{top} \theta$, where  $\phi$ is as in 
\eqref{choice.phi} and $\theta$ is the isomorphism \eqref{isom.tw=double}.  The isomorphism  $\theta$, along with the orientation of $TS^1$, induces a   spin structure $\mathfrak{s}$, compatible with $\psi$. Denote these choices by $\fo_{can}$.   
  
We repeat the same argument as in Case 1, taking into account that the orientation on the real DM moduli space is given by $b\in\R_+$ when $\si$ has real locus, and by $-b$ when $\si$ does not have real locus. Recall that in odd degree $\si$ must have  real locus, while in even degree there are two solutions, one with real locus and one without. 
By Lemma~\ref{L.tw=2bd} below, at the level of index bundles, the isomorphism $\theta$ has sign $(-1)^{\ind \del_\cx}=(-1)^{\chi/2}=-1$ and thus the orientation induced on $\Ind \dbar_{(T\Si, dc)}$ is opposite of that induced by $\tau\in \R$. So all maps whose domain involution has fixed locus contribute positively and all maps with fixed-point free domain involution contribute negatively. This implies that the maps in even degree cancel each other. In odd degree, the domains can only have real structure with fixed locus and thus contribute positively. This implies \eqref{triv.count} for $\fo= \fo_{can}$  (with $\ep_\fo=1$).

It remains to calculate how the invariants depend on the orientation data $\fo=(\Si \otimes \cx, \psi, \mathfrak s)$. Up to homotopy, there are 4 choices, two for $\psi$ and two for the spin structure $\mathfrak s$. As before, a change in the homotopy class of $\psi$ changes the orientation on all maps thus giving \eqref{diff.orient}. A change in the spin structure results in a change of $(-1)^d$ on the orientation of a degree $d$ map as it changes the pullback spin structure on the domain only if the degree is odd, cf. \cite[Corollary 5.7]{gz} and Lemma \ref{A.comp}. Since the even degree invariants vanish, changing the spin structure $\mathfrak s$ also gives \eqref{diff.orient}, completing the proof of \eqref{triv.count}.
 \end{proof}
 
 When $(L, \phi)\ra (\Si, c)$ is a Real bundle over a symmetric surface, then 
\bear\label{isom.tw=double}
\theta:(L\oplus c^*\ov L, c_{tw})\cong (L \oplus L, \phi\oplus \phi), \qquad (z; x,y)\mapsto (z;  x+\phi(y), -Jx+J\phi(y))
\eear
is a Real isomorphism. The index of the LHS has a natural  complex  orientation while   that of the RHS can be oriented as twice of a bundle. The next lemma compares these two orientations.  
\begin{lemma} \label{L.tw=2bd} Assume $(L, \phi)\ra (\Si, c)$ is a Real bundle. Then the index bundle 
$\Ind \del_{(L\oplus c^*\ov L, c_{tw})}$ has two natural orientations: 
\begin{enumerate}[(i)]
\item one induced by the isomorphism \eqref{ind.E} with $\Ind \del_L$ via the projection onto the first bundle.
\item the second one induced by \eqref{isom.tw=double} and the natural orientation on twice a bundle. 
\end{enumerate} 
The difference  between these orientations is  $(-1)^\iota$ where $\iota$ is the complex rank of $  \Ind \del_L$. Moreover, 
\bear\label{switch.signs} 
Id\oplus -Id:\Ind \del_{(L\oplus L, \phi\oplus \phi)}\ra\Ind \del_{(L\oplus L, \phi\oplus \phi)},
\eear
when both sides are oriented as twice a bundle, also has sign $(-1)^\iota$. 
\end{lemma} 
\begin{proof} The isomorphism \eqref{ind.E} with $\Ind \del_L$ induces a complex structure and therefore a complex orientation on the index bundle associated to the  left hand side of \eqref{isom.tw=double}. The isomorphism induced by $\theta$ at the level of index bundles would be orientation preserving if the complex orientation on twice of a bundle was used instead on the  right hand side; the two choices differ by $(-1)^\iota$. The second statement is immediate. 
\end{proof}

The proof of Proposition~\ref{crosscap}  constructs choices of orientation data $\fo=\fo_{can}$  that have the property that $\ep_\fo=1$; in particular, the sign of the degree 1 cover is +1. For such choices,  \eqref{level0caps} is equal to 
\bear\label{level0caps.can}
 RGW_{d} (0|0)_{\la}=\exp\left(\sum_{k=0}\frac{p_{2k+1}}{(2k+1)t}-\sum_{m=1}\frac{p_{m}^2}{2mt^2}\right)_{[p_\la]}, 
\eear
while for any other choice of orientation data 
\best\label{level0caps.other}
 RGW^{c,\mathfrak{o}}_{d} (0|0)_{\la}= (\ep_\fo)^d RGW_{d} (0|0)_{\la}
\eest
where $\ep_\fo=\pm 1$ is the sign of the degree 1 cover. This follows because substituting $p_m\mapsto \ep_\fo p_m$ for all $m=1, 2, \dots$ converts the sum in the exponential of \eqref{level0caps.can} to the one in \eqref{level0caps}, but also changes 
$p_\la \mapsto (\ep_\fo)^d p_\la$ when $\la$ is a partition of $d$. 
	
 \vskip.2in 
\section{Canonical orientation and independence of the target real structure} \label{S.Can.TO}
  \vskip.1in 

In this section we study how the RGW invariants depend on the choice of orientation data and on the real structure on the target. We show that a change in the orientation data or in the real structure results in a global change by a factor of $(\pm 1)^d$. 
We then use this information to define canonical  RGW invariants which are compatible with the splitting formulas. 

\subsection{Dependence on the orientation data and real structure}\label{S.dep.data} Assume $(\Si, c)$ is a  symmetric Riemann surface with $r$ pairs of conjugate marked points. We first describe how the RGW invariants depend on the choice of orientation data. 
 \begin{lemma} \label{indep.or.gen}  For  any two orientation data $\mathfrak{o}_1$, $\mathfrak{o}_2$  for $(\Si, c)$,  there exists $m\in \Z$ such that 
 \bear\label{diff.or.gen}
 	RGW^{c,\mathfrak{o}_1}_{d} (\Si|L)_{\la^1\dots \la^r}=(-1)^{dm}RGW^{c,\mathfrak{o}_2}_{d} (\Si|L)_{\la^1\dots \la^r}
 \eear 
 for all $d$ and all collections of $r$ partitions $\la^1, \dots, \la^r$ of d. 
 
 For every $(\Si,c)$ there exist two orientation data for which the sign difference is $(-1)^d$.
 	\end{lemma}
\begin{proof} It suffices to prove this when $\Si$ is either connected or a doublet. Assume 
$\mathfrak{o}_i=(\Theta_i, \psi_i,  \mathfrak{s}_i)$ are two orientation data  for $(\Si, c)$, cf.  Definition~\ref{TRO}. 

{\sc Case 1}. When $\Si$ is a doublet, Lemma~\ref{double.vfc.lemma} implies that the $RGW$ invariants for the two orientations  differ by a factor of $(-1)^{dm_2}$, where $m_2$ is the  difference between the degrees of the restrictions to the second component of $\Si$ of the bundles $\Theta_i$. Changing the degree of $\Theta_1$ by 1 on the first component and by -1 on the second gives rise to a sign difference of $(-1)^d$. 

{\sc Case 2}. Assume next that $\Si$ is connected. Choose a separating collection $\{\gamma_i\}$ of circles, each one of which is either fixed or a crosscap. Trivialize the complex line bundle $L$ in a neighborhood of the $\gamma_i$, and split off a level 0-sphere containing no marked points, one for each $\gamma_i$. The complement of these spheres is then a doublet. 
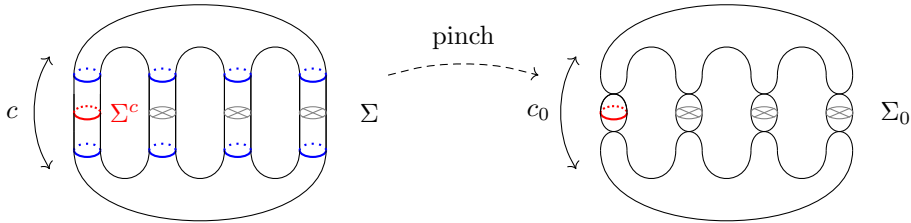
\best
\begin{tikzpicture}[scale=1/2, rotate=0, every tqft/.style={transform shape},
		tqft/.cd, 
		cobordism edge/.style={draw},
		]
\begin{scope}		
	\pic [tqft, incoming boundary components=0, outgoing boundary components=4, 
			cobordism edge/.style={draw},
			cobordism height=3, name=a,   at={(0,3)}]; 
	\pic [tqft/cylinder, cobordism height=2, 
		every upper boundary component/.style={draw,blue,  dotted, thick},
  		every lower boundary component/.style={draw, blue, thick},
	at={(0,0)}];  		
	\pic [tqft/cylinder, cobordism height=.5, 
	 	cobordism edge/.style={draw},
	at={(0,-1.5)}];  
	\pic [tqft/cylinder, cobordism height=.5, 
	 	cobordism edge/.style={draw},
	at={(0,-.5)}];  
	\pic [tqft/cylinder, cobordism height=.5, 
		every incoming upper boundary component/.style={draw, densely dotted, thick, red},
  		every incoming lower boundary component/.style={draw, thick, red},
	at={(0,-1)}];  
	
	\pic [tqft/cylinder, at={(2,0)}];    
	\pic [tqft/cylinder, at={(4,0)}];  
	\pic [tqft/cylinder, at={(6,0)}]; 
	\pic [tqft/cylinder, cobordism height=2, 
		every upper boundary component/.style={draw,blue,  dotted, thick},
  		every lower boundary component/.style={draw, blue, thick},
	at={(2,0)}];  
	\pic [tqft/cylinder, cobordism height=2, 
		every upper boundary component/.style={draw,blue,  dotted, thick},
  		every lower boundary component/.style={draw, blue, thick},
	at={(4,0)}];  
	\pic [tqft/cylinder, cobordism height=2, 
		every upper boundary component/.style={draw,blue,  dotted, thick},
  		every lower boundary component/.style={draw, blue, thick},
	at={(6,0)}];  
	\pic [tqft, incoming boundary components=4, outgoing boundary components=0, 
		cobordism edge/.style={draw},
		cobordism height=3, 
	name=b, at={(0,-2)}];   
	\begin{scope}[rotate=0, shift={(4,-1)}, scale=2,  gray]
		\draw (0,0) ellipse (.18 and .08);
		\draw[] (-.14,-.06)--(.14,.06);
		\draw[] (.14,-.06)--(-.14,.06);
	\end{scope};
	\begin{scope}[rotate=0, shift={(2,-1)}, scale=2,  gray]
		\draw (0,0) ellipse (.18 and .08);
		\draw[] (-.14,-.06)--(.14,.06);
		\draw[] (.14,-.06)--(-.14,.06);
	\end{scope};
	\begin{scope}[rotate=0, shift={(6,-1)}, scale=2, gray]
		\draw (0,0) ellipse (.18 and .08);
		\draw[] (-.14,-.06)--(.14,.06);
		\draw[] (.14,-.06)--(-.14,.06);
	\end{scope};
		\draw[<->, rounded corners=20pt,  solid] (-1,.5)--(-1.8, -1)--(-1,-2.5);
		\node at (-2,-1) {$c$};
		\node at (1,-1) {$\textcolor{red}{\Sigma^c}$};
		\node at (7.5,-1) {$\Sigma$};
\end{scope}
\begin{scope}[shift={(14,0)}]
	\pic [tqft, incoming boundary components=0, outgoing boundary components=4, 
		cobordism edge/.style={draw},
		cobordism height=3, 
		name=a,  at={(0,3)}]; 
	\pic [tqft/cap, 
		every upper boundary component/.style={draw, densely dotted, thick, red},
  		every lower boundary component/.style={draw, thick, red},
	at={(0,1)}];  
	\pic [tqft/cap, at={(0,0)}]; 
	\pic [tqft/cap, at={(2,0)}];    
	\pic [tqft/cap, at={(4,0)}];  
	\pic [tqft/cap, at={(6,0)}]; 
	\pic [tqft/cap, at={(0,1)}];  
	\pic [tqft/cap, at={(2,1)}];    
	\pic [tqft/cap, at={(4,1)}];  
	\pic [tqft/cap, at={(6,1)}]; 
	\pic [tqft/cup, at={(0,-1)}];  
	\pic [tqft/cup, at={(2,-1)}];    
	\pic [tqft/cup, at={(4,-1)}];  
	\pic [tqft/cup, at={(6,-1)}]; 
	\pic [tqft/cup, at={(0,0)}];  
	\pic [tqft/cup, at={(2,0)}];    
	\pic [tqft/cup, at={(4,0)}];  
	\pic [tqft/cup, at={(6,0)}]; 
	\pic [tqft, incoming boundary components=4, outgoing boundary components=0, 
		cobordism edge/.style={draw},
		cobordism height=3, 
	name=b, at={(0,-2)}];   
	\begin{scope}[rotate=0, shift={(4,-1)}, scale=2, gray]
		\draw (0,0) ellipse (.18 and .08);
		\draw[] (-.14,-.06)--(.14,.06);
		\draw[] (.14,-.06)--(-.14,.06);
	\end{scope};
	\begin{scope}[rotate=0, shift={(2,-1)}, scale=2, gray]
		\draw (0,0) ellipse (.18 and .08);
		\draw[] (-.14,-.06)--(.14,.06);
		\draw[] (.14,-.06)--(-.14,.06);
	\end{scope};
	\begin{scope}[rotate=0, shift={(6,-1)}, scale=2, gray]
		\draw (0,0) ellipse (.18 and .08);
		\draw[] (-.14,-.06)--(.14,.06);
		\draw[] (.14,-.06)--(-.14,.06);
	\end{scope};
		\draw[<->, rounded corners=20pt,  solid] (-1,.5)--(-1.8, -1)--(-1,-2.5);
		\node at (-2,-1) {$c_0$};
		\node at (7.5,-1) {$\Sigma_0$};
\end{scope}		
\draw[->, rounded corners=20pt, densely dashed] (8, 0)--(10, .5)--(12,0);
\node at (10,1) {pinch};
\end{tikzpicture}
\eest

Any orientation data on $\Si$ can similarly be split to induce an orientation data on the split surface  $\Si_0$. For two different orientation data on the split surface, the invariants of the $i$'th  sphere differ by a factor of $\ep_i^d$, where $\ep_i=\pm 1$ (by Proposition \ref{crosscap}), and the invariants on the doublet by $(-1)^{dm}$ (as above). The splitting formula \eqref{split.separate} then implies the same is true for the invariants of the original surface. 
\end{proof}
 
\begin{lemma} \label{indep.or} Assume $\Si$ is a connected surface with $2r$ pairs of marked points, and $c_1$, $c_2$ are two real structures on $\Si$. Then for every orientation data $\mathfrak{o}_{c_1}$ on $(\Si,c_1)$ there exists an orientation data $\mathfrak{o}_{c_2}$ on $(\Si,c_2)$ so that
 \bear\label{indep.cx}
 RGW^{c_1,\mathfrak{o}_1}_{d} (\Si|L)_{\la^1\dots \la^r}=RGW^{c_2,\mathfrak{o}_2}_{d} (\Si|L)_{\la^1\dots\la^r}. 
 \eear
 \end{lemma}
\begin{proof} Real structures on $\Si$ are classified topologically by the number of fixed circles of $\Si$ and the orientability of $\Si/c$. We can transform $(\Si,c_1)$ into $(\Si,c_2)$ via a sequence of splittings of a sphere around either a crosscap or a fixed circle as above and replacing that sphere by a sphere with the other real structure. By Proposition \ref{crosscap}  we can choose the orientation data on the new sphere so that its invariants match those of the old sphere. The claim follows from the splitting formula \eqref{split.formula.2}.  
\end{proof}

\begin{rem}\label{R.canord}  When the target is connected, Lemma  \ref{lemcdouble.vfc} implies that the orientation of the doublet moduli space depends neither on the choice of orientation data, nor on the real structure of the target. 

When the target is a doublet, then up to deformation different choices of orientation data are distinguished by the degree of $\Theta|_{\Si_2}$, cf. Example~\ref{R.or.choice}(b).  As in the proof above, the local RGW invariants then differ by a factor of $(-1)^{dm_2}$, where $m_2$ is the difference between these degrees. 
\end{rem} 

\subsection{Canonical  RGW invariants} Assume $(\Si, c)$ is a symmetric surface with $r$ pairs of conjugate points. The discussion above partitions the choices of orientation data $\fo$ on $(\Si, c)$ into two nonempty classes, distinguished by the sign $\ep_\fo=\pm  1$ of the $d=1$ cover of $(\Si, c)$. 
\begin{defn} \label{D.o.can} A \textsf{canonical twisted orientation} for $(\Si,c)$ corresponds to a choice of twisted orientation data 
$\fo_{can}=\fo$ for which the degree 1 cover of $\Si$  has sign  $\ep_\fo=+1$.
\end{defn}

\begin{cor}\label{C.indep.or} With the notation above,
\bear\label{def.RGW.can}  
RGW_{d} (\Si|L)_{\la^1\dots \la^r} = (\ep_\fo)^d  RGW^{c,\mathfrak{o}}_{d} (\Si|L)_{\la^1\dots \la^r} 
\eear
is well defined, independent of the orientation data $\fo$ and of the real structure $c$ on $\Si$;  in particular, 
\best\label{def.rwg.can}
RGW_{d} (\Si|L)_{\la^1\dots \la^r} = RGW^{c,\mathfrak{o}_{can}}_{d} (\Si|L)_{\la^1\dots \la^r}.
\eest
It is also compatible with the splitting formula \eqref{split.formula.2}, in the sense that 
 \bear\label{split.formula.2.1}
RGW_d(\Si, L)_{\mu^1\dots \mu^r}^{\nu^1\dots\nu^s}= \sum_{\la\vdash d} 
RGW_d(\wt \Si, \wt L)_{\la, \mu^1\dots\mu^r} ^ {\la, \nu^1\dots\nu^s}. 
\eear
\end{cor} 
\begin{proof} The fact that \eqref{def.RGW.can} is independent of the orientation data $\fo$ on $(\Si,c)$ follows from Lemma~\ref{indep.or.gen}. 
Next, \eqref{def.RGW.can} is also independent of the real structure $c$ on $\Si$  by Lemma~\ref{indep.or} since   
\eqref{indep.cx} for $d=1$ implies that the sign of the degree 1 cover is the same for both $\fo_1$ and $\fo_2$. 
Finally, under the splitting \eqref{split.formula.2} degree 1 covers split as degree 1 covers, giving \eqref{split.formula.2.1}. 
\end{proof}

We end this section with a few consequences of this discussion. 
\begin{cor}\label{C.cong}
The degree $d$, connected genus $h$ real invariants of a connected genus $g$ target vanish unless $d(g-1)+h-1\equiv 0 \mod 2$.
\end{cor}
\begin{proof} By Corollary \ref{C.indep.or} the  RGW invariants \eqref{def.RGW.can} are independent of the choice of real structure and of orientation data on the target.  For a connected target, the same is true for the doublet invariants $DRGW$ by Remark~\ref{R.canord}. Since the RGW invariants are equal to $\exp(CRGW+DRGW)$ as in \eqref{rgw=exp} it follows that the connected RGW invariants of a connected target $\Si$ are also independent of these choices. Finally, when the real structure on the connected genus $g$ target has no fixed locus, there are no real degree $d$ maps from a connected genus $h$ surface to $\Si$  unless $d(g-1)+h-1\equiv 0 \mod 2$ cf. \cite[Example 5.1]{gz3}. Therefore the connected invariants vanish for any choice of real structure and orientation data unless this condition is satisfied. 
\end{proof}

\begin{lemma}\label{signOm}
	Exchanging the order within the $i$-th pair of conjugate marked points of $\Si$ changes 
	$RGW_{d} (\Si|L)_{\la^1\dots \la^r}$ by a factor of $(-1)^{d-\ell(\la^i)}$. Exchanging two pairs of conjugate points does not change the invariant.
	\end{lemma}
\begin{proof}
	Exchanging $x_i^+ \leftrightarrow  x_i^-$ in the target  exchanges their $\ell(\la^i)$ preimages, contributing the 
	$(-1)^{\ell(\la^i)}$ factor. For a degree 1 map, $\ell(\la^i)=1$ and thus this changes the sign of the degree 1 
	cover by a factor of -1. This forces a change in the twisted orientation data to compensate for the $-$ sign on the degree 1 cover as in  Lemma~\ref{indep.or.gen}. The effect of this change on a degree~$d$ map is $(-1)^d$. Altogether, this implies the first claim. The second claim follows immediately since permuting pairs of conjugate points in the domain is relatively orientation preserving at the level of the DM moduli spaces. 
\end{proof}

\begin{cor}
	The  degree $d$ RGW invariants \eqref{def.RGW.can} of a connected target vanish unless $d-\ell(\la^i)\equiv 0$ mod $2$ for all $i$.
	\end{cor}
\begin{proof}
This follows by Lemma~\ref{signOm} since on a connected target we can find a path connecting $x^+$ to $x^-$ and therefore continuously deform the pair of  conjugate marked points $(x^+, x^-)$ into $(x^-, x^+)$. 
\end{proof}

\begin{cor} \label{compBP} For a $g$-doublet target with all the positive marked points on the  first component, we have 
\bear\label{doublet=gw}
	RGW_d(g,g|k_1,k_2)_{\la^1\dots\la^r}(u,t)=(-1)^{dk_2} GW_d(g|k_1,k_2)_{\la^1\dots \la^r}(iu,it).
\eear
	\end{cor}
\begin{proof} Since the degree $d=1$ cover of a complex curve counts positively, Lemma~\ref{double.vfc.lemma} implies that, for a doublet target $\Si=\Si_1\sqcup \Si_2$,  $\fo_{can}$ corresponds to a choice $(\Theta,\psi,\mathfrak{s})$ such that 
	\best c_1(\Theta)[\Si_2] \equiv 0  \mod 2. 
	\eest
(Note that $\ell_2=0$ because by assumption all the + points are on the first component.) By Corollary \ref{corR=GW.part.rel}, the real and complex invariants differ by a factor of  $(-1)^{dk_2+\chi/2}$. Since the correspondence $(u,t)\mapsto(iu,it)$
changes the coefficient $GW_{d, \chi}(g| k_1,k_2)_{\la^1\dots\la^r}$ by $(-1)^{\chi/2}$,  we obtain \eqref{doublet=gw}.
\end{proof}

 \vskip.2in 
\section{TQFT and Klein TQFT}\label{KTQFT}
 \vskip.1in 
 
We will use the local RGW invariants to define an extension of a semi-simple Klein TQFT in \S\ref{KTQFTforRGW}, which we completely solve in \S\ref{Solving}, obtaining explicit closed formulas for the local RGW invariants. This section contains a brief overview of TQFTs and Klein TQFTs,  following \cite[\S4]{bp1} and \cite[\S1]{braun} (up to some change in notation), and a discussion of   semi-simple ones. 

Let $\cob$ be the usual (oriented, closed) 2-dimensional cobordism category. It is the symmetric monoidal category with objects given by compact oriented 1-manifolds (without boundary) and morphisms given by (diffeomorphism classes of) oriented cobordisms. A \textsf{2-dimensional topological quantum field theory (2d TQFT)}  with values in a commutative ring $R$ is a symmetric monoidal functor
\best
F:\cob \ra R\mathrm {mod},
\eest
where $R\mathrm {mod}$ is the category of $R$-modules. This is equivalent to a commutative Frobenius algebra over $R$; the product and co-product correspond to the pair of pants while the unit and co-unit to the cap and cup respectively, see Figure~\ref{F.gen.2cob}.   

In \cite[\S4.2]{bp1}, Bryan and Pandharipande enlarge the category $\cob$ to a category $\cob^{L_1, L_2}$
with the same objects, but with extra morphisms. The morphisms are now equivalence classes of oriented cobordisms $W$ decorated by a pair of complex line bundles $L_1, L_2\ra W$ trivialized over the boundary. The equivalence is up to bundle isomorphisms covering diffeomorphisms between the cobordisms (and compatible with the trivializations over the boundary). The composition is given by concatenation of the cobordisms and gluing the bundles using the trivializations over the boundary. 

For example, a morphism in $\cob^{L_1, L_2}$ corresponding to a connected cobordisms $W$ is completely determined by the genus $g$ of $W$ together with a pair of integers $(k_1, k_2)$, called the  \textsf{level},  recording the  Euler classes $e(L_i) \in H^2(W, \bd W)$. Restricting the morphisms to $k_1=k_2=0$ defines an embedding 
\best\label{emb.cob} 
\cob\subset \cob^{L_1, L_2}.
\eest

In \cite[\S4.4]{bp1} Bryan-Pandharipande use the local GW invariants to define a symmetric monoidal functor
\bear\label{GW.functor}
{\bf GW}:\cob^{L_1, L_2} \ra R\mathrm {mod}. 
\eear 
on this larger category. The functor \eqref{GW.functor} extends the classical 2d TQFT that appeared in the work of 
Dijkgraaf-Witten \cite{DW} and Freed-Quinn \cite{FQ}, and whose Frobenius algebra is the center $\Q[S_d]^{S_d}$ of the group algebra of the symmetric group $S_d$.  It is used to completely solve the local Gromov-Witten theory.\\
\smallskip
 
 A different extension of $\cob$ is obtained by allowing unoriented and possibly unorientable surfaces as cobordisms; see \cite{AN, braun}. We refer to this category as $\kcob$, where $\bf{K}$ stands for Klein (surface).  The objects are closed unoriented 1-manifolds   and the morphisms are diffeomorphism classes of {\em unoriented} (and possibly unorientable) cobordisms. An equivalent point of view is to consider the orientation double covers of both the objects and the morphisms: (i) the objects are then closed oriented 1-manifolds with an orientation-reversing involution (deck transformation) exchanging the sheets of the cover and (ii) the morphisms are compact oriented 2-dimensional manifolds with a {\em fixed-point free} orientation-reversing involution extending the one on the boundary. Such 2-dimensional manifolds are called \textsf{symmetric surfaces} and we denote this category by $\SymRiem$.  Moreover
 $$
 \kcob \equiv \SymRiem
 $$
 where the identification is obtained by passing to the orientation double cover  in one direction and taking the quotient by the involution  in the other direction. Working from the perspective of $\SymRiem$ allows us to construct an extension of this category related to that of \cite{bp1} and completely solve the local real Gromov-Witten theory. For this reason we describe $\kcob$ and $\SymRiem$ in parallel below.

\begin{rem}\label{R.cob-subcateg} As mentioned after \cite[Definition~1.7]{braun}, it is convenient to identify $\kcob$ (and respectively $\cob$) with its  skeleton, which is the full subcategory whose objects are disjoint unions of copies of a {\em fixed} oriented circle $S^1$. For $\SymRiem$ we take the full subcategory whose objects are disjoint unions  of two circles ${\cal S}= (S^1\sqcup \ov{S^1}, \ep)$, where $\ov {S^1}$ denotes the circle with opposite orientation and $\ep|_{S^1}=id:S^1\lra \ov {S^1}$. This way, $\cob$ can be regarded as a subcategory of $\kcob$ with the same objects, but fewer morphisms:
	\best\label{emb.cob.kcob} 
	\cob\subset \kcob.
	\eest 
Note that even if a cobordism in $\kcob$ is orientable, there may not be way to orient it in a way compatible with the boundary identifications. For example,  Figure~\ref{F.id.U} shows two different cobordisms, the first one being the tube (which induces the identity). The second one reverses the orientation of the $S^1$ and we refer to it as the \textsf{involution} $\Omega$. It is a morphism in $\kcob$ but not in $\cob$. The difference is even more visible from the perspective of $\SymRiem$, cf.  second cobordism in \eqref{F.double.xcap}. 
\end{rem}

\begin{figure}[ht!]
	\begin{center}
		\begin{tikzpicture}[baseline={(current bounding box.center)}, scale=1, rotate=90, every tqft/.style={transform shape},
		tqft/.cd, 
		cobordism/.style={draw},
		every upper boundary component/.style={draw},
  		every incoming lower boundary component/.style={draw, densely dashed},
  		every outgoing lower boundary component/.style={draw},
		every outgoing upper boundary component/.style={decorate, decoration={markings,
				mark=at position .5 with {\arrow{<}}, },}, 
		every incoming upper boundary component/.style={decorate, decoration={markings,
				mark=at position .5 with {\arrow{<}}, },}  
		]
		\pic [tqft/cylinder, name=f];    
		\node at ([yshift=30pt] f-outgoing boundary 1) {$id$};
		\end{tikzpicture}
		\hskip.4in
		\begin{tikzpicture}[baseline={(current bounding box.center)}, scale=1, rotate=90, every tqft/.style={transform shape},
		tqft/.cd, 
		cobordism/.style={draw},
		every upper boundary component/.style={draw},
  		every incoming lower boundary component/.style={draw, densely dashed},
  		every outgoing lower boundary component/.style={draw},
		every outgoing upper boundary component/.style={decorate, decoration={markings,
				mark=at position .5 with {\arrow{>}}, },}, 
		every incoming upper boundary component/.style={decorate, decoration={markings,
				mark=at position .5 with {\arrow{<}}, },}  
		]
		\pic [tqft/cylinder, name=f];    
		\node at ([yshift=30pt] f-outgoing boundary 1) {$\Om$};
		\end{tikzpicture}
	\end{center}
	\caption{The tube (identity) and the involution $\Om$ in $\kcob$.  } 
	\label{F.id.U}
\end{figure}
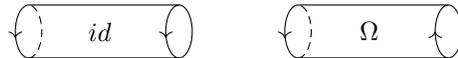

When $\cob$ is regarded as a subcategory of $\kcob$ as described in Remark~\ref{R.cob-subcateg}, its generators are given in Figure~\ref{F.gen.2cob}  (cf. \cite[Figure~1.1]{braun}). The corresponding elements of $\SymRiem$ are their orientation double covers, cf. Figure~\ref{F.gen.2cobSR}.
  
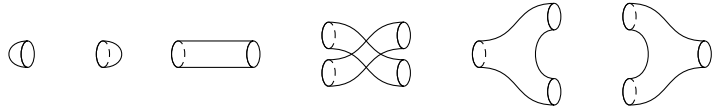
\begin{figure}[ht!]
	\begin{tikzpicture}[ baseline={(current bounding box.center)}, scale=1/2, rotate=90, every tqft/.style={transform shape},
		tqft/.cd, 
		cobordism/.style={draw},
		every upper boundary component/.style={draw},
  		every incoming lower boundary component/.style={draw, densely dashed},
  		every outgoing lower boundary component/.style={draw},
		]		
		\pic [tqft/cap, at={(0,10)}, name=a];    
		\pic[tqft/cup, at={(0,6)}, name=b];
		\pic [tqft/cylinder, at={(0,4)}, name=c];    
		\pic [tqft/cylinder to next, anchor=incoming boundary 1,  at={(-0.5,0)}, name=d];
		\pic [tqft/cylinder to prior,anchor=incoming boundary 1, 
		at=(d-outgoing boundary |- d-incoming boundary), name=e, ];     
		\pic [tqft/pair of pants, at={(0,-4)}];  
		\pic [tqft/reverse pair of pants, at={(-1,-8)}];  
	\end{tikzpicture}
 	\caption{The elementary cobordisms: cap, cup, tube, twist and  the pairs of pants in $\cob\subset\kcob$.}
	\label{F.gen.2cob}
\end{figure} 

\begin{figure}[h!]
	\begin{tikzpicture}[ baseline={(current bounding box.center)}, scale=1/2, rotate=90, every tqft/.style={transform shape},
		tqft/.cd, 
		cobordism/.style={draw},
		every upper boundary component/.style={draw},
  		every incoming lower boundary component/.style={draw, densely dashed},
  		every outgoing lower boundary component/.style={draw},
		]		
		\pic [tqft/cap, at={(0,10)}, name=a];    
		\pic[tqft/cup, at={(0,6)}, name=b];
		\pic [tqft/cylinder, at={(0,4)}, name=c];    
		\pic [tqft/cylinder to next, anchor=incoming boundary 1,  at={(-0.5,0)}, name=d];
		\pic [tqft/cylinder to prior,anchor=incoming boundary 1, 
		at=(d-outgoing boundary |- d-incoming boundary), name=e, ];     
		\pic [tqft/pair of pants, at={(0,-4)}];  
		\pic [tqft/reverse pair of pants, at={(-1,-8)}];  
\draw[dashed] (2,9)--(2,-11);
\draw[<->, rounded corners=20pt,  solid] (1,10)--(2, 11)--(3,10);
\node at (2,11) {$c$};
		\pic [tqft/cap, at={(4,10)}, name=a];    
		\pic[tqft/cup, at={(4,6)}, name=b];
		\pic [tqft/cylinder, at={(4,4)}, name=c];    
		\pic [tqft/cylinder to next, anchor=incoming boundary 1,  at={(3.5,0)}, name=d];
		\pic [tqft/cylinder to prior,anchor=incoming boundary 1, 
		at=(d-outgoing boundary |- d-incoming boundary), name=e, ];     
		\pic [tqft/pair of pants, at={(4,-4)}];  
		\pic [tqft/reverse pair of pants, at={(3,-8)}];  			
	\end{tikzpicture}
	\caption{The elementary cobordisms: cap, cup, tube, twist and  the pairs of pants in $\SymRiem$.}
	\label{F.gen.2cobSR}
\end{figure}
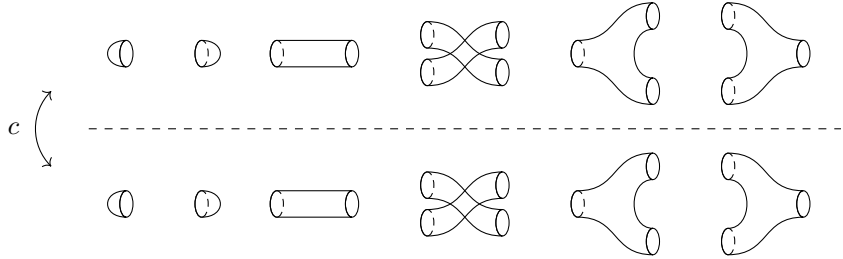	
The category $\kcob$ has two extra generators, the cross-cap   (a M\"obius band) and the involution
\bear\label{F.extra.gen.kcob}
		\begin{tikzpicture}[baseline={(current bounding box.center)}, scale=1, rotate=90, every tqft/.style={transform shape},
		tqft/.cd, 
		cobordism/.style={draw},
		every upper boundary component/.style={draw},
		every incoming lower boundary component/.style={draw, densely dashed},
		every lower boundary component/.style={draw},
		every outgoing upper boundary component/.style={decorate, decoration={markings,
				mark=at position .5 with {\arrow{>}}, },}, 
		every incoming upper boundary component/.style={decorate, decoration={markings,
				mark=at position .5 with {\arrow{<}}, },},  
		]
		\pic [tqft/cap, name=a, at={(0,1.67)}];    
		\begin{scope}[scale=1]
		\draw (0,0) ellipse (.18 and .08);
		\draw[] (-.14,-.06)--(.14,.06);
		\draw[] (.14,-.06)--(-.14,.06);
		\end{scope};
		\pic [tqft/cylinder, 
  		every incoming lower boundary component/.style={draw, densely dashed},
  		every outgoing lower boundary component/.style={draw},
		name=f, at={(0,-2)},   
		]; 
		\node at ([yshift=30pt] f-outgoing boundary 1) {$\Om$};

		\end{tikzpicture}
\eear
respectively. In  $\SymRiem$ these correspond to their orientation double covers:
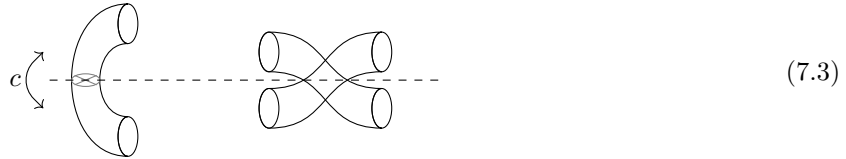
\bear\label{F.double.xcap} 
	\begin{tikzpicture}[baseline={(current bounding box.center)}, scale=3/4, rotate=90, every tqft/.style={transform shape},
		tqft/.cd, 
		cobordism/.style={draw},
		every upper boundary component/.style={draw},
		every lower boundary component/.style={draw},
		]
		\pic [tqft, 
		incoming boundary components=0, 
		outgoing boundary components=2, 
		anchor=incoming boundary 1, 
		 name=a, at={(0,5)}];    
		\begin{scope}[rotate=90, shift={(3.75,-1)}, scale=1.4, gray]
		\draw (0,0) ellipse (.18 and .08);
		\draw[] (-.14,-.06)--(.14,.06);
		\draw[] (.14,-.06)--(-.14,.06);
		\end{scope};
		\pic [tqft/cylinder to next, anchor=incoming boundary 1,name=c, at={(.5,.5)} ];
		\pic [tqft/cylinder to prior, anchor=incoming boundary 1, 
		at=(c-outgoing boundary |- c-incoming boundary), name=d, ];
		\draw[dashed] (1,-2.5)--(1,4.5);
		\draw[<->, rounded corners=10pt,  solid] (.5,4.5)--(1, 5)--(1.5,4.5);
		\node at (1,5) {$c$};
	\end{tikzpicture}
\eear
Note that in $\SymRiem$ the involution swaps the two outgoing circles.

The extra generators satisfy certain relations in $\kcob$ (see p 1840-1841 of \cite{braun}). For example,
moving a puncture  once around the M\"obius band   changes the  orientation of the puncture, cf. Figure~\ref{F.Moeb}; equivalently, the involution acts trivially on the product of the cross-cap with another element, cf. \eqref{k.ax.2.1}.
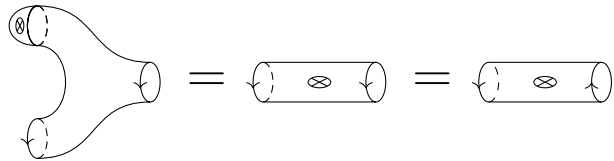
\begin{figure}[h!]
	\begin{centering}
		\begin{tikzpicture}[baseline={(current bounding box.center)}, scale=3/4, rotate=90, every tqft/.style={transform shape},
		tqft/.cd, 
		cobordism/.style={draw},
		every upper boundary component/.style={draw},
		every lower boundary component/.style={draw, densely dashed},
		]
		\pic [tqft/reverse pair of pants, name=c, anchor=incoming boundary 1, 
		outgoing lower boundary component 1/.style={draw, solid}, 
		incoming lower boundary component 1/.style={draw, densely dashed}, 
		every outgoing upper boundary component/.style={decorate, decoration={markings,
				mark=at position .5 with {\arrow{<}}, },}, 
		incoming upper boundary component 1/.style={decorate, decoration={markings,
				mark=at position .5 with {\arrow{<}}, },}, 
		 at={(2, 0)}];  
		\pic [tqft/cap, name=d, anchor=outgoing boundary 1, at=(c-incoming boundary 2)];  
		 \begin{scope}[shift={(4,.31)}, scale=.7, rotate=90] 
		\draw[rounded corners=28pt] (-.07,-.14)--(.07,.14);
		\draw[rounded corners=28pt] (-.07,.14)--(.07,-.14);
		\draw(0,-.2) arc (270:90:.1 and 0.2);
		\draw(0,-.2) arc (270:90:-.1 and 0.2);
		\end{scope}
		\path (c-outgoing boundary 1) ++(0,-1) node[font=\Huge] {\(=\)};
		\pic [tqft/cylinder,
		 outgoing lower boundary component 1/.style={draw, solid},
		 every outgoing upper boundary component/.style={decorate, decoration={markings,
				mark=at position .5 with {\arrow{<}}, },}, 
		incoming upper boundary component 1/.style={decorate, decoration={markings,
				mark=at position .5 with {\arrow{<}}, },}, 
		 at={(3,-4)}, name=f];    
		 \begin{scope}[shift={(3,-5)}, scale=1] 
		\draw[rounded corners=28pt] (-.07,-.14)--(.07,.14);
		\draw[rounded corners=28pt] (-.07,.14)--(.07,-.14);
		\draw(0,-.2) arc (270:90:.1 and 0.2);
		\draw(0,-.2) arc (270:90:-.1 and 0.2);
		\end{scope}			
		\path (f-outgoing boundary 1) ++(0,-1) node[font=\Huge] {\(=\)};
		\pic [ draw, tqft/cylinder, scale=1,  
		every lower boundary component/.style={draw},
		every incoming lower boundary component/.style={dashed},
		every outgoing lower boundary component/.style={solid},
		every outgoing upper boundary component/.style={decorate, decoration={markings,
				mark=at position .5 with {\arrow{>}}, },}, 
		every incoming upper boundary component/.style={decorate, decoration={markings,
				mark=at position .5 with {\arrow{<}}, },},  
		at={(3,-8)}, name=a];    
		 \begin{scope}[shift={(3,-9)}, scale=1] 
		\draw[rounded corners=28pt] (-.07,-.14)--(.07,.14);
		\draw[rounded corners=28pt] (-.07,.14)--(.07,-.14);
		\draw(0,-.2) arc (270:90:.1 and 0.2);
		\draw(0,-.2) arc (270:90:-.1 and 0.2);
		\end{scope}		
		\end{tikzpicture}
	\end{centering} 
	\caption{ The cobordism $K$ and relations in $\kcob$.}
	\label{F.Moeb}
\end{figure}
\begin{figure}[h!]
	\begin{centering}
		\begin{tikzpicture}[baseline={(current bounding box.center)}, scale=1/2, rotate=90, every tqft/.style={transform shape},
		tqft/.cd, 
		cobordism/.style={draw},
		every upper boundary component/.style={draw},
		every lower boundary component/.style={draw, densely dashed},
		]
		\pic [tqft, incoming boundary components=0, outgoing boundary components=2, name=c,  anchor=outgoing boundary 1];   
		\pic [tqft/reverse pair of pants, name=a,  
			outgoing lower boundary component 1/.style={draw, solid},
			anchor= incoming boundary 2, at=(c-outgoing boundary 1)];   
		\pic [tqft/reverse pair of pants, name=b,  
			outgoing lower boundary component 1/.style={draw, solid},
			anchor= incoming boundary 1, at=(c-outgoing boundary 2)];   
		 \begin{scope}[shift={(1,.75)}, scale=1.2, gray] 
		\draw[rounded corners=28pt] (-.07,-.14)--(.07,.14);
		\draw[rounded corners=28pt] (-.07,.14)--(.07,-.14);
		\draw(0,-.2) arc (270:90:.1 and 0.2);
		\draw(0,-.2) arc (270:90:-.1 and 0.2);
		\end{scope}		
		\path (b-incoming boundary 1) ++(-1,-3) node[font=\Huge] {\(=\)};
		\pic[tqft, incoming boundary components=2, outgoing boundary components=2, 
		every outgoing lower boundary component/.style={draw, solid},
		name=z,  anchor=incoming boundary 1, 
		at=({0,-5}) ];   
		\path (z-outgoing boundary 1) ++(1,-2) node[font=\Huge] {\(=\)};
		\pic[tqft, incoming boundary components=2, outgoing boundary components=2, name=w,  anchor=incoming boundary 1, 
		at=({0,-10}) ];   
		\pic[tqft, incoming boundary components=1, outgoing boundary components=1,  
		outgoing lower boundary component 1/.style={draw, solid},
		anchor=incoming boundary 1, at=(w-outgoing boundary 1), offset=1 ]; 
		\pic[tqft, incoming boundary components=1, outgoing boundary components=1, 
		outgoing lower boundary component 1/.style={draw, solid},
		anchor=incoming boundary 1, at=(w-outgoing boundary 2), offset=-1 ];     
		 \begin{scope}[shift={(1,-6)}, scale=2.2, gray] 
		\draw[rounded corners=28pt] (-.07,-.14)--(.07,.14);
		\draw[rounded corners=28pt] (-.07,.14)--(.07,-.14);
		\draw(0,-.2) arc (270:90:.1 and 0.2);
		\draw(0,-.2) arc (270:90:-.1 and 0.2);
		\end{scope}		
		 \begin{scope}[shift={(1,-11)}, scale=2.2, gray] 
		\draw[rounded corners=28pt] (-.07,-.14)--(.07,.14);
		\draw[rounded corners=28pt] (-.07,.14)--(.07,-.14);
		\draw(0,-.2) arc (270:90:.1 and 0.2);
		\draw(0,-.2) arc (270:90:-.1 and 0.2);
		\end{scope}		
		\draw[<->, rounded corners=20pt,  solid] (-.5,2)--(1, 2.8)--(2.5,2);
		\node at (1,3) {$c$};

		\end{tikzpicture}
	\end{centering} 
	\caption{ The cobordism $K$ and relations in $\SymRiem$.}
\end{figure}
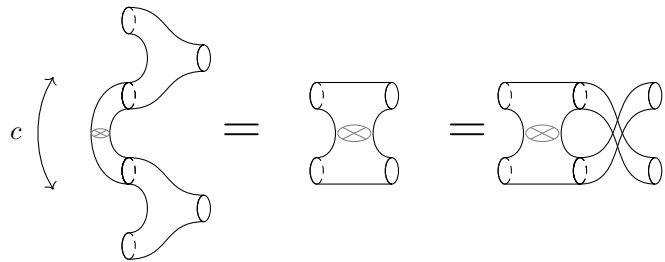

Another relation comes from decomposing the product of two cross-caps as in Figure~\ref{F.punct.Klein}, cf. \eqref{k.ax.2.2}.

\begin{figure}[h!] 
	\begin{centering}
		\begin{tikzpicture}[baseline={(current bounding box.center)}, scale=.6, rotate=90, every tqft/.style={transform shape},
		tqft/.cd, 
		cobordism/.style={draw},
		every upper boundary component/.style={draw},
		every lower boundary component/.style={draw, densely dashed},
		]
		\pic [tqft/cap, name=a,  anchor= outgoing boundary 1, at=({0,0})];   
		\pic [tqft/cap, name=b,  anchor= outgoing boundary 1, at={(2, 0)}];   
		\pic [tqft/reverse pair of pants, every outgoing boundary component/.style={draw},
		every outgoing upper boundary component/.style={decorate, decoration={markings,
				mark=at position .5 with {\arrow{<}}, },}, 
		outgoing lower boundary component 2/.style={draw, solid},  name=c, 
		anchor=incoming boundary 1, at=(a-outgoing boundary 1) ];  
		\path (c-outgoing boundary 1) ++(0,-1) node[font=\Huge] {\(=\)};
		\pic [tqft/cap, name=d,  anchor= outgoing boundary 1, at=({1,-4})];   
		\pic [tqft/pair of pants, name=f, anchor= incoming boundary 1, at=(d-outgoing boundary 1)];
		\pic [tqft/cylinder, name=c1,  every outgoing boundary component/.style={draw},
		every outgoing upper boundary component/.style={decorate, 
			decoration={markings, mark=at position .5 with {\arrow{>}}, },}, 
			anchor= incoming boundary 1, at=(f-outgoing boundary 1)];   
		\pic [tqft/cylinder, name=c2, every outgoing boundary component/.style={draw},
		every outgoing upper boundary component/.style={decorate, decoration={markings,
				mark=at position .5 with {\arrow{<}}, },}, 
			anchor= incoming boundary 1, at=(f-outgoing boundary 2)];   
		\node at ([yshift=30pt] c1-outgoing boundary 1) {$\Om$};
		\node at ([yshift=30pt] c2-outgoing boundary 1) {$id$}; 
		\pic [tqft/reverse pair of pants, every lower boundary component/.style={draw},
			every incoming lower boundary component/.style={dashed},
			every outgoing boundary component/.style={draw},
			every outgoing upper boundary component/.style={decorate, decoration={markings,
				mark=at position .5 with {\arrow{<}}, },}, 
			every incoming upper boundary component/.style={decorate, decoration={markings,
				mark=at position .5 with {\arrow{<}}, },} , 
    			name=f, anchor=outgoing boundary 1, at=({1, -11})];  
		\begin{scope}[shift={(2,0.32)}, scale=.8, rotate=90] 
		\draw[rounded corners=28pt] (-.07,-.14)--(.07,.14);
		\draw[rounded corners=28pt] (-.07,.14)--(.07,-.14);
		\draw(0,-.2) arc (270:90:.1 and 0.2);
		\draw(0,-.2) arc (270:90:-.1 and 0.2);
		\end{scope}		
		\begin{scope}[shift={(0,.32)}, scale=.8, rotate=90] 
		\draw[rounded corners=28pt] (-.07,-.14)--(.07,.14);
		\draw[rounded corners=28pt] (-.07,.14)--(.07,-.14);
		\draw(0,-.2) arc (270:90:.1 and 0.2);
		\draw(0,-.2) arc (270:90:-.1 and 0.2);
		\end{scope}		
		\end{tikzpicture}
	\end{centering} 
	\caption{Relation in  $\kcob$:     decomposing the punctured Klein bottle.}
	\label{F.punct.Klein}
\end{figure}
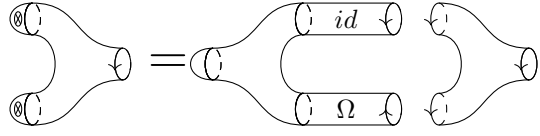

\subsection{Semi-simple Klein TQFT}\label{Semisimple}  
\begin{defn}\label{defKTQFT} A \textsf{(closed) 2d Klein TQFT}  is a symmetric monoidal functor
\bear\label{def.f.ktqft}
F:\kcob \ra R\mathrm {mod}.
\eear
\end{defn}
In fact, cf.  \cite[Prop~1.11]{braun}, 
 a (closed) 2d Klein TQFT is equivalent to a commutative Frobenius algebra $H=F(S^1)$ together with two extra structures: 
\begin{enumerate}[(a)]
	\item an involutive (anti)-automorphism $\Omega$ of the Frobenius algebra $H$,  denoted $x \mapsto x^*$. This means 
	\bear\label{k.ax.1} 
	(x^*)^*=x,\quad  (xy)^*=y^* x^* \quad \mbox{ and } \quad \lg x^*, y^* \rg = \lg x, y \rg  \quad \mbox{ for all $x, y\in H$.}
	\eear
	\item an element $U\in H$ such that 
	\bear\label{k.ax.2.1} 
	(aU)^*=aU\quad \text{for all}\,\,\, a\in H \quad \mbox {and}
	\eear
	\bear\label{k.ax.2.2} 
	U^2= m(id \otimes\Om ) (\De(1))=\sum \al_i \beta_i^*\,, \quad \mbox{ where the co-product} \quad \Delta (1)=\sum \al_i \otimes \beta_i. 
	\eear
\end{enumerate} 
The involution $\Om$ and the element $U$ correspond to the cobordisms \eqref{F.extra.gen.kcob}. For an interpretation of the relations (b), see Figure~\ref{F.Moeb} and ~\ref{F.punct.Klein}. 

There are several elements of $\kcob$ that play a special role; their images under  \eqref{def.f.ktqft} are denoted:
\bear\label{G+C}
F\l( 
\begin{tikzpicture}[baseline={(current bounding box.center)},scale=1/2, rotate=90, every tqft/.style={transform shape},
  tqft/.cd, 
  cobordism/.style={draw},
  every upper boundary component/.style={draw},
  every lower boundary component/.style={draw},
  ]
  \pic [tqft/cap];    
\end{tikzpicture}
\r)=1, 
\quad
F\l(  
\begin{tikzpicture}[baseline={(current bounding box.center)}, scale=1/3, rotate=90, every tqft/.style={transform shape},
  tqft/.cd, 
  cobordism edge/.style={draw},
      ]
 \pic [tqft/pair of pants, 
 	incoming upper boundary component 1/.style={draw}, 
	incoming lower boundary component 1/.style={draw, densely dashed},
 name=a, anchor=outgoing boundary 1, ];  
 \pic [tqft/reverse pair of pants, 
 	outgoing lower boundary component 1/.style={draw}, 
	outgoing upper boundary component 1/.style={draw}, 
 	every incoming lower boundary component/.style={draw, densely dashed},
	every incoming upper boundary component/.style={draw, densely dashed},
	anchor=incoming boundary 1];  
\end{tikzpicture}
\r)=G,
\quad 
\mbox{ and} 
\quad 
F\l( 
\begin{tikzpicture}[baseline={(current bounding box.center)}, scale=1/2, rotate=90, every tqft/.style={transform shape},
  tqft/.cd, 
  cobordism/.style={draw},
  every upper boundary component/.style={draw},
  every lower boundary component/.style={draw, densely dashed},
  ]
\pic [tqft/cup];    
\end{tikzpicture}
\r)=C,
\eear
\bear\label{O+U}
F\Big(
\begin{tikzpicture}[baseline={(current bounding box.center)},scale=1/2, rotate=90, every tqft/.style={transform shape},
  tqft/.cd, 
  cobordism/.style={draw},
   every upper boundary component/.style={draw},
  every incoming lower boundary component/.style={draw, densely dashed},
  every outgoing lower boundary component/.style={draw},
  every outgoing upper boundary component/.style={decorate, decoration={markings,
      mark=at position .5 with {\arrow{<}}, },}, 
   every incoming upper boundary component/.style={decorate, decoration={markings,
      mark=at position .6 with {\arrow{>}}, },}  
      ]
 \pic [tqft/cylinder, name=f];    
\end{tikzpicture}
\Big)= \Om, 
\quad 
F\Big(
\begin{tikzpicture}[baseline={(current bounding box.center)}, scale=3/4,  rotate=90, every tqft/.style={transform shape},
  tqft/.cd, 
  cobordism/.style={draw},
  every upper boundary component/.style={draw},
  every lower boundary component/.style={draw},
]
\pic [tqft/cap, name=d];    
\begin{scope}[shift={(0,-1.67)}, scale=.8, rotate=90] 
\draw[rounded corners=28pt] (-.07,-.14)--(.07,.14);
\draw[rounded corners=28pt] (-.07,.14)--(.07,-.14);
\draw(0,-.2) arc (270:90:.1 and 0.2);
\draw(0,-.2) arc (270:90:-.1 and 0.2);
\end{scope}
\end{tikzpicture}
\Big)=U,
\quad \mbox{ and } \quad
F\l( 
\begin{tikzpicture}[baseline={(current bounding box.center)},scale=1/2, rotate=90, every tqft/.style={transform shape},
  tqft/.cd, 
  cobordism edge/.style={draw},
  every upper boundary component/.style={draw},
  every incoming lower boundary component/.style={draw, densely dashed},
  every outgoing lower boundary component/.style={draw},
      ]
 \pic [tqft/cylinder, name=f];   
 \begin{scope}[shift={(0,-1)}, scale=1.4] 
\draw[rounded corners=28pt] (-.07,-.14)--(.07,.14);
\draw[rounded corners=28pt] (-.07,.14)--(.07,-.14);
\draw(0,-.2) arc (270:90:.1 and 0.2);
\draw(0,-.2) arc (270:90:-.1 and 0.2);
\end{scope}
\end{tikzpicture}
\r)=K. 
\eear
 When  \eqref{def.f.ktqft} is regarded as a morphism on $\SymRiem\equiv \kcob$ via the orientation double cover construction, we denote it 
 \bear\label{lift.F.sym}
 \wt F:  \SymRiem\ra R\mathrm{mod}.
 \eear In particular,  
 \bear\label{K.=} 
\wt F\l(
\begin{tikzpicture}[ baseline={([yshift=0, xshift=0] current bounding box.center)}, scale=1/2, rotate=90, every tqft/.style={transform shape},
		tqft/.cd, 
		cobordism/.style={draw},
		every upper boundary component/.style={draw},
		every lower boundary component/.style={draw},
		]
		\pic [tqft/cylinder to next, anchor=incoming boundary 1,name=c, at={(.5,.5)} ];
		\pic [tqft/cylinder to prior, anchor=incoming boundary 1, 
		at=(c-outgoing boundary |- c-incoming boundary), name=d, ];			
	\end{tikzpicture}
\r)= \Omega,
\quad 
\wt F\l(
\begin{tikzpicture}[ baseline={([yshift=0, xshift=0] current bounding box.center)}, scale=1/2, rotate=90, every tqft/.style={transform shape},
		tqft/.cd, 
		cobordism/.style={draw},
		every upper boundary component/.style={draw},
		every lower boundary component/.style={draw},
		]
		\pic [tqft, 
		incoming boundary components=0, 
		outgoing boundary components=2, 
		 name=a, at={(0,0)}];    
		\begin{scope}[rotate=90, shift={(-1.25,-1)}, scale=1.4, gray]
		\draw (0,0) ellipse (.18 and .08);
		\draw[] (-.14,-.06)--(.14,.06);
		\draw[] (.14,-.06)--(-.14,.06);
		\end{scope};
	\end{tikzpicture}
\r)=U,
\quad \mbox{ and }\quad 
\wt F\l(
\begin{tikzpicture}[baseline={(current bounding box.center)}, scale=1/3, rotate=90, every tqft/.style={transform shape},
		tqft/.cd, 
		cobordism/.style={draw},
		every upper boundary component/.style={draw},
		every lower boundary component/.style={draw},
		]
		\pic[tqft, incoming boundary components=2, outgoing boundary components=2 ];   
		 \begin{scope}[shift={(1,-1)}, scale=2.2, gray] 
		\draw[rounded corners=28pt] (-.07,-.14)--(.07,.14);
		\draw[rounded corners=28pt] (-.07,.14)--(.07,-.14);
		\draw(0,-.2) arc (270:90:.1 and 0.2);
		\draw(0,-.2) arc (270:90:-.1 and 0.2);
		\end{scope}		
\end{tikzpicture}
\r)= K.
\eear

\begin{defn}
	A \textsf{semi-simple} Klein TQFT is a Klein TQFT whose associated Frobenius algebra is semi-simple.
\end{defn}
A semi-simple TQFT is determined by the structure constants  $\{\la_\rho\}$, i.e. the coefficients of the co-multiplication $\Delta(v_\rho)=\la_\rho v_\rho \otimes v_\rho$ in the idempotent basis $\{v_\rho\}$. Moreover, 
 \begin{prop}\label{P.elem.cob} Assume \eqref{def.f.ktqft} is a semisimple KTQFT with idempotent basis $\{ v_\rho\}$, and assume that the ground ring $R$ has no zero divisors. Then 
\begin{enumerate}[(i)] 
\item $G(v_\rho)=\la_\rho v_\rho$ and $C(v_\rho)=\la_\rho^{-1}$.
\item $\Omega$ defines an involution on the idempotent basis $\Omega(v_\rho)= v_{\rho^*}$.
\item If $U=\sum_\rho U_\rho v_\rho$ then $U_\rho^2=\la_\rho$ if $\rho=\rho^*$, and $U_\rho=0$ if $\rho \ne \rho^*$. 
\item $K(v_\rho)=U_\rho v_\rho$. 
\end{enumerate} 
\end{prop}
\begin{proof} Property (i) holds for any semi-simple TQFT. To prove (ii), note that the second relation in \eqref{k.ax.1} implies
\best
\Omega(v_\rho) \Omega(v_\mu)=\Omega(v_\mu v_\rho)
\eest
for all $\rho$ and $\mu$. Since $\{ v_\rho\}$ is an idempotent basis, i.e. $v_\rho v_\mu=\de_{\rho \mu} v_\rho$ for all $\rho, \mu$,  this becomes
\best
 \sum_\nu\Omega_\rho^\nu \Omega_\mu^\nu v_\nu
=\de_{\mu \rho} \sum_\nu  \Omega^{\nu}_{\rho}v_{\nu}.
\eest
Therefore 
\bear\label{Om.sq.1}
\Omega_\rho^\nu \Omega_\mu^\nu =0   \quad\mbox{ for all }  \rho \ne \mu \mbox{ and all } \nu,  \quad \mbox{ while }
\eear
\bear\label{Om.sq.2}
(\Omega_\rho^\nu)^2= \Omega_\rho^\nu \quad\mbox{ so } \quad \Omega_\rho^\nu= 0 \mbox { or } 1  \quad\mbox{ for all }\rho, \; \nu
\eear
(because $R$ has no zero divisors). Equation \eqref{Om.sq.1} implies that there is at most one non zero element in each row of the matrix associated to 
$\Omega$ in this basis. But since $\Omega$ is invertible, there must be exactly one non-zero element in each row, which by \eqref{Om.sq.2} must be equal to 1. The invertibility of $\Omega$ also implies that there is precisely one non-zero element in every column. This proves (ii). 

Next, \eqref{k.ax.2.2}  (cf. Figure~\ref{F.punct.Klein}) implies
\bear\label{Omeq}
U^2_\rho = \Omega^\rho_\rho\la_\rho,
\eear
since $\Delta(1)=\sum_\rho \la_\rho v_\rho\otimes v_\rho$. This gives  (iii). 

Finally, property (iv) follows since $K(x)=U\cdot x$, i.e.  $K$ decomposes as
\[
K=F\l( 
\begin{tikzpicture}[baseline={(current bounding box.center)}, scale=3/4, rotate=90, every tqft/.style={transform shape},
  tqft/.cd, 
  cobordism edge/.style={draw},
  every upper boundary component/.style={draw},
  every incoming lower boundary component/.style={draw, densely dashed},
  every outgoing lower boundary component/.style={draw},
      ]
 \pic [tqft/cylinder, name=f];   
 \begin{scope}[shift={(0,-1)}, scale=1.4] 
\draw[rounded corners=28pt] (-.07,-.14)--(.07,.14);
\draw[rounded corners=28pt] (-.07,.14)--(.07,-.14);
\draw(0,-.2) arc (270:90:.1 and 0.2);
\draw(0,-.2) arc (270:90:-.1 and 0.2);
\end{scope}
\end{tikzpicture}
\r)=
F\l( 
\begin{tikzpicture}[baseline={(current bounding box.center)}, scale=3/4, rotate=90, every tqft/.style={transform shape},
		tqft/.cd, 
		cobordism/.style={draw},
		every upper boundary component/.style={draw},
		every lower boundary component/.style={draw, densely dashed},
		]
		\pic [tqft/reverse pair of pants, name=c, anchor=incoming boundary 1, 
		outgoing lower boundary component 1/.style={draw, solid},  at={(2, 0)}];  
		\pic [tqft/cap, name=d, anchor=outgoing boundary 1, at=(c-incoming boundary 2)];  
		 \begin{scope}[shift={(4,0.32)}, scale=.8, rotate=90] 
		\draw[rounded corners=28pt] (-.07,-.14)--(.07,.14);
		\draw[rounded corners=28pt] (-.07,.14)--(.07,-.14);
		\draw(0,-.2) arc (270:90:.1 and 0.2);
		\draw(0,-.2) arc (270:90:-.1 and 0.2);
		\end{scope}
\end{tikzpicture}
\r)
\]
cf. Figure~\ref{F.Moeb}. 
\end{proof} 
Assume $\Si$ is a closed symmetric surface, considered as a morphism in $\SymRiem$ from the ground ring to the ground ring. 
\begin{cor}\label{cor.calc.ktft} With the notation of Proposition \ref{P.elem.cob}, the morphism \eqref{lift.F.sym}  is given by:
\best
\wt F(\Si)&=& \sum_{\rho=\rho^*} U_{\rho}^{g-1} , \quad \mbox { when $\Si$ is a connected genus $g$ surface and}
\\
\wt F(\Si\sqcup \ov \Si)&=&\sum_{\rho} \la_{\rho}^{g-1}, \quad \mbox { when $\Si \sqcup \ov \Si$ is a $g$-doublet.}
\eest
\end{cor} 
\begin{proof} This follows from Proposition \ref{P.elem.cob} by decomposing the surface $\Si$ into elementary cobordisms. When $\Si=\P^1$, 
\best
\wt F(\P^1) = CU=\sum_{\rho} \la_\rho^{-1}U_\rho = \sum_{\rho=\rho^*}\la_\rho^{-1} U_\rho =  \sum_{\rho=\rho^*} U_\rho^{-1}.
\eest
Similarly, 
\bear\label{torus.split}
\wt F(T^2)= 
\wt F\l( 
\begin{tikzpicture}[baseline={(current bounding box.center)}, scale=1/2, rotate=90, every tqft/.style={transform shape},
		tqft/.cd, 
		cobordism/.style={draw},
		every upper boundary component/.style={draw},
		every lower boundary component/.style={draw},
		]
		\begin{scope}[shift={(2,.75)}, scale=1.2, gray] 
		\draw[rounded corners=28pt] (-.07,-.14)--(.07,.14);
		\draw[rounded corners=28pt] (-.07,.14)--(.07,-.14);
		\draw(0,-.2) arc (270:90:.1 and 0.2);
		\draw(0,-.2) arc (270:90:-.1 and 0.2);
		\end{scope};		
		\pic [tqft, 
		incoming boundary components=0, 
		outgoing boundary components=2, 
		 name=a, at={(1,2)}];    
		\pic[tqft, incoming boundary components=2, outgoing boundary components=2 ,at={(1,0)}];   
		 \begin{scope}[shift={(2,-1)}, scale=2.2, gray] 
		\draw[rounded corners=28pt] (-.07,-.14)--(.07,.14);
		\draw[rounded corners=28pt] (-.07,.14)--(.07,-.14);
		\draw(0,-.2) arc (270:90:.1 and 0.2);
		\draw(0,-.2) arc (270:90:-.1 and 0.2);
		\end{scope};	
		\pic [tqft/cup,at={(1,-2)}];
		\pic [tqft/cup, at={(3,-2)}];		
\end{tikzpicture}
\r) = 
CKU 
\eear
More generally, for a genus $g\ge 1$ symmetric surface $\Si$, 
\best
\wt F(\Si)= CK^gU = \sum_{\rho} \la_\rho^{-1}U_\rho^g U_\rho = \sum_{\rho=\rho^*} U_\rho^{g-1}. 
\eest
Finally, on a $g$-doublet, the morphisms is 
$$
\wt F(\Si\sqcup\ov\Si) = CG^g(1)=\sum_\rho \la_\rho^{g-1},
$$
recovering the  classical theory. 
\end{proof}

\subsection{The category $\SymRiem^L$} \label{S.kcob.categ}
We next construct a simultaneous extension of the categories $\cob^{L_1, L_2}$ and $\kcob\equiv\SymRiem$. Consider the category $\SymRiem^L$ whose 
\begin{itemize} 
	\item   objects are disjoint unions of copies of ${\cal S}=(S^1\sqcup \ov{S^1}, \ep)$, where $\ep$ swaps the two components, and 
	\item   morphisms correspond to isomorphism classes relative boundary of {\em decorated cobordisms} 
\best
W=(\Si, c, L),
\eest
where $\Si$ is an oriented cobordism with a fixed-point free orientation-reversing involution  $c$, extending~$\ep$, and $L$ is a complex line bundle 
over $\Si$, trivialized along the boundary of $\Si$. 
\end{itemize}
 The level zero theory corresponds to a trivial bundle $L$, and defines an embedding:
\bear\label{emb.kcob.in.symL}
\cob\subset \kcob\equiv  \SymRiem \subset \SymRiem^L. 
\eear
The doubling procedure defines an embedding 
\bear\label{emb.cob.in.kcob}
\cob^{L_1, L_2} \subset \SymRiem^L,\quad (\Si, L_1,L_2)\mapsto (\Si\sqcup\ov\Si, c_{|\Si}=id:\Si\ra\ov\Si, L_1\sqcup \ov L_2). 
\eear
The category $\cob^{L_1,L_2}$ has 4 extra generators,  the level $(\pm 1,0), (0,\pm 1)$-caps, besides those of $\cob$, cf. \cite[\S4.3]{bp1}. Similarly,  the generators of the category  $\SymRiem^L$ are those of $\SymRiem$ together with the images of  the $(\pm 1,0), (0,\pm 1)$-caps under \eqref{emb.cob.in.kcob}.  It is also useful to consider the  tubes
 \bear\label{two.tubes.k}
\begin{tikzpicture}[baseline={(current bounding box.center)},scale=1, rotate=90, every tqft/.style={transform shape},
  tqft/.cd, 
  cobordism/.style={draw},
  every lower boundary component/.style={draw},
		every incoming lower boundary component/.style={dashed},
		every outgoing lower boundary component/.style={solid},
		every outgoing upper boundary component/.style={decorate, decoration={markings,
				mark=at position .5 with {\arrow{<}}, },}, 
		every incoming upper boundary component/.style={decorate, decoration={markings,
				mark=at position .5 with {\arrow{<}}, },},  
      ]
 \pic [tqft/cylinder, name=f];    
 \node at ([yshift=30pt] f-outgoing boundary 1) {$(-1, 0)$};
\end{tikzpicture}
\qquad  \mbox{ and } \qquad
\begin{tikzpicture}[baseline={(current bounding box.center)},scale=1, rotate=90, every tqft/.style={transform shape},
  tqft/.cd, 
  cobordism/.style={draw},
  every lower boundary component/.style={draw},
		every incoming lower boundary component/.style={dashed},
		every outgoing lower boundary component/.style={solid},
		every outgoing upper boundary component/.style={decorate, decoration={markings,
				mark=at position .5 with {\arrow{<}}, },}, 
		every incoming upper boundary component/.style={decorate, decoration={markings,
				mark=at position .5 with {\arrow{<}}, },},  
      ]
 \pic [tqft/cylinder, name=f];    
 \node at ([yshift=30pt] f-outgoing boundary 1) {$(0, -1)$};
\end{tikzpicture}
\eear
in $\cob^{L_1, L_2} $, and their images
 \bear\label{two.tubes.s}
\begin{tikzpicture}[baseline={(current bounding box.center)}, scale=1, rotate=90, every tqft/.style={transform shape},
  tqft/.cd, 
  cobordism/.style={draw},
  		every lower boundary component/.style={draw},
  		every incoming lower boundary component/.style={dashed},
		every outgoing lower boundary component/.style={solid},
      ]
 \pic [tqft/cylinder, name=f, anchor= incoming boundary 1];    
  \pic [tqft/cylinder, name=f2, anchor= incoming boundary 1, at=({1,0})];  
 \node at ([yshift=30pt] f-outgoing boundary 1) {$0$};
  \node at ([yshift=30pt] f2-outgoing boundary 1) {$-1$};
\end{tikzpicture}
\qquad \mbox {and} \qquad
\begin{tikzpicture}[baseline={(current bounding box.center)}, scale=1, rotate=90, every tqft/.style={transform shape},
  tqft/.cd, 
  cobordism/.style={draw},
  		every lower boundary component/.style={draw},
  		every incoming lower boundary component/.style={dashed},
		every outgoing lower boundary component/.style={solid},
      ]
 \pic [tqft/cylinder, name=f, anchor= incoming boundary 1];    
  \pic [tqft/cylinder, name=f2, anchor= incoming boundary 1, at=({1,0})];  
 \node at ([yshift=30pt] f-outgoing boundary 1) {$-1$};
  \node at ([yshift=30pt] f2-outgoing boundary 1) {$0$};
\end{tikzpicture}
\eear
in $\SymRiem^L$ under  \eqref{emb.cob.in.kcob}. 

 As in \cite[Theorem 4.1]{bp1},  we obtain the following result.
 \begin{proposition}\label{extKTQFT} A symmetric monoidal functor 
\bear\label{F.on.sym.L}
 	F:\SymRiem^L\lra R\text{mod}
\eear
is uniquely determined by the level $0$ theory and the images $\eta$ and $\ov\eta$ of the level $(-1,0)$ and $(0,-1)$-caps. 
 \end{proposition}
The images 
\bear
A=F\l( 
\begin{tikzpicture}[baseline={(current bounding box.center)}, scale=1, rotate=90, every tqft/.style={transform shape},
  tqft/.cd, 
  cobordism/.style={draw},
  		every lower boundary component/.style={draw},
  		every incoming lower boundary component/.style={dashed},
		every outgoing lower boundary component/.style={solid},
      ]
 \pic [tqft/cylinder, name=f, anchor= incoming boundary 1];    
  \pic [tqft/cylinder, name=f2, anchor= incoming boundary 1, at=({1,0})];  
 \node at ([yshift=30pt] f-outgoing boundary 1) {$0$};
  \node at ([yshift=30pt] f2-outgoing boundary 1) {$-1$};
\end{tikzpicture}
\r) \quad \mbox{ and } \quad 
\bar  A=F\l( 
\begin{tikzpicture}[baseline={(current bounding box.center)}, scale=1, rotate=90, every tqft/.style={transform shape},
  tqft/.cd, 
  cobordism/.style={draw},
  		every lower boundary component/.style={draw},
  		every incoming lower boundary component/.style={dashed},
		every outgoing lower boundary component/.style={solid},
      ]
 \pic [tqft/cylinder, name=f, anchor= incoming boundary 1];    
  \pic [tqft/cylinder, name=f2, anchor= incoming boundary 1, at=({1,0})];  
 \node at ([yshift=30pt] f-outgoing boundary 1) {$-1$};
  \node at ([yshift=30pt] f2-outgoing boundary 1) {$0$};
\end{tikzpicture}
\r)
\eear
of \eqref{two.tubes.s} are called the  \textsf{level-decreasing} operators, and moreover
 $$
 A(x)=\eta\cdot x,\quad \ov A(x) = \ov \eta\cdot x. 
 $$
 If the restriction of \eqref{F.on.sym.L} to the level 0 theory defines a semi-simple KTQFT with idempotent basis $\{v_\rho\}$ then 
 \bear
 A(v_\rho)=\eta_\rho v_\rho \mbox{ and } \bar A(v_\rho)=\ov\eta_\rho v_\rho,\quad  \mbox{ where }  \quad 
 \eta=\sum_\rho \eta_\rho v_\rho  \mbox{ and } \bar \eta=\sum_\rho \ov\eta_\rho v_\rho.
\eear
As in Corollary~\ref{cor.calc.ktft}, then the value of $F$ on a closed connected genus $g$ symmetric surface $\Si$ at level 
$k=c_1(L)[\Si]$ is equal to
\bear\label{Si.k}
F(\Si| L )= CA^{-k}K^g(U)=\sum_{\rho=\rho^*} U_\rho^{g-1}\eta_\rho^{-k}.
 \eear
The value of  $F$ on a $g$-doublet $\Si\sqcup \ov \Si$ with a line bundle $L_1\sqcup L_2$  is similarly equal to
\best\label{Si.doublet}
 F(\Si\sqcup\ov\Si\,|L_1,L_2) =C(A^{-k_1}\bar A^{-k_2}G^g(1))=\sum_\rho \la_\rho^{g-1}\eta_\rho^{-k_1}\bar\eta_\rho^{-k_2}, 
\eest
where $k_1=c_1(L_1)[\Si]$ and $k_2=c_1(L_2)[\ov\Si]$. 
 
 \vskip.2in 
\section{The Klein TQFT induced by the RGW invariants}\label{KTQFTforRGW}
 \vskip.1in 

In this section we use the local RGW invariants   \eqref{def.RGW.can}  to define an extension of a Klein TQFT, i.e. a  functor {\bf RGW} from the category $\SymRiem^L$ described in \S \ref{S.kcob.categ}. This extends the Bryan-Pandharipande TQFT  constructed  from the $GW$ theory for the anti-diagonal action; see \S\ref{S.cxBP}. 

Let $R= \cx(t)((u))$ be the ring of Laurent series in $u$ whose coefficients are rational functions of $t$ and~$d$ be a positive integer. 
Denote by $\cal S=(S^1\sqcup \ov {S^1}, \ep)$ the disjoint union of two copies of a circle with opposite orientations and with the  involution $\ep$ swapping them. To the object $ {\cal S}$ we associate
\bear\label{def.Frob.R} 
{\bf RGW}_d( {\cal S})=H=\bigoplus_{\al\vdash d} R e_\al, 
\eear 
the free module with basis $\{ e_\al\}_{\al\vdash d}$ indexed by partitions $\al$ of $d$. Let 
\best
{\bf RGW}_d( {\cal S}\sqcup \dots \sqcup  {\cal S})=H\otimes \dots \otimes H. 
\eest
To each cobordism $W=(\Si,c, L)$ in $\SymRiem^L$ from $n$ copies of $ {\cal S}$ to $m$ copies of $ {\cal S}$, associate the $R$-module homomorphism 
\bear\label{def.Frob.cob} 
{\bf RGW}_d(W): H^{\otimes n} \ra H^{\otimes m}
\eear
defined by 
\best
e_{\la^1}\otimes \dots \otimes e_{\la^n}\mapsto \sum_{\mu^i\vdash d}  RGW_d(\Si_W|L_W)_{\la^1.. \la^n}^{\mu^1.. \mu^m}  \; e_{\mu^1}\otimes  \dots \otimes  e_{\mu^m}. 
\eest
  Here (i) $\Si_W$ is a closed marked symmetric Riemann surface whose topological type is that of $\Si$ after removing small disks around the pairs of marked points, (ii) the first element in each pair of marked points of $\Si_W$  corresponds to the first copy of $S^1$ in $ {\cal S}=(S^1\sqcup \ov{S^1})$, and (iii) $L_W\ra \Si_W$ is a holomorphic line bundle whose first Chern class corresponds to the Euler class of $L\ra \Si$.  Finally, $RGW_d(\Si_W|L_W)_{\vec \la}$ are the local RGW invariants defined by \eqref{def.RGW.can} and \eqref{defn.RGW.real.rel}, and the indices are raised by  \eqref{modified.metric}. The coefficients are invariant under smooth deformation, thus the assignment \eqref{def.Frob.cob} is well-defined.
 
 \begin{theorem}\label{ktqft} The assignment \eqref{def.Frob.cob} defines a symmetric monoidal functor 
	\bear\label{RGW.functor}
	{\bf RGW}_d: \SymRiem^L \ra R\mathrm{mod}. 
	\eear 
Its restriction to $\kcob$ under \eqref{emb.kcob.in.symL} is a Klein TQFT, while its restriction to $\cob^{L_1, L_2}$ under \eqref{emb.cob.in.kcob} is
\bear\label{rgw=gw}
		{\bf RGW}_d(\Si\sqcup\ov\Si|L_1\sqcup \ov{L}_2)(u,t)=(-1)^{dk_2}{\bf GW}_d(\Si|L_1,L_2)(iu,it).
\eear
	Here $k_i$ is the total degree of $L_i$ and ${\bf GW}_d$
	is the TQFT \eqref{GW.functor} considered by Bryan-Pandharipande (for the anti-diagonal action).  
\end{theorem} 
\begin{proof} By Lemma~\ref{signOm}, coefficients of the assignment \eqref{def.Frob.cob} are invariant under permuting two {\em pairs} of conjugate points of $\Si_W$, thus \eqref{RGW.functor} is symmetric. It is monoidal i.e. 
\best
	 {\bf RGW}_d  (W_1\sqcup W_2) = {\bf RGW}_d(W_1)\otimes {\bf RGW}_d(W_2) 
\eest
because the real moduli space over a disjoint union of {\em Real} curves is the product of the real moduli spaces on each piece, and the index bundle naturally decomposes as the direct sum of the two index bundles. The composition law 
\best
	 {\bf RGW}_d  (W_1\circ W_2)={\bf RGW}_d(W_1)\circ {\bf RGW}_d(W_2)
\eest
holds by \eqref{split.formula.2.1}, cf. \eqref{split.separate} and \eqref{split.nonseparate}.  
 	
When $W$ is a doublet, Corollary \ref{compBP} implies that the restriction of \eqref{GW.functor} to $\cob^{L_1, L_2}$ is the 
	Bryan-Pandharipande TQFT \eqref{GW.functor} modified as stated. In particular \eqref{RGW.functor} takes the identity in 
	$\cob$ to the identity morphism, and therefore \eqref{RGW.functor} is a functor.
\end{proof}

 \vskip.2in 
\section{Solving the theory}\label{Solving}
 \vskip.1in 
 
In this section we show that the  functor ${\bf RGW}_d$ defined by \eqref{RGW.functor} restricts at level 0 to a semi-simple Klein TQFT. We also provide an explicit expression in terms of representation theoretic data of its value on a closed symmetric surface with a line bundle over it, thus solving the local RGW theory. 
\smallskip

Conjugacy classes of the symmetric group $S_d$ are indexed by partitions $\al$ of $d$. If $\rho$ is an irreducible representation of $S_d$, let $\chi_\rho(\al)$ denote the trace of $\rho$ on the conjugacy class $\al$.  

Recall that the level 0 part of $\SymRiem^L$ is naturally identified with $\SymRiem=\kcob$. Then 
\begin{lemma}The restriction of the functor ${\bf RGW}_d$ to $\kcob\subset \SymRiem^L$ determines a semi-simple Klein TQFT with idempotent basis \eqref{idem.basis}. 
\end{lemma}
\begin{proof}
	This is a direct consequence of \cite{bp1} with small modifications as follows. Let $e_\alpha$ be as in (\ref{def.Frob.R}). Define a new basis 
	\begin{equation}\label{idem.basis}
	v_\rho = \frac{\dim \rho}{d!}\sum_\alpha (-t)^{\ell(\alpha)-d}\chi_\rho(\alpha)e_\alpha,
	\end{equation}
 indexed by the irreducible representations $\rho$ of $S_d$.
Note that  
	\bear\label{e=v.basis}
	e_\alpha=(-t)^{d-\ell(\alpha)}\sum_\rho \frac{d!}{\dim \rho}\frac{\chi_ \rho(\al)}{\zeta(\alpha)}v_\rho.
	\eear
The pair of pants product is determined by 
	$$
	RGW_d((0,0)|(0,0))_{\alpha,\beta}^\gamma
	$$
	and by Corollary \ref{compBP} and the last display on p. 35 of \cite{bp1}
	$$
	RGW_d((0,0)|(0,0))_{\al,\beta}^\gamma(t) = GW_d(0|0,0)_{\alpha,\beta}^\gamma(it) = t^{d-\ell(\al)-\ell(\beta)+\ell(\gamma)}\sum_\rho \l(\frac{d!}{\dim \rho}\r)\frac{\chi_\rho(\al)\chi_\rho(\beta)}{\zeta(\al)\zeta(\beta)}.
	$$
As in \cite{bp1} this implies that $\{v_\rho\}$ is an idempotent basis and therefore ${\bf RGW}_d$ is semi-simple.
	\end{proof}
Note that the relation between  $v_\rho$, defined in  \eqref{idem.basis}, and $v_\rho^{BP}$, defined in
 \cite[Equation (20)]{bp1}, is
\bear\label{v.rho=BP}
	v_\rho(t) = v_\rho^{BP}(it).
\eear

As discussed in \S\ref{KTQFT}, the theory is determined by the genus-adding operator $G$, the level-decreasing operators $A$, $\bar A$, the cross-cap $U$, and the involution $\Omega$. Moreover,
\begin{lemma}
In the idempotent basis $\{v_\rho\}$, the genus-adding operator $G$, the  $(-1, 0)$-tube $A$, and the $(0, -1)$-tube $\bar A$ have eigenvalues respectively 
\bear\label{la+eta}
\la_\rho=t^{2d}\l(\frac{d!}{\dim \rho}\r)^{2},\quad  \eta_\rho= t^dQ^{c_\rho/2}\l(\frac{\dimh_Q \rho}{\dim \rho}\r), \quad \ov \eta_\rho=t^dQ^{-c_\rho/2}\l(\frac{\dimh_Q \rho}{\dim \rho}\r). 
\eear
Here $Q=e^u$, $c_\rho$ is the total content of the Young diagram associated to $\rho$, and 
\bear\label{def.dimh}
\dimh_Q \rho = d!\prod_{\square \in \rho}\left(2\sinh\tfrac{ h(\square)u}{2} \right)^{-1} = d!\prod_{\square\in \rho} \l(Q^{\tfrac{h(\square)}{2}}-Q^{-\tfrac{h(\square)}{2}}\r)^{-1},
\eear
where $h(\square)$ denotes the hooklength of the square $\square$ in the Young diagram associated to 
$\rho$. 
\end{lemma}
\begin{proof} By  \eqref{rgw=gw} and \eqref{v.rho=BP}, the relation between the ${\bf RGW}_d$ and ${\bf GW}_d$ is obtained by the change of variables 
$(u,t)\mapsto (iu, it)$ and multiplication by $(-1)^{dc_1(L_2)}$ in both  the standard and the idempotent bases. The result then  follows from \cite[\S7.3]{bp1}.  In particular, \eqref{def.dimh} is related to the quantum dimension defined in \cite{bp1} via
$$
\dimh_Q\rho(u)= (-i)^d\dim_{Q_{BP}}\rho(iu),  \quad \mbox{ where} 
$$
$$
\frac{\dim_{Q_{BP}}\rho}{d!} = \prod _{\square \in \rho}\left(2\sin\tfrac{ h(\square)u}{2} \right)^{-1} = \prod_{\square\in \rho} i\l({Q_{BP}}^{\tfrac{h(\square)}{2}}-{Q_{BP}}^{-\tfrac{h(\square)}{2}}\r)^{-1}
$$ 
and ${Q_{BP}}=e^{iu}$. 
\end{proof}

It remains to determine $U$ and $\Omega$ in the idempotent basis. 
\begin{prop}\label{L.Om=conj} The involution $\Omega$ is given by
\bear\label{Om=conj}
	\Om (e_\al)= (-1)^{d-\ell(\al)} e_\al  \quad \mbox{and }   \quad \Om(v_\rho)= v_{\rho'} 
\eear	
in the standard basis $\{e_\al\}$ and in the idempotent basis $\{v_\rho\}$, respectively. Here  $\rho'$ denotes the conjugate representation.  
\end{prop} 
 
\begin{proof}  Consider the moduli space of real maps into the doublet corresponding to $\Om$.  It is the same as the moduli space of real maps into the doublet associated to the level 0 tube (the identity), 
except for the change $x_2^+\leftrightarrow x_2^-$  of the order within the pair of marked points in the target corresponding to the outgoing boundary.  Lemma~\ref{signOm} then implies the first equality. In the idempotent basis \eqref{idem.basis}  
	\best
	\Omega(v_\rho)=\frac{ \dim \rho} {d!}  \sum_\al (-it)^{\ell(\al)-d}\chi_\rho(\al) \Omega(e_\al) = \frac{ \dim \rho} {d!}  \sum_\al (-it)^{\ell(\al)-d}\chi_{\rho'}(\al) e_\al =v_{\rho'}, 
	\eest
	where the second equality holds since $\chi_{\rho'}(\al)=(-1)^{d-l(\al)}\chi_{\rho}(\al)$. 
	\end{proof}
Note that \eqref{Omeq} and \eqref{Om=conj} imply that the coefficients $U_\rho$ vanish unless $\rho= {\rho'}$. If $\rho=\rho'$ then \eqref{Omeq}  and \eqref{la+eta} imply that $U_\rho =\pm t^d\frac{d!}{\dim \rho}$ determining it up to a sign. Proposition~\ref{CC} below calculates $U$ directly, independent of these considerations, including the signs. The signed Frobenius-Schur indicator, defined  in \S\ref{SFSsection}, plays a crucial role in this calculation.

\subsection{The level 0 cross-cap $U$} Consider next the level 0 cross-cap $U$ corresponding to \eqref{F.double.xcap}. Its coefficients in the standard basis are obtained from the RGW invariants of a sphere with 2 marked points, real structure $c(z)=-1/\ov z$, and a trivial line bundle. 

 Before we proceed, it is convenient to make the following definition. For a partition $\la$ of $d$, let 
\bear\label{sq.la}
sq(\lambda) 
\eear
denote the partition of $d$ obtained from $\lambda$ by splitting all of the even parts of $\lambda$ into two equal parts e.g. $sq(4,3,3,2,1) = (2,2,3,3,1,1,1)$. This is motivated by the fact that if $g\in S_d$ has cycle type $\la$, then $g^2$ has cycle type $sq(\la)$.  Recall  that the sign morphism on $S_d$ descends to the conjugacy class
\bear\label{sign.g}
\sign(g)= (-1)^{s(g)} = (-1)^{d-\ell(\la)}= \sign(\la),
\eear
where $s(g)$ is the parity of the permutation $g\in S_d$ and $\la$ is its cycle type. This is also the parity of the number of even parts of $\la$ and in particular, $\sign(sq(\la))=+1$.
 
 We start with the following combinatorial identity, which uses the notation \eqref{p-la}; see also \eqref{level0caps}.

\begin{lemma}\label{L.triv.covers} For any partition $\al\vdash d$, the coefficient $r_\al$ of the monomial $p_\al$ in the expansion below is given by
\bear\label{r=trivial.cover}
r_\al=\l[\exp \l( \sum_{d \; \mathrm{odd}} \frac{1}{d} p_d -\frac 12 \sum_{d=1}^\infty \frac{1}{d} (p_d)^2  \r)\r]_{p_\al} 
	= \sum_d \sum_{\lambda\vdash d \atop{sq(\lambda)=\al}} \frac{(-1)^{d-\ell(\lambda)}}{\zeta(\lambda)}.
\eear  
 In particular, $r_\al$ vanishes unless $\al$ has an even number of even parts.
 \end{lemma} 
\begin{proof} The coefficient $r_\al$ on the  LHS of \eqref{r=trivial.cover} is the sum over all possible ways of decomposing $\al$ into a partition $a$ containing only odd elements and 2 copies of a partition $b$:
\bear\label{RGW-0-cross-cap}
\sum_\al r_ \al p_\al = \sum_{a, b } \prod_{k\; \rm odd} \frac {p_k^{a_k}} {a_k! k^{a_k}} \prod_{k} \frac {(-1)^{b_k}p_k^{2b_k}} 
{ 2^{b_k}b_k! k^{b_k}},
\eear
where $a_k, b_k$ denote the multiplicities of $k$ in the partitions $a$ and $b$ respectively. Every such decomposition 
$\al=a\sqcup b\sqcup b$ corresponds to a partition $\la=a\sqcup (2b)$, where $(2b)$ denotes the partition obtained from $b$ by multiplying by 2 each of its parts;  
in particular, $sq(\la)=\al$ and $\sum_k b_k\equiv d-\ell(\la) \mod 2$.  Therefore \eqref{RGW-0-cross-cap} becomes
$$
\sum_{\al} r_\al p_\al = \sum_\al\Big(\sum_{\lambda\vdash d \atop sq(\lambda)=\al} \frac{(-1)^{d-\ell(\lambda)}}{\zeta(\lambda)}\Big)p_{\al}.
$$
\end{proof}
Combined with Proposition~\ref{crosscap} , Lemma \ref{L.triv.covers} implies that the invariant $RGW(0|0)_\al$ is equal to 
\bear\label{rgw00}
RGW(0|0)_\al= r_\al= \sum_d \sum_{\lambda\vdash d \atop{sq(\lambda)=\al}} \frac{(-1)^{d-\ell(\lambda)}}{\zeta(\lambda)}.
\eear
The coefficients $RGW(0|0)^\al$ of $U$ in the standard basis $\{ e_\al\}$ are obtained by raising the indices in \eqref{rgw00} via \eqref{modified.metric}. Combinatorial considerations then allow us to determine the coefficients of
\bear\label{U.in.bases}
U=\sum_\al RGW(0|0)^\al e_\al= \sum_\rho U_\rho v_\rho 
\eear
in the idempotent basis $\{ v_\rho\}$.

 Decompose the set of partitions $\la$ of $d$ into even and odd according to the parity of $d-\ell(\la)$, cf. \eqref{sign.g}.  
  
 \begin{prop}\label{CC} The level 0 cross-cap \eqref{U.in.bases} is equal to the sum over self-conjugate irreducible representations of $S_d$
\bear\label{U=sum}
 &&U= \sum_{\rho=\rho'} \ep_\rho t^d \frac{d!}{{\dim\rho}}v_\rho, 
 \quad \mbox{where }
\\
\label{ep=} 
 &&\ep_\rho=(-1)^{o(\rho)}\quad \mbox{ and} \quad o(\rho) =\sum_{  \beta\vdash d \atop \beta\; \mathrm{odd}}\frac{\chi_\rho(sq(\beta))}{\zeta(\beta)}.
 \eear
 The expression $o(\rho)$ takes values $0$ or $1$ on a self-conjugate irreducible representation $\rho$.
 	\end{prop}
 \begin{proof} By \eqref{U.in.bases}, \eqref{rgw00} and \eqref{e=v.basis} 
 	\begin{gather*}
 	U=\sum_\al RGW(0|0)^\al e_\al = \sum_\al \Big(\sum_{sq(\lambda)=\al} \frac{(-1)^{d-\ell(\lambda)}}{\zeta(\lambda) 
	t^{\ell(\al)}}\Big)\zeta(\al) t^{2l(\al)} e_\al =
	\\
 	=\sum_\al \Big(\sum_{sq(\lambda)=\al} \frac{(-1)^{d-\ell(\lambda)}}{\zeta(\lambda)}\Big)\zeta(\al) t^{\ell(\al)}
	 \Big((-t)^{d-\ell(\alpha)}\sum_\rho \frac{d!}{\dim \rho}\frac{\chi_\rho(\alpha)}{\zeta(\alpha)}v_\rho
 	\Big)=
	\\
 	=\sum_\rho \Big(\sum_\al \sum_{sq(\lambda)=\al} (-1)^{d-\ell(\lambda)}\frac{\chi_\rho(\al)}{\zeta(\lambda)}\Big) \frac{d!t^d}{\dim \rho} v_\rho
 		\end{gather*}
 In the last equality we used the fact that  for $\al=sq(\la)$ the parity of $d$ and $\ell(\al)$ is the same.  It remains to show that the expression in the parenthesis  is given by \eqref{ep=}.  For this we use the following combinatorial identity
 \bear\label{sum.sq=sum.self.conj}
	\sum_{\la\vdash d \atop sq(\la)=\alpha}(-1)^{d-\ell(\la)} \frac{\zeta(\alpha)}{\zeta(\la)}= 
	\sum_{ \rho=\rho'} \ep_\rho \chi_\rho(\alpha)
\eear
cf. Lemma \ref{SFS}, which is of independent interest and whose proof is deferred to \S \ref{SFSsection}. Then 
 	\begin{gather*}
 	 \sum_\al\sum_{sq(\lambda)=\al}(-1)^{d-\ell(\lambda)} \frac{\chi_\rho(\al)}{\zeta(\lambda)}= 
	 \sum_\al \sum_{\mu=\mu'}\ep_{\mu}  \chi_{\mu}(\al)\frac{\chi_\rho(\al) }{\zeta(\al)}=
 	 \sum_{\mu=\mu'}\ep_{\mu} \sum_\al \frac{\chi_{\mu}(\al)\chi_\rho(\al)}{\zeta(\al)}=
	 \begin{cases}\ep_\rho &\text{if } \rho=\rho',\\ 0&\text{otherwise}.\end{cases}
 	\end{gather*}
 The result follows.
 	\end{proof}

 The next lemma provides a simpler expression for the sign $\ep_\rho$ appearing in \eqref{U=sum}. 
 \begin{lemma}\label{CCsign}
 	Let $\rho$ be an irreducible representation of $S_d$, $r(\rho)$ the length of the main diagonal of its Young diagram, and $\ep_\rho$ be as \eqref{ep=}. Then
 \bear\label{ep=2}
 	\ep_\rho= (-1)^{\frac{d-r(\rho)}{2}}.
 \eear
 \end{lemma}
 \begin{proof} Let $x=(x_1,\dots, x_n)$.
 	The power sum functions $p_k(x)$ and the Schur functions $s_\rho(x)$ are related by
 	$$
 	s_\rho(x)=\sum_\la \frac{\chi_\rho(\la)}{\zeta({\la})}p_\la(x)\qquad\text{and}\qquad p_\la(x)=\sum_\rho \chi_\rho(\la) s_\rho(x)
 	$$
 	and form basis for the space of symmetric polynomials on $n$ variables whenever $n>d$. By 
	\cite[9c, p79]{mac}, we have
 	$$
 	\sum_{\rho=\rho'} (-1)^{(d+r(\rho))/2}s_\rho(x_1,\dots,x_n) = \prod_i(1-x_i)\prod_{i<j}(1-x_ix_j),
 	$$
 	where $\rho'$ denotes the conjugate representation.
 	This equality follows from Weyl's identity for $B_n$.
 	
 	\smallskip
 	By Lemma \ref{L.triv.covers},
 	\bear\label{sum.r.alpha=exp}
 	\sum_\al \sum_{\lambda \atop sq(\lambda)=\al} \frac{(-1)^{d-\ell(\lambda)}}{\zeta(\lambda)}
 	 p_\al  = \exp \Big( \sum_{d \; \mathrm{odd}} \tfrac{1 } d p_d -
 	\tfrac 12 \sum_{d=1}^\infty \tfrac{1} d (p_d)^2  \Big). 
 	\eear
 	Substitute $p_d=p_d(x)$, and consider first the LHS of \eqref{sum.r.alpha=exp}:
 	\begin{gather*}
 	\sum_\al \sum_{\lambda \atop sq(\lambda)=\al} \frac{(-1)^{d-\ell(\lambda)}}{\zeta(\lambda)}
 	p_\al(x)=\sum_\al \sum_{\lambda \atop sq(\lambda)=\al} \frac{(-1)^{d-\ell(\lambda)}}{\zeta(\lambda)}
 	\sum_{\rho}\chi_\rho(\al)s_\rho(x)=\\
 	=
 	\sum_\rho \Big(\sum_\al \sum_{\lambda \atop sq(\lambda)=\al} \frac{(-1)^{d-\ell(\lambda)}}{\zeta(\lambda)}
 	\chi_\rho(\al)\Big)s_\rho(x)=
 	\sum_{\rho=\rho'} \ep(\rho) s_\rho(x).
 	\end{gather*}
 For the RHS of  \eqref{sum.r.alpha=exp}, we start with  
 	$$
 	 \sum_{d \; \mathrm{odd}} \tfrac{1} d p_d(x)=
 	  \sum_{d \; \mathrm{odd}} \tfrac{1 } d \sum_i x_i^d=
 	  \sum_i \log \l(\frac{1+x_i}{1-x_i}\r)^{1/2}
 	$$
 	and
 	$$
 	 -
 	\tfrac 12 \sum_{d=1}^\infty \tfrac{1} d (p_d(x))^2=
 	 -
 	\tfrac 12 \sum_{d=1}^\infty \tfrac{1} d \Big(\sum_i x_i^d\Big)^2=
 	 -
 	\tfrac 12 \sum_{d=1}^\infty \tfrac{1} d \sum_{i,j}x_i^d x_j^d 
 	 =
 	 \sum_{i,j}\log (1-x_ix_j)^{1/2}.
 	$$
 	Therefore,
 	\begin{gather*}
 	 \exp \l( \sum_{d \; \mathrm{odd}} \tfrac{1 } d p_d(x) -
 	\frac 12 \sum_{d=1}^\infty \tfrac{1} d p_d(x)^2  \r)=
 	\prod_i \Big(\frac{1+x_i}{1-x_i}\Big)^{1/2} \prod_{i,j}(1-x_ix_j)^{1/2} =\\
 	\prod_i \l(\frac{1+x_i}{1-x_i}\r)^{1/2}\prod_{i<j}(1-x_ix_j)\prod_{x_i=x_j}(1-x_ix_j)^{1/2} = \prod_i(1+x_i)\prod_{i<j}(1-x_ix_j).
 	\end{gather*}
 	Since $s_\rho(-x)=(-1)^d s_\rho(x)$ then
 	$$
 	\sum_{\rho=\rho'} (-1)^d\ep(\rho) s_\rho(x)=\sum_{\rho=\rho'} \ep_\rho s_\rho(-x)=\prod_i(1-x_i)\prod_{i<j}(1-x_ix_j)=\sum_{\rho=\rho'} (-1)^{(d+r(\rho))/2}s_\rho(x).
 	$$
 	But $s_\rho(x)$ is a basis so \eqref{ep=2} holds. 
	 \end{proof}
 
Lemma \ref{CCsign} and Proposition \ref{CC} then imply:
\begin{cor}\label{CCsign2} In the idempotent basis, the level 0 cross-cap $U$ is given by
\bear\label{U.idemp.basis}
	U=\sum_{\rho\vdash d
	\atop \rho=\rho'} (-1)^{(d-r(\rho))/2}\; t^d \frac{d!}{\dim \rho}v_\rho,
\eear
where $r(\rho)$ is the length of the main diagonal of the Young diagram of $\rho$.  
\end{cor}
 
\subsection{Local Calabi-Yau over a sphere}\label{S.local.CY.sphere} Consider next the local RGW invariants associated to the Real Calabi-Yau 3-fold $Y$ defined by \eqref{L.plus.twL} for $\Si=\P^1$   and $L=\O(-1)$. 
Note that $Y$ is biholomorphic to the total space of $\O(-1)\oplus \O(-1)\ra \P^1$, thus contains no holomorphic curves besides multiple covers of the zero section. In particular, the only real curves in $Y$ are the multiple covers of the zero section 
$\Si\subset Y$. Moreover, the discussion in the paragraph above \eqref{ind.E} implies that the zero section in $L\oplus c^*\ov L$ with $L= \O(-1)$ is super-rigid and therefore \eqref{defn.Z.real} is precisely the contribution of its multiple covers to the Real Gromov-Witten invariants of $Y.$
 
\begin{theorem}\label{T.CY.sphere} The generating function for the RGW invariants is
\bear \label{-1.cross.cap.exp}
	\sum_d RGW_d(0|-1)q^d
	&=&\sum_{\rho=\rho'}(-1)^{\frac 12 (|\rho|-r(\rho))} \prod_{\square \in \rho} \l( 2\sinh \tfrac{ h(\square)u}2 \r)^{-1} q^{|\rho|}
\\  \label{-1.cross.cap}
&=&
 \exp\l(\sum_{k\; \mathrm{odd}}  \tfrac 1k\l(2 \sinh \tfrac{ku}2\r)^{-1} q^{k}-\tfrac 12  \sum_k  \tfrac 1k\l(2 \sinh \tfrac{ku}2\r)^{-2} q^{2k} \r).
\eear
In particular, the generating function for the connected real invariants is
\bear \label{-1.cross.cap.conn}
	\sum_d CRGW_d(0|-1)q^d =\sum_{k\; \mathrm{odd}}  \tfrac 1k\l(2 \sinh \tfrac{ku}2\r)^{-1} q^{k}.
\eear
\end{theorem} 
\begin{proof} Recall that $RGW(0|-1)=CAU$  cf. \eqref{Si.k} and the coefficients of $C$, $A$ and $U$ in the idempotent basis $\{ v_\rho\}$ are given by \eqref{la+eta} and \eqref{U.idemp.basis}. Since the content $c_\rho$ of a self-conjugate partition vanishes, this gives \eqref{-1.cross.cap.exp}.

Next, the invariants $RGW$ are related to the connected and doublet invariants by
	$$
	\sum_d RGW_d(0|-1)q^d =\exp\l( \sum_d CRGW_d(0|-1)q^d +\sum_d DRGW_{2d}(0|-1)q^{2d} \r)
	$$
cf. \eqref{rgw=exp}. 
Corollary \ref{C.Z.part.anti-diag-balanced} relates the doublet invariants to the connected $GW$ invariants and along with the classical  calculation of \cite{FP} we obtain 
	$$
	\sum_d DRGW_{2d}(0|-1)(u)q^{2d} = \tfrac 12 \sum_d GW^{conn}_{d}(0|-1,-1)(iu)q^{2d}= - \tfrac 12  \sum_k  \tfrac 1k\l(2 \sinh \tfrac{ku}2\r)^{-2} q^{2k}. 
	$$
It thus remains to prove \eqref{-1.cross.cap}. Substituting $p_d = (2\sinh\frac{du}{2})^{-1}$ in \eqref{r=trivial.cover} we obtain
	$$
	\exp\Big(\sum_{k\; \mathrm{odd}}  \tfrac 1k\l(2 \sinh \tfrac{ku}2\r)^{-1} q^{k}-\tfrac 12  \sum_k  \tfrac 1k\l(2 \sinh \tfrac{ku}2\r)^{-2} q^{2k} \Big)=
	\sum_\al\Big(\sum_{\la\atop sq(\lambda)=\al}\frac{(-1)^{d-l(\lambda)}}{\zeta(\lambda)}\Big)
	\frac{q^{|\al|}}{\ma\prod_{i}2\sinh \tfrac{\al_iu}2}	
	$$ 
	 Using \eqref{sum.sq=sum.self.conj} the coefficient of $q^d$ is equal to 
	$$
	\sum_{\al\vdash d}\sum_{\rho=\rho'}\ep_\rho\frac{\chi_\rho(\al)}{\zeta(\al)} \frac{1}{\ma\prod_{i}2\sinh\tfrac{\al_iu}2}
	=
	\sum_{\rho=\rho'}\ep_\rho Q^{d/2} \sum_\al\frac{\chi_\rho(\al)}{\zeta(\al)} \frac{(-1)^{\ell(\al)}}{\ma\prod_{i}(1-Q^{\al_i})}, 
	$$
	with $Q=e^u$. But  $\ep_\rho=(-1)^{\frac{d-r(\rho)}{2}}$ by Lemma \ref{CCsign}, and the sum over $\al$ in the above expression equals the Schur function $s_{\rho'}$ for the conjugate representation $\rho'$ times $(-1)^d$. Since
	$$
	s_{\rho'}=Q^{c_{\rho'}-d/2}(-1)^d\frac{1}{\ma\prod_{\square\in \rho'}(Q^{h(\square)/2}-Q^{-h(\square)/2})}
	$$ 
	and $\rho=\rho'$, $c_\rho=0$, we obtain  \eqref{-1.cross.cap}.
\end{proof}

\begin{rem}
	Note the similarity between \eqref{-1.cross.cap.conn} and the equivariant localization computation of the open GW invariants considered in \cite[Theorem 7.2]{KL} for the weight $a=0$ (see also \cite[\S6]{psw}). In this case the contributions of the graphs computing the invariants in the real and open case match in odd degree; for the real invariants in even degree, the graphs come in pairs depending on the type of the real structure, and there is a cancelation between open and crosscap contributions cf \cite[\S3.3]{Wal}. The $\sin$ vs $\sinh$ difference comes from the difference in orientation conventions, cf.  \cite[\S3.1]{GZ2}. 
	\end{rem}
\begin{rem} The right hand side of  \eqref{-1.cross.cap} has another expansion besides \eqref{-1.cross.cap.exp}. By   \cite[(4.5)]{Nak} with $t_1=-t^{-1}_2=e^u$
	$$
	\exp\Big(\sum_{k\; \mathrm{odd}}  \tfrac 1k\l(2 \sinh \tfrac{ku}2\r)^{-1} q^{k}-\tfrac 12  \sum_k  \tfrac 1k\l(2 \sinh \tfrac{ku}2\r)^{-2} q^{2k} \Big)
	=\sum_\rho\frac{(-1)^{a(\rho)}q^{|\rho|}}{\ma\prod_{\square\in\rho}2\sinh \tfrac{h(\square)u}2},	
	$$
	where $a(\rho)$ is the sum of the arm lengths of $\square\in \rho$. Note that this sum is over all representations, not only self-conjugate ones, and the  signs $(-1)^{a(\rho)}$ and $\ep_\rho=(-1)^{\frac{d-r(\rho)}{2}}$ are different in general. Nevertheless, the two sums are equal. 
\end{rem}

\subsection{Local Calabi-Yau over a torus} Consider next the local RGW invariants associated to the Real CY 3-fold $Y$ given by \eqref{L.plus.twL} for $\Si$  a torus (elliptic curve) and $L$   a degree 0 holomorphic line bundle.  When $L$ is not a torsion element in the Picard group, its total space contains no holomorphic curves other than the multiple covers of the zero section. Therefore as in \S\ref{S.local.CY.sphere}, the zero section of $Y$ is super-rigid and \eqref{defn.Z.real} is the contribution of its multiple covers to the Real Gromov-Witten invariants of the 3-fold $Y.$
\begin{theorem} The generating function for the $RGW$ invariants is
\bear\label{disconn.tori}
	\sum_d RGW_d(1|0)q^d=\sum_{\rho=\rho'} q^{|\rho|}   
	=\exp\l(\sum_d(-1)^{d-1}\sum_{k\; \mathrm{odd}}  \tfrac 1k  q^{dk}+\tfrac 12  \sum_{d,k}  \tfrac 1k q^{2dk} \r).
\eear
Moreover, the generating function for the connected RGW invariants is
\bear\label{conn.tori}
	\sum_d CRGW_d(1|0)q^d =\sum_d(-1)^{d-1}\sum_{k\; \mathrm{odd}}  \tfrac 1k q^{dk}.
\eear
\end{theorem} 
\begin{proof} By \eqref{torus.split}, 
$$
RGW_d(1|0)=CKU=\ma\sum_{\rho\vdash d \atop \rho=\rho'} 1
$$ 
giving the   first equality in \eqref{disconn.tori}. Note that the generating function of the self-conjugate partitions  is
$$
	\sum_{\rho=\rho'}q^{|\rho|}= \prod_d \frac{1}{1+(-q)^d}.
$$
As in the proof of Theorem~\ref{T.CY.sphere}, relation \eqref{rgw=exp} and Corollary \ref{C.Z.part.anti-diag-balanced} imply
$$
\sum_d RGW_d(1|0)q^d =\exp\l( \sum_d CRGW_d(1|0)q^d +\sum_d DRGW_{2d}(1|0)q^{2d} \r) \quad \mbox{ and }
$$
$$
\sum_d DRGW_{2d}(1|0)q^{2d} = \tfrac 12\sum_d GW^{conn}_d(1|0,0)q^{2d}=  \tfrac 12  \sum_{d,k}  \tfrac 1k  q^{2dk}. 
$$
  In the last equality we used the classical calculation 
\best
\exp\l( \sum_d  GW^{conn}_d(1|0,0) q^d\r)=\sum_d GW_d(1|0,0)q^{d} =\sum_{\rho} q^{|\rho|}=\prod_d \frac{1}{1-q^d},
\eest 
cf. \cite[Corollary~7.3]{bp1}. Since
	$$
 \exp\l(\sum_d(-1)^{d-1}\sum_{k\; \mathrm{odd}}  \tfrac 1k  q^{dk}+\tfrac 12  \sum_{d,k}  \tfrac 1k q^{2dk} \r)=\prod_d\l(\frac{1-(-q)^d}{1+(-q)^d}\r)^{1/2}\l(\frac{1}{1-q^{2d}}\r)^{1/2},
 $$
	 we obtain  the second equality in \eqref{disconn.tori} and therefore \eqref{conn.tori}.
\end{proof}

\begin{rem}
The connected invariants \eqref{conn.tori} can also be computed directly. By Lemma~\ref{L.level.0} and  \S\ref{S.dep.data}, it suffices to consider only real (unramified) covers of a torus without fixed locus by a torus; passing to the universal cover reduces this to a {\em signed} count of sub-lattices that are invariant under a lift of the complex conjugation. In fact, if we fix two separating crosscaps in the target, their inverse image consists of $d+d$ circles, each winding around the crosscap $k$ times. One can show that exactly two of the circles are preserved by the involution in the domain (thus are crosscaps) and in particular $k$ must be odd; $d$ could be either even or odd. If $d$ is odd, the two crosscaps in the domain 
map to the two crosscaps in the target; otherwise they map to a single crosscap in the target. Such a cover has degree $dk$, $k$ automorphisms, and its sign is determined by whether or not the induced orientation on the crosscaps on the domain from the orientation on the crosscaps on the target coincides with the boundary orientation when we cut along the domain crosscaps (since the canonical orientation corresponds to having the crosscaps oriented in this manner). In particular, when $d$ is odd we have +1 and when $d$ is even -1, therefore contributing $ (-1)^{d-1}\tfrac 1k q^{dk}$ to \eqref{conn.tori}.
\end{rem}

\subsection{The general case}  Consider next a local Real 3-fold  $(L\oplus c^*\ov L, c_{tw})\ra \Si$ over a connected surface. 
\begin{theorem}\label{T.str.gen} Assume $\Si$ is a connected genus $g$ symmetric surface and $L\ra \Si$ a holomorphic line bundle with $c_1(L)=k$.  Then the degree $d$ local RGW  invariants are equal to 
\bear\label{R-g.part}
	RGW_d (g| k) =\sum_{\rho=\rho'} \l( (-1)^{\frac{d-r(\rho)}{2}}t^d  \frac{d!}{\dim \rho}\r)^{g-1}   \l( t^d\; \frac{\dimh_Q \rho}{\dim \rho}\r)^{-k}.
\eear
Here the sum is over self-conjugate partitions $\rho$ of $d$, $r(\rho)$ is the rank \eqref{rank.la}, and $\dimh_Q \rho$ is \eqref{def.dimh}.
\end{theorem}
\begin{proof} The result follows as before from $RGW(g|k)=CK^{g}A^{-k}U$, cf.  \eqref{Si.k}. Note that when $g>1$, $RGW(g|k)$ can also be obtained as the trace of the composition of the diagonal operators $K^{g-1}A^{-k}$. 
 \end{proof}
 
\begin{cor} In the (Real) Calabi-Yau case, the contribution becomes 
	\bear\label{rgw=CY}
	RGW_d(g| g-1)&=&\sum_{\rho=\rho'}  \Big( (-1)^{\frac{d-r(\rho)}{2}}  \prod_{\square\in \rho} 2\sinh \tfrac{h(\square)u} 2 \Big)^{g-1}.
	\eear
\end{cor}
In the equivariant CY case, the (complex) GW invariants defined in \cite{bp1} are equal to
\bear\label{gw=CY}
GW_d(g| g-1, g-1)&=&\sum_{\rho\vdash d}  \Big( \prod_{\square\in \rho} 2\sin \tfrac{h(\square)u} 2 \Big)^{2g-2}
\eear
cf. \cite[Corollary~7.3]{bp1}. Note that here the sum is over all partitions of $d$, not just self-conjugate ones.

\begin{rem} In the level 0 case, the proof of Lemma~\ref{L.level.0} implies that $RGW_d(g|0)$ is  a {\em signed} count of degree $d$ unramified real covers of a genus $g$ Riemann surface (i.e. a real Hurwitz number), and \eqref{R-g.part} becomes: 
	\bear\label{rgw=level.0}
	RGW_d(g|0)&=&\sum_{\rho=\rho'}  \l( (-1)^{\frac{d-r(\rho)}{2}}t^d  \frac{d!}{\dim \rho}\r)^{g-1}. 	
	\eear
In contrast, the combinatorial count of real Hurwitz covers gives rise to a different KTQFT,  cf. \cite{AN, MP};  in this case, all covers count positively and the number of unramified real covers of a symmetric genus~$g$ surface $(\Si,c)$ with {\em empty} real locus is equal to
\best
H^\R_{(\Si,c)}  = \sum_{\rho \vdash d} 
\l( \frac{d!}{\dim \rho} \r)^{g-1},
\eest
where the sum is over {\em all} partitions $\rho$ of $d$. For this combinatorial KTQFT, the involution $\Omega$ is trivial and the   coefficients $U_\rho$ of the crosscap are equal to the {\em positive} square roots of the structure constants.  However, unlike RGW, $H^\R_{(\Si,c)}$ depends on the real structure $c$.
  
\end{rem}

\setcounter{equation}{0}
    \vskip.2in 
\section{The local Real Gopakumar-Vafa formula}\label{LocalGV}
    \vskip.1in

We are now ready to prove the real Gopakumar-Vafa formula  (cf. \cite[\S5]{Wal}) for the local RGW invariants   defined in this paper. The  local GV conjecture in the classical setting, proved in \cite[Proposition 3.4]{ip-gv}, states that the connected GW invariants defined in \cite{bp1} have the following structure:
 \bear\label{Cx.GV}
   	\sum_{d=1}^\infty \sum_\chi GW^{conn}_{d, \chi}(g|g-1, g-1)u^{-\chi}q^d = \sum_{d=1}^\infty\sum_{h} n^{\cx}_{d,h}(g)\sum_{k=1}^\infty\tfrac{1}{k}(2\sin(\tfrac{ku}{2}))^{2h-2}q^{kd},
  \eear	
 where the coefficients $n^{\cx}_{d,h}(g)$,  called the local BPS states, satisfy (i) $n^{\cx}_{d,h}(g)\in \Z$ and (ii)  for each $d$, 
	$n^{\cx}_{d,h}(g)=0$ for large $h$. 
 
In the real setting, the local real GV formula takes the following form. 
   \begin{theorem}[Local real GV formula]\label{T.local.GV}Fix a genus $g$ symmetric surface $\Si$ and consider the local real Calabi-Yau  
  3-fold $(L\oplus c^*\ov L, c_{tw})\ra \Si$.  Then the generating function for the connected local RGW invariants has the following structure: 
    	\bear\label{R.GV}
   	\sum_{d=1}^\infty \sum_{h=0}^\infty CRGW_{d, h}(g|g-1)u^{h-1}q^d = \sum_{d=1}^\infty\sum_{h=0}^\infty n^{\R}_{d,h}(g)
	\sum_{k \;\mathrm{odd}\atop k>0}\tfrac{1}{k}(2\sinh(\tfrac{ku}{2}))^{h-1}q^{kd},
   	\eear
   	where the coefficients $n^{\R}_{d,h}(g)$ satisfy (i) (integrality) $n^{\R}_{d,h}(g)\in \Z$, (ii) (finiteness) for each $d$, 
	$n^{\R}_{d,h}(g)=0$ for large $h$,  and  (iii) (parity) $n^\R_{d,h}(g)=n^\cx_{d,h}(g)\mod 2$. Moreover,
   	\begin{enumerate}[(a)]
   		\item for $g=0$, $n^{\R}_{d,h}(0)=1$ when $d=1$ and $h=0$ and vanish otherwise. 
   		\item for $g=1$, $n^{\R}_{d,h}(1)=(-1)^{d-1}$ when $h=1$ and vanish otherwise. 
		\item for any $g\ge 0$, $n^\R_{1, h}(g)=1$ when $h=g$ and vanish otherwise.
   	\end{enumerate} 
   \end{theorem} 
   \begin{proof} The results for the genus $g\le 1$ cases are obtained in \eqref{-1.cross.cap.conn} and \eqref{conn.tori}. So it suffices to assume $g\ge 2$.  For  every integer  $n\ge 0$, let 
   	\bear\label{H.n}
   	H_n(u, q)=\sum_{k\;\rm odd }\tfrac{1}{k}(2\sinh(\tfrac{ku}{2}))^{n} q^k= u^n q(1 + \dots ).
   	\eear
	Then $\{ q^{-1} H_n(u, q^d)\}_{n\ge 0, d\ge 1} $ is a basis of the power series in $u$ and $q$. In particular, for $g\geq 2$, there exists an expansion of the connected invariant in the form \eqref{R.GV}, for some some coefficients $n^{\R}_{d,h}(g)\in \Q$, with $h\ge g$ (because there are no covers of a genus $g$ curve by a lower genus curve). 
\smallskip

Denote by $Z_g=Z_g(u)$ the $RGW$ invariants of the local real Calabi-Yau 3-fold and by $C_g=C_g(u)$ and   $D_g=D_g(u)$ the connected and  the doublet invariants, respectively. Then $Z_g=\exp(C_g+ D_g)$ cf. \eqref{rgw=exp}.

Corollary \ref{C.Z.part.anti-diag-balanced}, relating the doublet and the GW invariants of \cite{bp1}, and \eqref{Cx.GV} imply 
 \best
  \sum_{d=1}^\infty D_g(u)q^{2d} &=& \tfrac{1}{2}\sum_d   GW_d^{conn}(iu)q^{2d} =\quad  
	\\
   	&=&\tfrac 12
   	\sum_{d=1}^\infty\sum_{h> 0} n^{\cx}_{d,h}(g)(-1)^{h-1}\sum_{k=1}^\infty\tfrac{1}{k}(2\sinh(\tfrac{ku}{2}))^{2h-2}q^{2kd},
  \eest
where $n^{\cx}_{d,h}(g)$  are integers  and have the finiteness property. But $Z_g=\exp(C_g+ D_g)$, so combined with  \eqref{R.GV} this gives:
\bear\label{Z=exp.mess} 
Z_g=  \exp\Big(\sum_{d, h>0} n^{\R}_{d,h}(g)\sum_{k\;\rm odd }\tfrac{1}{k}f(Q^k)^{h-1}q^{kd}
   	+ \tfrac 12
   	\sum_{d, h>0}  n^{\cx}_{d,h}(g)\sum_{k>0}\tfrac{1}{k}F(Q^{2k})^{h-1}q^{2kd}\Big),  
\eear
where 
	$$
   	f(Q)=Q-Q^{-1}\, ,\qquad F(Q)=2-Q-Q^{-1},  \quad \mbox{ and} \quad Q=e^{u/2}.
   	$$
Note that for all $s\ge 0$, 
 \bear\label{f.and.F}
   	f(Q)^s=\sum_{l=-s}^s \phi^s_lQ^l\quad \phi^s_l\in \Z,\qquad F(Q)^s=\sum_{l=-s}^s\psi_l^sQ^l\quad  \psi_l^s\in\Z
 \eear
   	are Laurent polynomials in $Q$ with integer coefficients and with leading coefficients $\phi^s_{\pm s}$ and  $\psi^s_{\pm s}$ equal to $\pm1$.  Moreover, 
   	\bear\label{f=Fmod2}
   	f(Q)=F(Q)\mod 2 \quad \mbox{ and therefore } \phi_l^s=\psi_l^s \mod 2 
   	\eear
   	for all $-s\le l\le s$  and $s\ge 0$. 
	\smallskip

On the other hand, with this notation, \eqref{rgw=CY} becomes
\bear\label{Z=sum}
   	Z_g=1+\sum_{d=1}^\infty\sum_{\rho\vdash d\atop  \rho = \rho'}\Big(\epsilon_\rho\prod_{\square\in\rho}f(Q^{h(\square)})\Big)^{g-1}
	q^d
 \eear
 and therefore the coefficient of  $q^d$ is also a Laurent polynomial in $Q$ with integer coefficients.	

Comparing the coefficient of $q^1$ in \eqref{Z=exp.mess} and \eqref{Z=sum} gives 
   	$$
   	[Z_g]_{q^1}= \sum_{h> 0} n^\R_{1,h}(g)f(Q)^{h-1} = f(Q)^{g-1}.
   	$$
As before, $\{f(Q)^n\}_{n\geq 0}$ are linearly independent, therefore $n^\R_{1, h}(g)=1$ for $h=g$ and vanish otherwise, proving (c). In particular, $n^\R_{1, h}(g)\in \Z$. Recall also that 
   	$n^\cx_{1, h}(g)=1$ for $h=g$ and vanish otherwise, and therefore $n^\R_{1, h}(g)= n^\cx_{1, h}(g)$ for all $h$. 
\smallskip

We next proceed by induction on the degree, with initial step for $d=1$ just proved. So we fix a degree $p\ge 2$ and assume by induction that 
\bear\label{n.ind.step} 
n^\R_{d,h}(g) \in \Z  \quad \mbox{ and } \quad n^\R_{d, h}\equiv n^\cx_{d, h}(g)\mod 2, 
\eear
for all $d<p$. We also assume that for all $d<p$, $n^\R_{d,h}(g)=0$ for $h$ large.

 Note that the coefficient of $q^p$ in \eqref{Z=exp.mess} has the form 
   	$$
   	\sum_{h> 0} n^\R_{p,h}(g)f(Q)^{h-1}  \! \! +  \! \! 
	\l[ \exp\! \l(\!\! \sum_{d, h>0 \atop d\ne p } \! \! n^{\R}_{d,h}(g) \! \! 
	\sum_{k\;\rm odd\atop k>0 } \! \!  \tfrac{1}{k}f(Q^k)^{h-1}q^{kd}
   	+ \tfrac 12  \!  \sum_{d, h>0} \! \! n^{\cx}_{d,h}(g) \! 
	\sum_{k=1}^\infty\tfrac{1}{k}F(Q^{2k})^{h-1}q^{2kd}\r)\r]_{q^p}
   	$$
   	where the second term involves only $n^\R_{d,h}(g)$ with $d<p$.
	
Next, using the expansions \eqref{f.and.F}, note that 
   	$$
   	\sum_{k\, \text{odd}}\tfrac{1}{k}f(Q^k)^{h-1}q^{kd}=\sum_{k\, \text{odd}}\tfrac{1}{k}\sum_{l=1-h}^{h-1}\phi^{h-1}_lQ^{kl}q^{kd}= \sum_{l=1-h}^{h-1}\phi^{h-1}_l\log \left(\tfrac{1+Q^lq^d}{1-Q^lq^d}\right)^{1/2}
   	$$
   	and 
   	$$
   	\sum_{k=1}^\infty\tfrac{1}{k}F(Q^{2k})^{h-1}q^{2kd} =\sum_{k=1}^\infty\tfrac{1}{k}\sum_{l=1-h}^{h-1}\psi^{h-1}_lQ^{2kl}q^{2kd}= \sum_{l=1-h}^{h-1}\psi^{h-1}_l\log\tfrac{1}{1-Q^{2l}q^{2d}}. 
   	$$
Combining the last three displayed equations gives
\bear\label{Z=2nd}
   	[Z_g]_{q^p}=\sum_{h> 0} n^\R_{p,h}(g)f(Q)^{h-1} +\left[\prod_{d\neq p}\prod_{h> 0}\prod_{l=1-h}^{h-1}\frac{(1+Q^lq^d)^{\tfrac12(n^\R_{d,h}(g)\phi^{h-1}_l-n^\cx_{d,h}(g)\psi^{h-1}_l)}}{(1-Q^lq^d)^{\tfrac12(n^\R_{d,h}(g)\phi^{h-1}_l+n^\cx_{d,h}(g)\psi^{h-1}_l)}}  \right]_{q^p},
   	\eear
   	where the second summand is a Laurent polynomial in $Q$ with integer coefficients by the induction hypothesis \eqref{n.ind.step} and the fact that $\psi_l^s=\phi_l^s \mod 2$, cf.  \eqref{f=Fmod2}. Since $[Z_g]_{q^p}$ is a Laurent polynomial in $Q$ with integer coefficients by \eqref{Z=sum}, therefore so is 
   	\best
   	\sum_{h> 0} n^\R_{p,h}(g)f(Q)^{h-1}.
   	\eest
   	 Since the coefficients of its expansion and $n_{p,h}^\R(g)$ are related by an integral triangular transformation with 1's along the diagonal this implies $n^\R _{p,h}(g)\in \Z$ for all $h>0$. This also shows the finiteness property of $n^\R_{p,h}(g)$ i.e. that for fixed $g, p$ and large enough $h$ these numbers vanish.
	
\smallskip
   	   	
It remains to show that $n^\R_{p,h}(g)\equiv n^\cx_{p,h}(g)\mod 2$ for all $h$. Similar considerations for the complex GW invariants
\best
 \exp(GW^{conn}(iu))=GW(iu)
\eest
using \eqref{Cx.GV} and \eqref{gw=CY} imply
   	\begin{gather*}
   	\sum_{h> 0} n^\cx_{p,h}(g)F(Q)^{h-1} +  	\left[\prod_{d\neq p}\prod_{h> 0}\;\prod_{l=1-h}^{h-1}\frac{1}{(1-Q^lq^d)^{ n^\cx_{d,h}(g)\psi^{h-1}_l}}  \right]_{q^p}=\sum_{\rho}\left(\prod_{\square\in\rho}F(Q^{h(\square)})\right)^{g-1} q^{|\rho|}.
   	\end{gather*}
Using again \eqref{f=Fmod2}, we see that, mod 2, the Laurent series with integer coefficients 
   	$$
   	1+\sum_{\rho = \rho'}\left(\epsilon(\rho)\prod_{\square\in\rho}f(Q^{h(\square)})\right)^{g-1}q^{|\rho|}\;\; \equiv\; 
   	1+\sum_{\rho}\left(\prod_{\square\in\rho}F(Q^{h(\square)})\right)^{g-1} q^{|\rho|}\quad \text{mod} \,2
   	$$
   	are equal, keeping in mind that  the terms corresponding to $\rho$ and $\rho'$ on the RHS are equal, thus their contribution vanishes mod 2 unless $\rho$ is self-conjugate.   
   	
  The second inductive hypothesis \eqref{n.ind.step} implies that, mod 2, the second summand in \eqref{Z=2nd} equals 
   	$$
   	\left[\prod_{d\neq p}\prod_{h> 0}\;\prod_{l=1-h}^{h-1}\frac{1}{(1-Q^lq^d)^{ n^\cx_{d,h}(g)\psi^{h-1}_l}}  \right]_{q^p} \mod 2.
   	$$
   	  Together these imply that 
   	\best
   	\sum_{h> 0} n^\R_{p,h}(g)f(Q)^{h-1}\;\; \equiv\; \; \sum_{h> 0} n^\cx_{p,h}(g)F(Q)^{h-1}\quad \text{mod} \,2,
   	\eest
   	  which in turn implies  $n^\R_{p,h}(g)\equiv n^\cx_{p,h}(g) \mod 2$, completing the proof of the induction step.
   \end{proof}

\begin{cor} The coefficients $n^{\R}_{d,h}(g)$ vanish unless $d(g-1)+h-1\equiv 0 \mod 2$. In particular, $n^{\cx}_{d,h}(g)$ are even whenever  $d(g-1)+h-1\not\equiv 0 \mod 2$.
\end{cor}
\begin{proof} By Corollary \ref{C.cong}, the connected real invariants $CRGW_{d,h}(g|g-1)$ vanish unless $d(g-1)+h-1\equiv 0 \mod 2$. Therefore the left hand side of \eqref{R.GV} is invariant under the change of variables $(u, q)\to(-u, (-1)^{g-1}q)$. Making this change of variables on the right hand side of \eqref{R.GV} and using the fact that the functions $\{ H_n(u, q^d)\} $ are linearly independent (cf. \eqref{H.n}), implies that $n_{d, h}(g)= (-1)^{h-1+d(g-1)} n_{d, h}(g)$. The result follows.
 \end{proof}

\vskip.2in  
\section{Signed Frobenius-Schur indicator}\label{SFSsection}
 \vskip.1in  
 
In this  final section, which is of independent interest, we introduce the notion of a signed Frobenius-Schur indicator and show that it takes values $0,\pm 1$ on any irreducible real representation of a finite group (unlike the classical FS indicator, which is +1). This was used in \S\ref{Solving} to determine signs the $\ep_\rho$ in the expression of the RGW-invariants.

The classical Frobenius-Schur indicator of a representation of finite group is the character evaluated at the sum of squares of the group elements divided by the order of the group; its possible values for an irreducible representation are 1, 0, and -1, corresponding  to the partition of the irreducible representations into real, complex, and quaternionic representations. 

Below we  construct a signed FS indicator for the symmetric group $S_d$, but these considerations remain valid for any real representation of a finite group $G$ with a sign homomorphism $G\ra \Z_2$. 
\begin{defn} The \textsf{signed Frobenius-Schur indicator} is defined by
\bear\label{FS.ind} 
SFS(\rho)\ma=^{\rm def} \frac{1}{d!}\sum_{g\in S_d} \chi_\rho(g^2)(-1)^{s(g)},
\eear  
where $s(g)$ is the parity of the permutation $g\in S_d$. 
\end{defn}
The sign morphism on $S_d$ descends to conjugacy classes, which are indexed by partitions of $d$, decomposing them into even and odd ones. If $\al$ is a conjugacy class, let $sq(\al)$ denote the conjugacy class of $g^2$, where $g$ is a representative of $\al$, cf. \eqref{sq.la}.
\begin{lemma} On irreducible representations, the signed Frobenius-Schur indicator takes values $0,\pm 1$. Specifically, 
	\best
	SFS(\rho)  = \begin{cases} 0,&\text{if}\, \rho\, \text{ is not self-conjugate},\\
		\pm 1,  &\text{if}\, \rho\, \text{ is   self-conjugate}.
		\end{cases}
	\eest
	Furthermore, when $\rho$ is self-conjugate, $SFS(\rho)$ is given by
	\bear\label{def.de.rho}
	\ep_\rho=(-1)^{o(\rho)}, \quad \text{with}\quad o(\rho) = \frac{1}{d!}\sum_{g\in S_d \atop g\; {\rm odd} } \chi_\rho(g^2)
	 =\sum_{\alpha \vdash d \atop \alpha \; {\rm odd}} \frac{\chi_\rho(sq(\alpha))}{\zeta(\alpha)}.
	\eear
	The expression $o(\rho)$ takes values $0$ or $1$ on a self-conjugate representation $\rho$.
\end{lemma}
\begin{proof} The proof is similar to that for the standard Frobenius-Schur indicator; see for example \cite[\S3.2.3]{D} (or \cite[\S5.1]{E}). 	By \cite[Lemma 3.2.17]{D}, 
	$$
	\chi_\rho(g^2)=\chi_{Sym^2 \rho}(g)-\chi_{Alt^2\rho}(g).
	$$
The space $B$ of bilinear forms on $V$ splits as a direct sum of symmetric and alternating forms 
	$$
	B=V^*\otimes V^*=Sym^2\oplus Alt^2  \quad \mbox { and }\quad  \Hom(V,V^*)=B.
	$$
	Given a representation $\rho$, consider the following $S_d$-action on B:
	$$
	g\cdot b(v,w)=b(\rho(g^{-1})v, (-1)^{s(g)}\rho(g^{-1})w) 
	$$
	and on $V^*\otimes V^*$: 
	$$
	g\cdot \lambda(v)\otimes \mu(w) = \lambda(\rho(g^{-1})v)\otimes \mu((-1)^{s(g)}\rho(g^{-1})w). 
	$$ 
	Then the isomorphism
	$$
	V^*\otimes V^* \longrightarrow B\, ,\quad \lambda\otimes\mu \mapsto b=\lambda\otimes\mu
	$$
	is $S_d$-equivariant. The induced action on $\Hom(V,V^*)$ gives rise to an isomorphism
	$$
	\Hom(V,V^*)^{S_d} = B^{S_d}.
	$$
	Thus, if $\rho$ is irreducible,
	$$
	\dim B^{S_d}\leq 1.
	$$
	Now, let $T^{2'}=T^{2'}(\rho)$ be the conjugate of the $2$-tensor representation of $\rho$
	$$
	T^{2'}(\rho): S_d\longrightarrow \End(B)\, ,\quad T^{2'}(\rho)(g)=(-1)^{s(g)}\rho(g)\otimes\rho(g)
	$$
	and
	$$
	\pi= \frac{1}{d!}\sum_{g\in S_d} g.
	$$
	For every $h\in S_d$, we have $h. \pi =\pi$ and thus, for every $b\in B$,
	$$
	T^{2'}(\pi)(b)\in B^{S_d}
	$$
	and for $b\in B^{S_d}$, $T^{2'}(b)=b$. Thus taking trace we obtain
	$$
	\frac{1}{d!}\sum_{g\in S_d} \chi_{T^{2'}}(g) = \dim B^{S_d}\leq 1.
	$$
	Furthermore,
	$$
	\frac{1}{d!}\sum_{g\in S_d} \chi_{T^{2'}}(g) = \frac{1}{d!}\sum_{g\in S_d} \chi^2_{\rho}(g)(-1)^{s(g)}.
	$$
By the   display above it,	the last expression is equal to $0$ or 1. Since $ \frac{1}{d!}\sum_{g\in S_d} \chi^2_{\rho}(g)=1$ and $\chi_\rho(g)\in\R$, the signed expression is  1 iff $\chi_\rho(g)=0$ for odd parity $g$, i.e. iff $\rho$ is self-conjugate.\\
	
	\noindent
	Similarly, we have induced actions on $Sym^2$ and $Alt^2$ and
	$$
	B^{S_d}=(Sym^2)^{S_d}\oplus (Alt^2)^{S_d}.
	$$
	Let $Sym^{2'}$ and $Alt^{2'}$ be the conjugates of the representations $Sym^2(\rho)$ and $Alt^2(\rho)$. As before 
	$$
	\frac{1}{d!}\sum_{g\in S_d} \chi_{Sym^{2'}}(g) = \dim (Sym^2)^{S_d} \quad \mbox{ and} \quad 
	\frac{1}{d!}\sum_{g\in S_d} \chi_{Alt^{2'}}(g) = \dim (Alt^2)^{S_d}.
	$$ 
	Since $\dim B^{S_d}\leq 1$,  the possible pairs of dimensions $(\dim(Sym^2)^{S_d}, \dim(Alt^2)^{S_d})$ are $(0,0)$, $(1,0)$, and $(0,1)$, with the latter two cases appearing only for self-conjugate representations.
	Finally, we have
	\best
	SFS(\rho)&=&\frac{1}{d!}\sum_{g\in S_d} \chi_\rho(g^2)(-1)^{s(g)}= \frac{1}{d!}\sum_{g\in S_d}  (-1)^{s(g)}\chi_{Sym^2 \rho}(g)-(-1)^{s(g)}\chi_{Alt^2\rho}(g)
	\\ 
	&=&\frac{1}{d!}\sum_{g\in S_d}  \chi_{Sym^{2'} }(g)-\chi_{Alt^{2'}}(g)
	= \dim(Sym^2)^{S_d}-\dim(Alt^2)^{S_d}.
	\eest
	By the previous paragraph, this expression thus equals $\pm 1$ iff $\rho$ is self-conjugate and is 0 otherwise. 
	
\smallskip
It remains to determine when is  the SFS indicator +1 and when -1. The standard FS indicator for $S_d$ is always +1 therefore
\bear\label{1=odd}
	1=\frac{1}{d!}\sum_{g\in S_d} \chi_{Sym^2\rho}(g)=\frac{1}{2d!}\sum_{g\in S_d} \chi^2_\rho(g)+\chi_\rho(g^2)
\\\label{2=even}
	0=\frac{1}{d!}\sum_{g\in S_d} \chi_{Alt^2\rho}(g)=\frac{1}{2d!}\sum_{g\in S_d} \chi^2_\rho(g)-\chi_\rho(g^2).
\eear
Fix a self-conjugate representation $\rho$ and let
	$$
	e=\frac{1}{d!}\sum_{g\in S_d  \atop g\; \text{even} }\chi_\rho(g^2)\, \qquad \mbox{ and } \quad 
	o=\frac{1}{d!}\sum_{g\in S_d \atop g\; \text{odd}} \chi_\rho(g^2).
	$$
 Subtracting the two equalities \eqref{2=even} and \eqref{1=odd} gives $e+o=1$.  Since $SFS(\rho)=e-o=\pm 1$ for the self conjugate representation $\rho$, then $2o=1-SFS(\rho)$ is either 0 or 2; in either case $SFS(\rho)=(-1)^{o}$ completing the proof.  
\end{proof}

The following lemma played a key role in \S \ref{Solving}. 

\begin{lemma}\label{SFS} Let $\al$ be a partition of $d$. With the notations above
\be
	\sum_{\beta \vdash d\atop sq(\beta)=\alpha} (-1)^{s(\beta)} \frac{\zeta(\alpha)}{\zeta(\beta)}=\sum_\rho SFS(\rho)\chi_\rho(\al)=
	\sum_{\rho=\rho'} \ep_\rho \chi_\rho(\alpha)
\ee
	where $\ep_\rho$ is given by \eqref{def.de.rho}, and $s(\beta)$ denotes the parity of $\beta$. 
	\end{lemma} 

\begin{proof} The partition $\al$ corresponds to conjugacy class of $S_d$ and let $a \in S_d$ denote a representative of it. Then 
	$$
	\sum_{\beta\vdash d\atop sq(\beta)=\alpha} \frac{(-1)^{s(\beta)}\zeta(\alpha)}{\zeta(\beta)}=
	\sum_{g\in S_d \atop g^2=a} (-1)^{s(g)}.
	$$
	Consider the class function 
	$$
	\theta(a)=\sum_{g\in S_d \atop g^2=a} (-1)^{s(g)}.
	$$
	It has an expansion in the basis of irreducible characters $\{\chi_\rho\}$
	$$
	\theta = \sum_\rho \langle \theta, \chi_\rho\rangle \chi_\rho,
	$$
	where $\langle \, , \rangle$ is the inner product on the class functions. We have
	$$
	\langle \theta, \chi_\rho\rangle = \frac{1}{d!}\sum_{h\in S_d} \theta(h)\chi_\rho(h)=
	\frac{1}{d!}\sum_{h\in S_d} \sum_{g\in S_d \atop g^2=h}(-1)^{s(g)}\chi_\rho(h)=\frac{1}{d!}\sum_{g\in S_d} (-1)^{s(g)}\chi_\rho(g^2).
	$$
	Therefore 
	$$
	\sum_{g\in S_d \atop g^2=a} (-1)^{s(g)}=\theta(a)=
	\sum_\rho\frac{1}{d!}\sum_{g\in S_d} (-1)^{s(g)}\chi_\rho(g^2)\chi_\rho(a) 
	=\sum_\rho\chi_\rho(\alpha)\frac{1}{d!}\sum_{g\in S_d} (-1)^{s(g)} \chi_\rho(g^2).
	$$	
\end{proof}

 \vskip.2in 


\setcounter{equation}{0}
\setcounter{section}{0}
\setcounter{theorem}{0}

\renewcommand{\theequation}{{A}.\arabic{equation}}

\appendix
\section{Real Orientation and Twisted Orientation Data}\label{sectionA}
 \vskip.1in 


In this appendix we review the key ideas of \cite{gz} that provided a criterion for orienting the real moduli space. We then describe a small variant that covers additional cases. 

 Let $(X,\omega, \phi)$ be a Real symplectic manifold and consider the real moduli space  $\ov \M_{B,g,\ell}^{\phi}(X)$ of pseudo-holomorphic real maps $f:(C, \si) \ra (X, \phi)$ from a genus $g$ surface with $\ell$ pairs of conjugate marked points and representing the class $B\in H_2(X)$. 
 
We will use the setup of \cite[\S4.3]{gz}. For every real map $f:(C, \si)\ra (X, \phi)$ and Real vector bundle $(V,\Phi)\lra (X,\phi)$ with a $\Phi$-compatible connection $\nabla$, let 
 \best
 D_{(V, \Phi), f}: \Gamma( f^* V)^{\Phi} \ra \La^{0,1}(f^*V)^{\Phi}
 \eest
 be the restriction of the real Cauchy-Riemann operator $D_{V,f}=\dbar^{f^*\nabla}$ to the space of invariant sections. Denote by  
 $\det D_{(V, \Phi)}$ the determinant line bundle of the family of operators over the real moduli space; see  \cite[\S4.3]{gz}  for more details.

 Unlike the classical line bundle $\det D_{V}$, which is always canonically oriented (cf. proof of \cite[Theorem 3.1.5(i)]{MS}), the line bundle  $\det D_{(V, \Phi)}$ is not always orientable. 

The considerations of  \cite{gz} imply that, after stabilization of the domain if necessary, the 
 orientation sheaf of the real moduli space is canonically identified with 
\bear\label{or.moduli.abs.a} 
\det  T\ov  \M^{\phi}_{B, g, \ell}  (X) =\det D_{(TX, d\phi)} \otimes  \mathfrak{f}^*\det T \ov{ \R\M}_{g, \ell} 
\eear
 cf. \cite[(3.3)]{gz}. Here  $\mathfrak{f}$ is the map to the real Deligne-Mumford moduli space parametrizing real curves of Euler characteristic $\chi$ and $2\ell$ pairs of conjugate marked point. 
\smallskip

 A notion of \textsf{real orientation} on a Real bundle  $(V, \Phi)\ra (X, \phi)$ was introduced in \cite{gz}. For the tangent bundle $(TX,d\phi)$, a real orientation consists of a Real line bundle $(L,\wt\phi)\rightarrow (X,\phi)$ such that 
\bear\label{L2=TX}
\psi: (L,\wt\phi)^{\otimes 2}\cong \Lambda^{\text{top}}(TX,d\phi),
\eear
along with a choice of a homotopy class of such isomorphism and a spin structure on $TX^\phi\oplus 2(L^*)^{\wt\phi^*}$; see \cite[Definition 1.2]{gz}. When the complex dimension of $X$ is odd, such structure induces an orientation on all real moduli spaces $\ov \M_{B,g,l}^{\phi}(X)$, cf. \cite[Theorem 1.3]{gz}. The main ingredients in the proof are as follows.

One of the key results of \cite{gz} is Proposition 5.2 which states that a real orientation on a rank $n$ Real bundle $(V,\Phi)\ra(\Si,\si)$ determines a canonical class of isomorphisms 
\bear\label{isom.psi.2L}
\Psi: (V\oplus 2L^* , \Phi\oplus 2\wt{\phi}^*)\cong (\Si\times \cx^{n+2}, \si\times c_{std}),
\eear
Here $(L^*, \wt{\phi}^*)$ denotes the dual of the Real bundle $(L, \wt{\phi})$. In turn, \eqref{isom.psi.2L} induces a canonical isomorphism 
\bear\label{detV=cx.n}
\det D_{(V,\Phi)}\cong (\det\dbar_{(\cx, c_{std})})^{\otimes n},
\eear
using the fact that 
\best
\det D_{(2L^*,2\wt{\phi}^*)}=(\det (D_{(L^*,\wt{\phi}^*)})^{\otimes 2}
\eest 
is canonically oriented as twice a bundle. In \cite[\S5 and \S 6]{gz} family versions of this result are proved for families of (possibly nodal) surfaces.  In particular, if $f:(C,\si)\ra (X,\phi)$ is a point in the real moduli space, by \cite[Proposition 5.2]{gz} the real orientation on the target determines by pullback a homotopy class of isomorphisms 
\bear\label{oriso.a.2L}
f^*(TX\oplus 2L^* , d\phi\oplus 2\wt{\phi}^*)\cong (\Si\ti \cx^{n+2}, \si\ti c_{std})
\eear
which varies continuously with $f$. By the proof of \cite[Theorem 1.3]{gz} this induces canonical isomorphisms 
\bear\label{det=cx.n}
\det D_{(TX,d\phi)}\cong \det D_{(TX,d\phi)}\otimes (\det  D_{(L^*,\wt\phi^*)})^{\otimes 2}
\cong (\det\dbar_{(\cx,  c_{std})})^{\otimes (n+2)},
\eear
over the real moduli spaces $\ov \M_{B,g,l}^{\phi}(X)$. Here the first isomorphism is induced by the canonical orientation on twice 
a bundle, while the second one is induced by the isomorphism \eqref{oriso.a.2L}.  

Moreover by \cite[Theorem 1.3]{gz}, there is also a canonical isomorphism
\bear\label{or.DM.a} 
\det (T\ov{\R\mathcal{ M}}_{h, \ell})= \det \dbar_{(\cx, c_{std})},
\eear
where the forgetful morphism of a pair of marked points is oriented via the first elements in the pairs. Therefore \eqref{or.moduli.abs.a}, \eqref{det=cx.n} and \eqref{or.DM.a} combine to give a canonical isomorphism 
 \bear\label{det=real.gz}
\det  \ov  \M^{\phi}_{d, g, l}  (X) \cong (\det\dbar_{(\cx, c_{std})})^{\otimes (n+1)},
\eear
cf.  \cite[Theorem 1.3]{gz}, and therefore an orientation on all the real moduli spaces $\ov \M_{B,g,\ell}^{\phi}(X)$ when the complex dimension $n$ of $X$ is odd. 
\medskip

The   canonical isomorphism \eqref{isom.psi.2L} constructed in \cite[Proposition 5.2]{gz} only requires
\begin{enumerate}[(a)]
\item a homotopy class of isomorphisms 
$\Lambda^{\text{top}}(V\oplus 2L^*, \phi\oplus 2\wt\phi^*)\cong (\Si\times\cx, \sigma\times c_{std})$ and
\item a spin structure on $V^\phi\oplus 2(L^*)^{\wt\phi^*}$ 
\end{enumerate} 
and it does not depend on whether or not $E=2L^*$ is twice of a bundle, cf. \cite[Corollary 5.5]{gz}. The fact that $E$ is twice of a bundle is only used in the proof of \cite[Theorem 1.3]{gz} to argue that $\det\dbar_{(2L^*, 2\wt\phi^*)}$ is canonically oriented; cf. \eqref{det=cx.n}. Thus $(2L^*, 2\wt \phi^*)$ can be replaced by any real bundle pair $(E, \tilde \phi)$ for which we know that the determinant line is canonically oriented. 
\medskip

Such a choice $(E, \tilde \phi)$ can also be obtained as follows. Let $L\ra X$ be a complex line bundle and let $E=L\oplus \phi^* \ov L$ with the real structure $\tilde \phi=\phi_{tw}$ defined as in \eqref{L.plus.twL}. Then the projection onto the first factor induces a canonical isomorphism 
\bear\label{ind.gen.E}
\det D_{(E, \tilde  \phi)}= \det D_{(L\oplus \phi^*\ov L, \phi_{tw})}\ma\cong^{ \pi_1}\det D_{L} 
\eear
over the real moduli space, as in \eqref{ind.E}. The right hand side is the determinant line of a real Cauchy-Riemann operator and is thus canonically oriented (cf. proof of \cite[Theorem 3.1.5(i)]{MS}).  Therefore it induces a canonical orientation on the left-hand side. 

This motivates the following variant of \cite[Definition 1.2]{gz}. 
\begin{defn}\label{TRO.gen}  Assume $(X, \phi)$ is a Real symplectic manifold. A \textsf{twisted orientation}   
$\fo=(L, \psi,\mathfrak{s})$ for it consists of 
\begin{enumerate}[(i)]
	\item a complex line bundle $L\ra X$ such that the bundle pair $(E, \tilde  \phi)=(L\oplus \phi^* \ov L, \phi_{tw})$ satisfies: 
\bear\label{cond.tw.or}
w_2(TX^\phi)= w_2(E^{\tilde \phi}) \quad \mbox{ and } \quad 
\La^{\mathrm{top}} (TX^\phi, d\phi) \cong  \La^{\mathrm{top}} (E, \tilde \phi)
\eear	
	\item  a homotopy class $[\psi]$ of isomorphisms satisfying \eqref{cond.tw.or}. 
	\item  a spin structure $\mathfrak{s}$ on the real vector bundle $TX^\phi\oplus (E^*)^{\tilde \phi^*}$ over the real locus, compatible with the orientation induced by $\psi$.
\end{enumerate}	
\end{defn}

Then \cite[Theorem 1.3]{gz} extends to  give: 
\begin{lemma} A twisted orientation on $(X,\phi)$ induces a canonical orientation on the real moduli spaces $\ov \M_{B,g,\ell}^{\phi}(X)$ when the target $X$ is odd complex dimensional.
\end{lemma} 
\begin{proof} As in \cite[Proposition 5.2]{gz}, a twisted orientation determines by pullback a canonical homotopy class of isomorphisms
\bear\label{oriso.a.tw}
f^*(TX\oplus E^*, d\phi\oplus \wt{\phi}^*)\cong (\Si\times \cx, \si\ti c_{std})^{\oplus (n+2)}
\eear
varying continuously with $f\in\ov \M_{B,g,\ell}^{\phi}(X)$.
The rest of the proof is the same as that of \cite[Theorem 1.3]{gz} except now   \eqref{det=cx.n} is replaced by  
\bear\label{det=tw.cx.n}
\det D_{(TX,d\phi)}\cong \det D_{(TX,d\phi)}\otimes \det D_{(E^*, \tilde  \phi^*)}
\cong (\det\dbar_{(\cx, c_{std})})^{\otimes (n+2)}
\eear
over the real moduli space $\ov \M_{B,g,\ell}^{\phi}(X)$. Here the first isomorphism is induced by \eqref{ind.gen.E}  and the complex orientation on $\det \dbar_{L^*}$ and the second isomorphism is induced by \eqref{oriso.a.tw}. Combined with \eqref{or.DM.a} and \eqref{or.moduli.abs.a} this determines a canonical homotopy class of isomorphisms \eqref{det=real.gz} as in \cite[Theorem 1.3]{gz}. 
\end{proof}
\medskip

These considerations similarly extend to the relative real moduli spaces considered in this paper, with the following modification. Assume the target $(\Si, c)$ is a symmetric Riemann surface with $r$ pairs of conjugate points, and consider the relative real moduli space  
$$\ov\M= \ov  \M^{\bullet, c}_{d, \chi} (\Si)_{\la^1, \dots, \la^r}$$ of Definition~\ref{D.rel,real}. The deformation-obstruction theory (with fixed domain) at a point $f\in \ov \M$ is determined by the linearization 
$\del_{f^*(T\Si, dc)}$ where $T\Si$ is the relative tangent bundle \eqref{T.punctured} to the {\em marked} curve 
$\Si=(S, j, \{x_i^\pm\}_{i=1,  \dots,r})$.  This is analogous with the situation for the complex moduli space and can be seen as follows. A point $f$ in the moduli space is a real map $f:(C, \si)\ra (\Si, c)$ which is ramified of order $\la_j^i$ at the points $y_{ij}^\pm$, where 
$f^{-1}(x_i^\pm)=\{ y_{ij}^\pm\}_{j=1, \dots, \ell(\la^i)}$, and $i=1, \dots r$. Variations in $f$ with fixed domain must vanish to order 
$\la_j^i$ at $y_{ij}^\pm$, i.e. correspond to sections of 
\best
(f^*T S)\otimes \O\Big( -\ma{\textstyle\sum}_{i, j} \la_j^i y_{ij}^+ - \ma{\textstyle\sum}_{i,j} \la_j^i y_{ij}^-\Big) = 
f^*\Big(T S \otimes \O\Big( -\ma{\textstyle\sum}_ix_i^+ -\ma{\textstyle\sum}_i x_i^-\Big)\Big)= f^* T\Si
\eest
which are invariant under the involutions on the domain and target. 
 Therefore the orientation sheaf of the relative real moduli space is canonically identified with 
\bear\label{or.moduli.rel.a} 
\det  T \ov  \M^{c, \bullet}_{d, \chi}  (\Si)_{\vec \la} =\det \del_{(T\Si, dc)} \otimes  \mathfrak{f}^*\det T \ov{ \R\M}^\bullet_{\chi, \ell},
\eear
where $\ell= \sum_i\ell(\la_i)$ is the total number of pairs of marked points on the domain, cf. \eqref{or.moduli.abs.a}.  A twisted orientation  on the {\em marked curve $(\Si,c)$} determines an orientation on these moduli spaces via the same procedure as in the absolute case. (Note that Definition \ref{TRO.gen} and Definition \ref{TRO} are equivalent, with $L^*=\Theta$.)

When $(\Si, c)$ is a connected Riemann surface a real orientation in the sense of \cite{gz} exists except in the case when $c$ is fixed-point free and $g(\Si)$ is  even.   A twisted orientation exists on any Real curve.  
\bigskip

 We end this appendix with the following observation. Assume $(X,\phi)$ is a Real symplectic manifold. For any Real line   bundle $(L, \wt \phi)\ra (X, \phi)$ there is an isomorphism 
\bear\label{Aisom.tw=double}
\theta: (L\oplus \phi^*\ov L, \phi_{tw}) \ra (L\oplus L, \wt \phi\oplus \wt \phi),
\eear
as in \eqref{isom.tw=double}; it induces an isomorphism
\best
\theta^\R: L_{|X^\phi} \cong 2 L^{\wt \phi}
\eest
along the fixed locus $X^\phi$. In particular, if a real orientation $(L, \wt \phi)$ exists, we can use either the real orientation or the twisted orientation $(L\oplus \phi^*\ov L, \phi_{tw}) $ to induce an orientation on the moduli space. The difference between the two orientation procedures is determined by Lemma~\ref{L.tw=2bd} as follows. 

\begin{lemma}\label{A.comp} Let  $\fo=((L, \wt\phi), \psi, \mathfrak{s})$  be a real orientation for   $(X,\phi)$.
Let $\fo'=(L\oplus c^*\ov L,\psi', \mathfrak{s}')$ denote the associated twisted orientation obtained from $\fo$ via the isomorphism \eqref{Aisom.tw=double}, where $ \mathfrak{s}'=(id\oplus \theta^\R)^*\mathfrak{s}$ and $\psi'= \psi\circ \theta$.
The difference between the orientation on the real moduli spaces $\ov\M^\phi(X)$ induced by $\fo$ and that induced by $\fo'$ is $(-1)^\iota $, where $\iota$ is the complex rank of $\Ind\dbar_{L^*}.$
\end{lemma} 

\begin{proof}
	As explained at the beginning of the appendix, the two orienting procedures differ only in the way the auxiliary index bundle is oriented i.e. it is the difference in  how the index bundles of  $L^*\oplus \phi^*\ov L^*$ and $2L^*$ are oriented, in the first case via the projection onto the first factor and in the second as twice a bundle. This difference is given by Lemma \ref{L.tw=2bd}.
	\end{proof}
 

\end{document}